\documentclass[11pt]{amsart}
\usepackage{amssymb}
\usepackage{enumerate}
\usepackage{amsmath,stmaryrd,amssymb,amsfonts,amsthm,
	latexsym, amscd, epsfig, epic,color,bbm}
\usepackage{mathtools, graphics, epstopdf} 
\usepackage{extarrows}
\usepackage{enumitem}
\usepackage{enumerate}
\usepackage{mathrsfs}
\usepackage{eucal}%    caligraphic-euler fonts: \mathcal{ }
\usepackage{eufrak}%   frak-euler        fonts: \mathfrak{ }
\usepackage{xspace}
\usepackage{a4wide}
\usepackage{colordvi}
\usepackage{comment}
\usepackage{hyperref,eucal}
\usepackage{tikz}\usetikzlibrary{matrix,arrows}
\usepackage{imakeidx}
\makeindex
\newtheorem{Theorem}{Theorem}[section]
\newtheorem{theorem}[Theorem]{Theorem}
\newtheorem{corollary}[Theorem]{Corollary}
\newtheorem{lemma}[Theorem]{Lemma}
\newtheorem{definition}[Theorem]{Definition}
\newtheorem{conjecture}[Theorem]{Conjecture}
\newtheorem{proposition}[Theorem]{Proposition}

\theoremstyle{plain}

\theoremstyle{example}
\newtheorem{example}{Example}[section]

\theoremstyle{openproblem}

\def\W{\CMcal{W}}

\def\O{\CMcal{O}}

\def\E{\CMcal{E}}
\def\F{\CMcal{F}}
\def\T{\CMcal{T}}

\def\R{\CMcal{R}}

\def\A{\mathcal{A}}

\newcommand{\bh}{\mathbf{h}}

\def\emma{\color{black}}

\theoremstyle{remark}
\newtheorem{remark}{Remark}
\numberwithin{equation}{section}
\numberwithin{figure}{section}
\title[Towards equivalent thickened ribbon Schur functions]{Towards equivalent thickened ribbon Schur functions}

\author{Emma Yu Jin}
\address{School of Mathematical Sciences, Xiamen University, Xiamen, Fujian 361005, P.R. China}
\email{yjin@xmu.edu.cn}

\author{Shu Xiao Li}
\address{
		Research Center for Mathematics and Interdisciplinary Sciences,
		Shandong University \& Frontiers Science Center for Nonlinear Expectations, Ministry of Education,
		Qingdao, Shandong 266237, P.R. China}
\email{lishuxiao@sdu.edu.cn}



\begin{document}
	\begin{abstract}
		
		Two skew diagrams are defined to be equivalent if their corresponding skew Schur functions are equal. The equivalence classes of ribbons (edgewise connected skew diagrams without a $2\times 2$ block of boxes) have been classified by Billera, Thomas and van Willigenburg in 2006. 
		
		In this paper, we provide a complete characterization of the equivalence classes of connected skew diagrams with exactly one inclusion-maximal $2\times m$ or $m\times 2$ block of boxes for every $m\ge 2$. In particular, the possible sizes of such equivalence classes are one, two, or four, demonstrating that a single $2\times m$ or $m\times 2$ block dramatically reduces the sizes of equivalence classes. Our result confirms special cases of the elusive conjecture on equivalent connected skew diagrams proposed by McNamara and van Willigenburg in 2009.
		
	\end{abstract}
	\maketitle
	\section{Introduction and main results}
	Littlewood-Richardson coefficients (LR coefficients) appear in a wide variety of areas in mathematics including representation theory, combinatorics and algebraic geometry, etc; see, for example, \cite{FH:91,Fulton:97,Sagan:00,Stanley,Stembridge:02}. From a combinatorial viewpoint, the LR coefficients are defined as the structure constants $c_{\mu\nu}^{\lambda}$ appearing in the Schur expansion of the product of two Schur functions:
	\begin{align}\label{E:LReqn}
		s_{\mu}s_{\nu}&=\sum_{\lambda}c_{\mu\nu}^{\lambda}s_{\lambda}.
	\end{align}
	Several combinatorial interpretations of these coefficients have been discovered, including the Young tableaux model \cite{Fulton:97}, the honeycomb model \cite{KnuTao:99} and the puzzle model \cite{KTW:04}, {\emma all of which prove} that the LR coefficients are always nonnegative. {\emma Additionally}, the LR coefficients {\emma arise} in the decomposition of tensor products of irreducible $GL_n$-modules, and count intersection points in the Schubert calculus on a Grassmannian.

	%From the representation theory viewpoint, 
	
	%\begin{align*}
	%S_{\mu}V\otimes_{\mathbb{C}} S_{\nu}V=\bigoplus_{\lambda}(S_{\lambda}V)^{\bigoplus c_{\mu\nu}^{\lambda}},
	%\end{align*}
	%meaning that $c_{\mu\nu}^{\lambda}$ equals the multiplicity of the irreducible representation $S_{\lambda}V$ in the tensor product $S_{\mu}V\otimes_{\mathbb{C}} S_{\nu}V$. From the algebraic geometry viewpoint,  the LR coefficients count intersection points in the Schubert calculus on a Grassmannian:
	%\begin{align*}
	%[X_{\mu}]\smile [X_{\nu}]=\sum_{\lambda\subseteq r\times (n-r)}c_{\mu\nu}^{\lambda}[X_{\lambda}].	
	%\end{align*}
	
	Motivated by these rich connections, it becomes important to investigate properties of these nonnegative coefficients such as
	%such as the conditions on the partitions $\lambda,\mu,\nu$ for which $c_{\mu\nu}^{\lambda}$ is positive \cite{Azenhas,Kly:98,KnuTao:99,TG:22} and nice identities involving the LR coefficients (see for instance \cite{HW:14,Steinberg:61}), etc. 
	%Narayanan \cite{Narayanan:06} found that computing LR coefficients is $NP$-complete, implying that an efficient algorithm to compute them is unlikely to exist unless $P=NP$. An alternative is to explore 
	equality among the LR coefficients. A natural and promising {\emma approach} to this question is to {\emma study} when two skew Schur functions are equal \cite{BTvW,G:09,MvW,OH:24,RSvW,Y:17}. {\emma More precisely}, $s_{\lambda/\mu}=s_{\rho/\sigma}$ if and only if $c_{\mu\nu}^{\lambda}=c_{\sigma\nu}^{\rho}$ for any $\nu$.
	This {\emma follows immediately from} an equivalent form of (\ref{E:LReqn}), i.e.,
	\begin{align*}
		s_{\lambda/\mu}&=\sum_{\nu}c_{\mu\nu}^{\lambda}s_{\nu},
	\end{align*}
	and the fact that the set of Schur functions $s_{\nu}$ where $\nu$ {\emma ranges over all partitions} of $n$ forms a basis of the space $\Lambda^n$ of homogeneous symmetric functions of degree $n$. 
	
	Another consequence of {\emma the equality of} skew Schur functions is an equality of the {\emma corresponding} Kostka numbers. Expanding $s_{\lambda/\mu}$ in {\emma the} monomial symmetric basis $m_{\nu}$ gives
	\begin{align*}
		s_{\lambda/\mu}&=\sum_{\nu}K_{\lambda/\mu}^{\nu}m_{\nu},
	\end{align*}
	where $K_{\lambda/\mu}^{\nu}$ is the Kostka number counting semistandard Young tableaux of shape $\lambda/\mu$ and with {\emma weight} $\nu$ (see \cite{Stanley}). Since the monomial symmetric functions $m_{\nu}$ {\emma form} a basis of $\Lambda^n$, {\emma it follows immediately} that $s_{\lambda/\mu}=s_{\rho/\sigma}$ if and only if $K_{\lambda/\mu}^{\nu}=K_{\rho/\sigma}^{\nu}$ for all $\nu$.
	
	For convenience, we use the following definition.
	\begin{definition}[\cite{BTvW,MvW,RSvW}]\label{Def:equi1}
		Two skew diagrams $\CMcal{F}$ and $\E$ are called {\em equivalent}, denoted $\CMcal{F}\sim \E$, if their corresponding skew Schur functions are equal, i.e., $s_{\CMcal{F}}=s_{\E}$. The equivalence class of $\CMcal{F}$ is denoted by $[\CMcal{F}]$.
	\end{definition}
	The study of equivalent skew diagrams was initiated by Billera, Thomas, and van Willigenburg \cite{BTvW} {\emma who completely determined the} equivalence classes of ribbons. A {\em ribbon} is an edgewise connected skew diagram {\emma containing no} $2\times 2$ block of boxes. 
	{\emma In brief}, two ribbons are equivalent if and only if both ribbons admit parallel factorizations that differ {\emma only} by reversing some of the factors (see Theorem 4.1 of \cite{BTvW}). Subsequently, Reiner, Shaw and van Willigenburg \cite{RSvW} introduced an operation between a ribbon and a skew diagram in order to construct larger equivalent non-ribbon skew diagrams from smaller diagrams. This {\emma construction} was {\emma later} extended to {\emma arbitrary} skew diagrams by McNamara and van Willigenburg \cite{MvW}, {\emma who also} established sufficient conditions for {\emma the equivalence of} two skew diagrams. Remarkably, these sufficient conditions are conjectured to be almost necessary, in the sense that only one hypothesis of {\emma the} sufficient conditions is absent in the conjectured necessary conditions. 
	
	An abbreviated form of this conjecture is as follows. Given a skew diagram $\CMcal{F}$, let $\F^*$ be the antipodal rotation of $\F$ obtained by rotating $\F$ by 180 degrees in the plane.   
	\begin{conjecture}(\cite{MvW}, Conjecture 5.7)\label{Conj:1}
		Two skew diagrams $\F$ and $\E$ satisfy $\F\sim \E$ if and only if, for some $r$,
		\begin{align*}
			\F&=(\cdots((\E_1\circ_{W_2}\E_2)\circ_{W_3}\E_3)\cdots)\circ_{W_r}\E_r \quad \mbox{ and }\\
			\E&=(\cdots((\E'_1\circ_{W'_2}\E'_2)\circ_{W'_3}\E'_3)\cdots)\circ_{W'_r}\E'_r,\,\, \mbox{ where }
		\end{align*}
		\begin{itemize}
			\item $\E_1,\E_2,\ldots,\E_r$ are skew diagrams;
			\item for $2\le i\le r$, $\E_i=W_iO_iW_i$ satisfies Hypotheses I--IV;
			\item $\E_1'=\E_1$ or $\E_1'=\E_1^*$;
			\item for $2\le i\le r$, $(\E_i',W_i')=(\E_i,W_i)$ or $(\E_i',W_i')=(\E_i^*,W_i^*)$.
		\end{itemize}
		The equivalence class of $\F$ contains $2^{\kappa}$ elements, where $\kappa$ counts {\emma the} factors $\E_i$ in any irreducible factorization of $\F$ such that $\E_i\ne \E_i^*$.
	\end{conjecture}
	{\emma Since the full statement} of Hypotheses I--IV {\emma is quite lengthy}, we do not {\emma reproduce them} here and refer the {\emma reader} to \cite{MvW} for more details. 
	
	Recently, Olson-Harris \cite[Theorem 5.3.9]{OH:24} proved the sufficiency of Conjecture \ref{Conj:1} using a Hopf-algebraic approach. However, {\emma establishing its} necessity {\emma appears to be out of reach}; even the proof for ribbons is {\emma highly nontrivial} (see Sections $4$ and $5$ of \cite{BTvW}). We quote here the {\emma comments of} Olson-Harris \cite[Section 5.3.2]{OH:24} {\emma regarding} the difficulty of Conjecture \ref{Conj:1}:
	
	``{\em Showing that such factorizations are necessary seems extremely difficult, but even showing that they are sufficient in all cases has remained open since the publication \cite{MvW}.}''
	
	Related equality questions have also been studied for multiplicity-free skew characters by Gutschwager \cite{G:09} in 2009, providing further evidence that strong restrictions on the Schur expansion often {\emma impose} rigid constraints on the underlying skew diagrams.
	
	The purpose of this paper is to characterize the equivalence classes of (edgewise) connected skew diagrams containing exactly one $2\times m$ or $m\times 2$ block of boxes, as a first step towards {\emma establishing} the necessity of Conjecture \ref{Conj:1}. %The {\em thickened ribbons} are defined as connected skew diagrams 
	%without $3\times 3$ block of boxes, signifying that thickened ribbons differ from ribbons only by some $2\times m$ or $m\times 2$ blocks of boxes for any $m\ge 2$. 
	
	Our main result is stated as follows: 
	\begin{theorem}\label{T:1}
		Two connected skew diagrams $\F$ and $\E$ {\emma each} with exactly one $2\times m$ or $m\times 2$ block for a fixed $m\ge 2$, are equivalent, %satisfy $\F\sim \E$ 
		if and only if %$\F=\CMcal{S}\circ_{\W}\T$ and $ \E=\CMcal{S}'\circ_{\W}\T'$
		\begin{align*}
			\F=\CMcal{S}\circ_{\W}\T\quad \mbox{ and }\quad \E=\CMcal{S}'\circ_{\W}\T'
		\end{align*}
		such that
		\begin{itemize}
			\item if $\F$ admits a nontrivial factorization, say $\F=\CMcal{S}\circ_{\W}\T$, then $\CMcal{S}$ and $\T$ are ribbons and $\W\in \{(m-1),1^{m-1}\}$; \\
			otherwise, $\F$ has only a trivial factorization, and $(\CMcal{S},\W,\T)=(1,\varnothing,\F)$;
			\item for {\emma each} $\CMcal{G}\in \{\CMcal{S},\T\}$, {\emma we have} $\CMcal{G}'=\CMcal{G}$ or $\CMcal{G}'=\CMcal{G}^*$.
		\end{itemize}
		The equivalence class of $\F$ contains $2^{\kappa}$ elements, where $\kappa$ {\emma is} the number of occurrences of $\CMcal{G}\ne \CMcal{G}^*$ for $\CMcal{G}\in \{\CMcal{S},\CMcal{T}\}$, independent of the choice of nontrivial factorization of $\F$.
	\end{theorem}
	Here $\CMcal{S}\circ_\W\T$ denotes the composition operation {\emma obtained by} replacing each box of the ribbon $\CMcal{S}$ with a copy of $\T$, and {\emma joining} adjacent copies along the designated overlap $\W$; see Section~\ref{S:3} for the formal definition.
		
	Theorem \ref{T:1} is consistent with Conjecture \ref{Conj:1}. Recall from  \cite{BTvW} that an equivalence class of ribbons {\emma can} be as large as possible -- its size {\emma can be} any power of two. This shows that, {\emma although} connected skew diagrams with {\emma a} $2\times 2$ block of boxes exhibit features {\emma analogous to} ribbons in the determinantal formulas {\emma for} skew Schur functions (see for instance \cite{BR:10,Jin:18,KimYoo:21}), the {\emma presence of a} single $2\times m$ or $m\times 2$ block of boxes substantially reduces the {\emma size} of equivalence classes. 
	
	We prove Theorem \ref{T:1} by {\emma analyzing} the $h$-expansion of $s_\F$. The use of algebraic structure to study identities among Schur-type functions also appears in the Hopf-algebraic framework developed by Yeats \cite{Y:17}. Our argument is more specialized: rather than deriving identities from Hopf-algebra operations, we exploit the Jacobi–Trudi expansion and the cancellation pattern created by the unique rectangular overlap. Nevertheless, both approaches illustrate how algebraic structure can be used to organize and detect Schur-function identities.
	
	The major difficulty we {\emma must} overcome is the {\em cancellations} induced by the single $2\times m$ block in the Jacobi-Trudi formula {\emma for} skew Schur functions (see (\ref{eq:JT})). {\emma More precisely}, any connected skew diagram $\F$ with exactly one $2\times m$ block is obtained by piling two ribbons with an overlap of $m$ columns, denoted by $\F=\alpha\boxdot\beta$; see Subsection \ref{subsection:2.1} for the formal definition. However, not all coarsened compositions of $\alpha$ and $\beta$ contribute to the expansion of $s_{\F}$ (see Example \ref{Eg2}). {\emma We overcome these} cancellations by classifying positive integers into equivalence classes (see Definition \ref{def:equiv}) and {\emma carefully analyzing} when cancellations {\emma occur}.
	
	%Another advantage of this classification is an efficient verification of the periodicity when $\F$ has a nontrivial factorization.
	
	%{\em Proof sketch}. First we expand the skew Schur function $s_{\F}$ as a linear combination of homogeneous complete symmetric functions via the Jacobi-Trudi identity, yielding an expression (\ref{eq:JT}) for two equal skew Schur function $s_{\F}$. We next make use of this criterion to classify positive integers into equivalence classes of type $\mathcal{A}$ to $\mathcal{D}$, proving that if equivalence thickened ribbons of $\F=\alpha\boxdot \beta$ other than itself or its antipodal rotation $\F^*$ do exist, then $\alpha$ and $\beta$ must be concatenations of an identical ribbon and $\alpha\ne\beta$.    
	
	In parallel, {\emma equality of} skew Schur $Q$-functions has also been studied; see for instance \cite{BvW:09,GS:22,Sal:08}.
	
	The rest of the paper is {\emma organized} as follows: In Sections \ref{S:2} and \ref{S:3}, we {\emma recall} some preliminaries on skew Schur functions. In Section \ref{S:31}, {\emma we introduce} an equivalence relation on integer entries and {\emma present several preparatory} lemmas {\emma for the} proof. We then distinguish the cases based on whether {\emma the} two ribbons $\alpha$ and $\beta$ are of equal size for a given connected skew diagram $\F=\alpha\boxdot \beta$, and establish Theorem \ref{T:1} for each case in Sections \ref{S:5} and \ref{S:4}, respectively. {\emma We conclude the paper with final remarks in Section \ref{S:finalremark}.}

	%Since their introduction by Schur in 1901, the skew Schur functions lie at the center of the theory of symmetric functions and algebraic combinatorics. Yet a very basic question remains unsolved: for which skew shapes, their corresponding skew Schur functions are equal?
	
	%In terms of the Schur basis and monomial basis, skew Schur functions can be expanded as
	%\begin{align*}
	%	s_{\lambda/\mu}&=\sum_\nu c_{\mu\nu}^\lambda s_\nu\\
	%	s_{\lambda/\mu}&=\sum_\nu K_{\nu}^{\lambda/\mu} m_\nu
	%\end{align*}
	%where $c_{\mu\nu}^\lambda$, $K_{\nu}^{\lambda/\mu}$ are the Littlewood-Richardson coefficients and Kostka numbers, respectively.
	
	\section{Background and preliminaries}\label{S:2}
	We assume that the reader {\emma is} familiar with standard concepts of symmetric functions, the {\emma bases} of complete homogeneous symmetric functions and Schur functions, and skew Schur functions. {\emma Further} details can be found in \cite{Stanley}.
	
	A {\em composition} $\alpha$ of $n$ is a {\emma finite} sequence $\alpha_1\alpha_2\cdots\alpha_{\ell}$ of positive integers such that $\alpha_1+\alpha_2+\cdots+\alpha_{\ell}=n$, denoted by $\alpha\vDash n$. The {\em size} of $\alpha$, denoted $|\alpha|$, equals $n$, and the {\em length} of $\alpha$, denoted $\ell(\alpha)$, equals $\ell$.
	Each $\alpha_i$ is called a {\em part} of $\alpha$. The reverse of $\alpha$, denoted by $\alpha^*$, is the composition $\alpha_{\ell}\alpha_{\ell-1}\cdots\alpha_1$. {\emma For any positive integer $n\in \mathbb{N}^+$, we write $[1,n]=\{1,\ldots,n\}$.}
	
	A composition $\alpha=\alpha_1\alpha_2\cdots \alpha_{\ell}\vDash n$ can be naturally transformed into  a set
	\begin{align}\label{E:comtoset}
		\mathrm{SET}(\alpha)=\{\alpha_1,\alpha_1+\alpha_2,\ldots,\alpha_1+\alpha_2+\cdots+\alpha_{\ell-1}\}{\emma \subseteq [1,n-1]} .
	\end{align}
	Conversely, given a subset $S=\{s_1,s_2,\ldots,s_{\ell-1}\}$ of $[1,n-1]$ with $s_1<s_2<\cdots<s_{\ell-1}$, the corresponding composition is
	\begin{align*}
		\mathrm{Comp}(S)=s_1(s_2-s_1)(s_3-s_2)\ldots(n-s_{\ell-1})\vDash n.
	\end{align*} 	
	If a composition $\lambda=\lambda_1\lambda_2\ldots\lambda_{\ell}$ of $n$ {\emma satisfies} $\lambda_i\ge \lambda_{i+1}$ for all $1\le i<\ell$, then $\lambda$ {\emma is called} a {\em partition}, denoted $\lambda\vdash n$. Given a composition $\alpha$, let $\lambda(\alpha)$ denote the partition obtained by rearranging the parts of $\alpha$ in weakly decreasing order. To avoid {\emma ambiguity when consecutive parts are not clearly separated, we sometimes write compositions and partitions with brackets and commas}, i.e., $\alpha=(\alpha_1,\alpha_2,\ldots,\alpha_{\ell})$ and $\lambda=(\lambda_1,\lambda_2,\ldots,\lambda_{\ell})$.
	
	A partition $\lambda$ can be represented as a Ferrers diagram: {\emma place} $\lambda_1$ left-justified boxes in the first row (from the top), $\lambda_2$ boxes in the second row, and so on. A {\em skew diagram} $\lambda/\mu$ is the {\emma set-theoretic} difference of two Ferrers diagrams $\lambda$ and $\mu$ with $\mu\subseteq \lambda$. We denote by $|\lambda/\mu|$ the number of boxes in the skew diagram $\lambda/\mu$. If a box is located in the $i$-th row from the top and the $j$-th column from the left, we say that the box has {\em coordinates} $(i,j)$. %and {\em content} $c(i,j)=j-i$.
	
	The Jacobi-Trudi identity \cite{Stanley} {\emma gives} an expansion of {\emma the} skew Schur function $s_{\lambda/\mu}$ in terms of complete homogeneous symmetric functions: $h_{\lambda}=h_{\lambda_1}h_{\lambda_2}\cdots h_{\lambda_{\ell(\lambda)}}$:
	\begin{align}\label{E:JT}
		s_{\lambda/\mu}=\det(h_{\lambda_i-\mu_j-i+j})_{i,j=1}^{\ell(\lambda)}.
	\end{align}
	Relating (\ref{E:JT}) to Definition \ref{Def:equi1} {\emma yields} a second criterion {\emma for identifying} equivalent skew diagrams.
	\begin{definition}\label{Def:equi}
		{\emma For} any skew diagram $\F$ and a partition $\lambda$, let $[h_\lambda]s_{\F}$ denote the coefficient of $h_\lambda$ in the $h$-expansion of $s_{\F}$. Two skew diagrams $\F$ and $\E$ are equivalent if and only if $$[h_\lambda]s_{\F}=[h_\lambda]s_{\E}$$ 
		for all partitions $\lambda$ of $|\F|$.
	\end{definition}
	In the proof of Theorem \ref{T:1}, we use Definition \ref{Def:equi} to {\emma distinguish} non-equivalent skew diagrams by {\emma finding} a partition $\lambda$ {\emma for which} $[h_\lambda]s_{\F}\ne [h_\lambda]s_{\E}$.
	
	A skew diagram is {\em edgewise connected} if, for every pair of boxes $b$ and $b'$, there is a sequence of boxes
	$b=b_0,b_1,\ldots,b_k=b'$ such that \(b_i\) and \(b_{i+1}\) share an edge for every \(1\le i<k\). 
	A {\em ribbon} (also called a strip) is an edgewise connected skew diagram {\emma containing no} $2\times 2$ block of boxes. 
	
	Just {\emma as} partitions {\emma are represented} as Ferrers diagrams, compositions can be visualized as ribbons. For {\emma a composition} $\alpha=\alpha_1\alpha_2\cdots\alpha_{\ell}\vDash n$, the corresponding ribbon {\emma has} exactly $\alpha_i$ boxes in the $i$th row from the bottom. For example, the ribbon in Figure \ref{F:b2} uniquely corresponds to the composition $31231$.
	\begin{figure}[ht]
		\centering
		\includegraphics[scale=0.6]{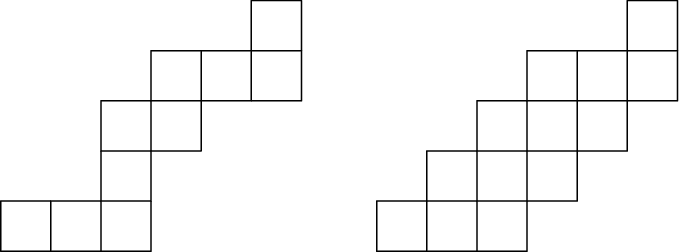}
		\caption{A ribbon $31231$ and a thickened ribbon $3\boxdot 3\boxdot 3 \boxdot 31$\label{F:b2}}
	\end{figure}
	\subsection{Thickened ribbons}\label{subsection:2.1}
	 We now recall the definition of thickened ribbons from \cite{BR:10}. The main focus of this paper is a {\emma particular family} of thickened ribbons, called $m$-regular thickened ribbons, {\emma along with} their conjugates.
	
	\begin{definition}\label{def:regular}
		A $k\times m$ block is the rectangular shape $(m^k)$. %By a $2\times m$ block we mean an inclusion-maximal two-row rectangle of width $m$; equivalently, it is not contained in a $2\times(m+1)$ rectangle. 
		A skew diagram is {\emma called} a {thickened ribbon} if it is edgewise connected and contains no $3\times3$ block.
		
		A thickened ribbon is called \emph{$m$-regular} {\emma for $m\ge 2$}, if it contains at least one inclusion-maximal $2\times m$ block {\emma and no $3\times2$ block}. A {conjugate $m$-regular thickened ribbon} is the transpose of an $m$-regular thickened ribbon.
	\end{definition}
	
	\begin{figure}[ht]
		\centering
		\begin{tikzpicture}[scale=.55]
			% A 3-regular example: the two rows overlap in exactly three columns.
			\foreach \x in {0,...,4} \draw (\x,0) rectangle ++(1,1);
			\foreach \x in {2,...,5} \draw (\x,1) rectangle ++(1,1);
			\node at (3.5,-.7) {$3$-regular};
			% A nonexample: the inclusion-maximal overlap has width four.
			\begin{scope}[xshift=9cm]
				\foreach \x in {0,...,3} \draw (\x,0) rectangle ++(1,1);
				\foreach \x in {1,...,3} \draw (\x,1) rectangle ++(1,1);
				\foreach \x in {2,...,4} \draw (\x,2) rectangle ++(1,1);
				\node at (2.8,-.7) {not $3$-regular};
			\end{scope}
		\end{tikzpicture}
		\emma{\caption{The left skew diagram is a $3$-regular thickened ribbon, whereas the right one is not $3$-regular, as it contains a $3\times 2$ block.}}
		\label{F:m3}
	\end{figure}
	
	Given any skew diagram $\F$, let $\F^t$ denote its {\em transpose}, obtained by reflecting $\F$ {\emma across} the diagonal from northwest to southeast (through boxes with coordinates $(i,i)$). {\emma As} we shall see in Lemma \ref{lem:reduction}, the problem of characterizing the equivalence classes of {\emma conjugate $m$-regular thickened ribbons reduces to the corresponding problem for $m$-regular thickened ribbons.}
	
	%Compared with ribbons, some $2\times m$ blocks (resp. $m\times 2$ blocks) can be part of an $m$-regular (resp. a conjugate $m$-regular) thickened ribbon. 
	
	\begin{lemma}\label{lem:reduction}
		Two skew diagrams $\F$ and $\E$ are equivalent, i.e., $\F\sim \E$, if and only if their transposes satisfy $\F^t\sim \CMcal{E}^t$.
	\end{lemma}
	\begin{proof}
		Recall the involution $\omega$ on the algebra of symmetric functions, which satisfies $\omega(s_{\F})=s_{\F^t}$ (see for instance \cite{Stanley}). Then 
		{\emma
		\begin{align*}
			\CMcal{F}\sim \CMcal{E}\iff s_{\F}=s_{\E} \iff \omega(s_{\F})=s_{\F^t}=s_{\E^t}=\omega(s_{\E}) \iff
			\F^t\sim \CMcal{E}^t.
		\end{align*}
	   This completes the proof.
       }
	\end{proof}

	\subsection{Schur function of regular thickened ribbons}
	Through the simple bijection between ribbons and compositions, each $m$-regular thickened ribbon {\emma can be analogously represented} as a \emph{box-dotted composition}: 
	a sequence of positive integers and box-dots such that the left and the right neighbours of every box-dot are integers of values at least $m$. 
	
	For an $m$-regular thickened ribbon $\F$, the sizes of parts of the corresponding composition are the numbers of boxes in each row, {\emma with} a box-dot {\emma placed between two adjacent parts} if the corresponding rows of $\F$ form a $2\times m$ block of boxes. For example, the $2$-regular thickened ribbon in Figure \ref{F:b2} is written as $\F=3\boxdot 3\boxdot 3\boxdot31$.
	
	For an $m$-regular thickened ribbon $\F=\F_1\boxdot\F_2\boxdot\cdots\boxdot\F_{l}$ where {\emma each} $\F_i$ is a ribbon, let $\alpha(\F)$ denote the underlying composition obtained by removing all box-dots from $\F$, and let $|\F|$ denote the number of boxes of $\F$.
	The {\em antipodal rotation} of $\F$, 
	denoted $\F^*$, is the thickened ribbon $\F_{l}^*\boxdot \F_{l-1}^*\boxdot\cdots\boxdot \F_{1}^*$.
	It is {\emma immediate} that $\alpha(\F^*)=\alpha(\F)^*$. 
	Furthermore, we define $\lambda(\F)=\lambda(\alpha(\F))$ and $\ell(\F)=\ell(\alpha(\F))$.
	
	Let $(\F;k)$ denote a pair consisting of an $m$-regular thickened ribbon $\F$ and a nonnegative integer $k$, the \emph{size} of $(\F;k)$ is defined as $|(\F;k)|=|\alpha(\F)|+k(m-1)$. For convenience, if $k=0$, we simply write $(\F;0)=(\F)$.
	
	Let $\CMcal{P}_n$ be the set of pairs $(\F;k)$ of size $n$. We introduce a partial order $\preceq$ on $\CMcal{P}_n$ as follows: For $(\CMcal{S};k), (\CMcal{T};p)\in\CMcal{P}_n$, let $\alpha(\CMcal{S})=\alpha_1\alpha_2\cdots\alpha_\ell$. We say that $(\CMcal{T};p)$ {\em covers} $(\CMcal{S};k)$ if one of the following {\emma holds}:
	\begin{enumerate}
		\item $p=k$, for some $i$, {\emma there is} no box-dot between $\alpha_i$ and $\alpha_{i+1}$ in $\CMcal{S}$, and $\CMcal{T}$ is obtained from $\CMcal{S}$ by merging $\alpha_i\alpha_{i+1}$ into $\alpha_i+\alpha_{i+1}$;
		\item $p=k+1$, for some $i$, a box-dot appears between $\alpha_i$ and $\alpha_{i+1}$ in $\CMcal{S}$, and $\CMcal{T}$ is obtained from $\CMcal{S}$ by replacing $\alpha_i\boxdot\alpha_{i+1}$ {\emma with} $\alpha_i+\alpha_{i+1}-m+1$.
	\end{enumerate}
	Further, we say that $(\CMcal{T};p)$ \emph{coarsens} $(\CMcal{S};k)$ if $(\CMcal{S};k)\preceq(\CMcal{T};p)$.
	
	For example, when $m=2$, $(5\boxdot 3\boxdot 31;1)$ covers $(3\boxdot 3\boxdot 3\boxdot31)$, $(3\boxdot 5\boxdot 31;1)$ covers $(3\boxdot 3\boxdot 3\boxdot31)$ and $(3\boxdot 3\boxdot 3\boxdot 4)$ covers $(3\boxdot 3\boxdot 3\boxdot31)$.
	
    {\emma 
	\begin{remark}
		Note that the merge operation defined by (1) and (2) preserves the size of $(\CMcal{S};k)$. 
		
		If $p=k$, then (1) implies that $\alpha(\CMcal{T})$ is obtained by merging two adjacent parts of $\alpha(\CMcal{S})$. It follows that $|\alpha(\CMcal{T})|=|\alpha(\CMcal{S})|$ and hence $|(\CMcal{T};k)|=|(\CMcal{S};k)|$. 
		
		If $p=k+1$, then (2) gives $|\alpha(\CMcal{T})|=|\alpha(\CMcal{S})|-m+1$. Consequently $|(\CMcal{T};k+1)|=|\alpha(\CMcal{T})|+(k+1)(m-1)=|\alpha(\CMcal{S})|+k(m-1)=|(\CMcal{S};k)|$.
	\end{remark}}

	%Both covering moves preserve $|(\F;k)|=|\alpha(\F)|+k(m-1)$. The ordinary merge leaves both terms unchanged in total, while the $\boxdot$-merge decreases $|\alpha(\F)|$ by $m-1$ and increases the auxiliary integer by one. Thus every element of $\CMcal{P}_n$ records a coarsening of the same total degree $n$.
	
	We will apply the Jacobi-Trudi identity (\ref{E:JT}) to express skew Schur functions {\emma of thickened ribbons} in terms of complete symmetric functions.
	
	\begin{proposition}\label{prop:JT}
		The skew Schur function of an $m$-regular thickened ribbon 
		$\F$ is given by
		\begin{align}\label{E:tform}
			s_{\F}=(-1)^{\ell(\F)}\sum_{(\CMcal{S};k)\succeq(\F)}
			(-1)^{\ell(\CMcal{S})}\,
			h_{m-1}^{k}\,h_{\lambda(\CMcal{S})}.
		\end{align}
	\end{proposition} 
	\begin{proof}
		Fix an $m$-regular thickened ribbon $\F$, suppose that  $\alpha(\F)=\alpha_1\alpha_2\cdots\alpha_k$. If there is no box-dot between the last two parts $\alpha_{k-1}$ and $\alpha_k$, then we have
		\begin{align*}
			s_{\F}=s_{\F_1}h_{\alpha_k}-s_{\F_2}
		\end{align*}
		where $\F_1$ is obtained from $\F$ by removing the trailing part $\alpha_k$, and $\F_2$ is obtained from $\F$ by merging $\alpha_{k-1}\alpha_k$ into $\alpha_{k-1}+\alpha_k$.
		
		Otherwise, $\F$ ends with $\alpha_{k-1}\boxdot\alpha_k$, and we have
		\begin{align*}
			s_{\F}=s_{\F_1}h_{\alpha_k}-h_{m-1}s_{\F_2}
		\end{align*}
		where $\F_1$ is obtained from $\F$ by removing the trailing $\boxdot\alpha_k$, and $\F_2$ is obtained from $\F$ by replacing $\alpha_{k-1}\boxdot\alpha_k$ with $\alpha_{k-1}+\alpha_k-m+1$.
		
		The formula (\ref{E:tform}) then follows by induction on $\ell(\F)$.
	\end{proof}
	\begin{example}
		For {\emma the} $2$-regular thickened ribbon $\F=3\boxdot 3\boxdot 31$, we have 
		\begin{align*}
			s_{\F}=h_{3}h_{3}h_3h_1-2h_5h_3h_1^2-h_3h_3h_4+h_7h_1^3-h_8h_1^2+h_5h_4h_1+h_3h_6h_1.
		\end{align*}
		The terms correspond {\emma respectively} to the pairs $(3\boxdot3\boxdot31),(5\boxdot 31;1),(3\boxdot 51;1),(3\boxdot3\boxdot 4),(71;2)$, $(8;
		2),(5\boxdot 4;1)$ and $(3\boxdot 6;1)$, all of which coarsen $(3\boxdot3\boxdot 31)$.
	\end{example}
	\begin{example}\label{Eg2}
		{\emma Consider the} two $2$-regular thickened ribbons $\F=122\boxdot 3$ and $\E=12\boxdot 32$ (see Figure \ref{F:e2}). By Proposition \ref{prop:JT}, we have
		\begin{align*}
			s_{\F}=s_{\E}=h_3h_2^2h_1-h_3^2h_2-h_4h_2h_1^2+h_5h_3+h_6h_1^2-h_7h_1,
		\end{align*}
	    thus $\F\sim \E$.
		\begin{figure}[ht]
			\centering
			\includegraphics[scale=1.0]{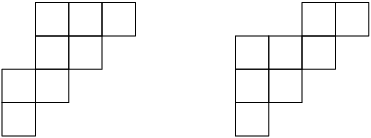}
			\caption{Two equivalent thickened ribbons $\F=122\boxdot 3$ (left) and $\E=12\boxdot 32$ (right). \label{F:e2}}
		\end{figure}
		An important point here is that although both $(14\boxdot 3),(34;1)\succeq (\F)$, they appear with opposite signs in (\ref{E:tform}) and therefore cancel each other. The same phenomenon occurs for $(12\boxdot 5),(52;1)$ with respect to $(\E)$.
	\end{example}
	Example \ref{Eg2} illustrates a {\emma key} difference between ribbons and thickened ribbons: every coarsening of a composition contributes a non-vanishing term to the $h$-expansion of a ribbon Schur function (see  \cite[Proposition 2.1]{BTvW}), but this is no longer true for thickened ribbons. 
	
	An immediate consequence of Proposition \ref{prop:JT} is {\emma the following} necessary condition for two equivalent thickened ribbons. 
	\begin{corollary}\label{lem:basic}
		If $\F=\F_1\boxdot\F_2\boxdot\cdots\boxdot\F_p$ and $\E=\E_1\boxdot\E_2\boxdot\cdots\boxdot\E_k$ are equivalent $m$-regular thickened ribbons where each $\F_i$ and $\E_i$ is a ribbon, then %$p=k$ and 
		$\lambda(\F)=\lambda(\E)$. 
		%as multisets, one has $\left\{\vert\F_1\vert, \vert \F_2\vert,\ldots,\vert\F_m\vert\right\}=\left\{\vert\E_1\vert, \vert\E_2\vert,\ldots,\vert\E_m\vert\right\}$.
	\end{corollary}
	This corollary is covered by a more general result in \cite[Proposition 8.4 and Corollary 8.11]{RSvW}, {\emma which} states that two {\emma equivalent} skew diagrams must contain the same number of $m\times \ell$ rectangles as subdiagrams for any $m,\ell\in\mathbb{N}^+$. We include a short proof here for completeness.
	\begin{proof}
		By definition, $|\F_i|\geq m$ for all $i$. Among all $(\CMcal{S};0)\succeq (\F)$, the pair $(|\F_1|\boxdot|\F_2|\boxdot\cdots\boxdot|\F_p|;0)$ is the unique maximal element. In other words, among all partitions $\lambda$ with no part of size $(m-1)$ such that $[h_{\lambda}]s_{\F}\ne 0$ in (\ref{E:tform}), the unique shortest partition is $\lambda(\F)$, of length $p$. The same statement holds for $\E$ with $p$ replaced by $k$. Since $\F\sim \E$, we have $[h_{\lambda}]s_{\F}=[h_{\lambda}]s_{\E}$ for all partitions $\lambda$, and therefore $\lambda(\F)=\lambda(\E)$, as desired.
	\end{proof}
	Consequently, two {\emma equivalent} $m$-regular thickened ribbons must contain the same number of $2\times m$ blocks. Hence it suffices to study equivalence classes of thickened ribbons with a fixed number of {\emma such} blocks. 
	
	\section{A composition of thickened ribbons}\label{S:3}
	The purpose of this section is to introduce a composition of thickened ribbons, {\emma following \cite{BTvW,MvW}. We will describe the operation $\circ_{\W}$ specifically} for $m$-regular thickened ribbons, 
	and refer the readers to \cite[Section 3]{MvW} for its general definition for arbitrary skew diagrams.
	
	The composition {\emma $\CMcal{S}\circ_\W\T$} is best viewed as {\emma a substitution}: each box of {\emma the ribbon $\CMcal{S}$} is replaced by a copy of {\emma the ribbon $\T$}, and neighbouring copies {\emma with an overlapping $\W$} are joined according to whether the corresponding boxes in $\CMcal{S}$ are horizontally or vertically adjacent. 
	
	%In the one-block setting considered here, the overlap $\W$ is either empty or a row of $m-1$ boxes, so the general construction reduces to four attachment patterns.
	
	In what follows, unless otherwise specified, we assume that $\F=\alpha\boxdot\beta$ is an $m$-regular thickened ribbon with exactly one $2\times m$ block. For any $m\ge 2$, let $\W\in\{\varnothing, m-1\}$, i.e., $\W$ is empty, or a row of $m-1$ boxes. If $\W=\varnothing$, we omit the subscript $\W$ from $\circ_{\W}$ and call $\F=1\circ\F=\F\circ 1$ a {\em trivial} factorization of $\F$.
	
	Otherwise, $\W=m-1$, we define a {\em nontrivial} factorization of $\F$, written $\F=\CMcal{S}\circ_{\W} \T$, where $\CMcal{S}\ne 1$ and $\T\ne 1$ are ribbons.
	
	Given $\F$ and $\W$, we say that $\W$ is located at the bottom (resp. top) of $\F$ if $\W$ is a connected subdiagram of $\F$ containing the southwesternmost (resp. northeasternmost) box of $\F$.
	
	Let $\O$ denote the subdiagram of $\F$ obtained by removing copies of $\W$ at the bottom and top. If such an $\O$ does not exist, or if $\O$ is not a skew diagram, then $\F$ admits only a trivial factorization; otherwise,  $\F$ may have a nontrivial decomposition.
	\begin{example}
		For $m=2$ and $\W=(1)$, we have the following decompositions, respectively, of $31412\boxdot312$, $121412\boxdot 312$, $31412\boxdot3111$ and $121412\boxdot 3111$, as shown from left to right in Figure \ref{F:3}.
		\begin{figure}[ht]
			\centering
			\includegraphics[scale=0.4]{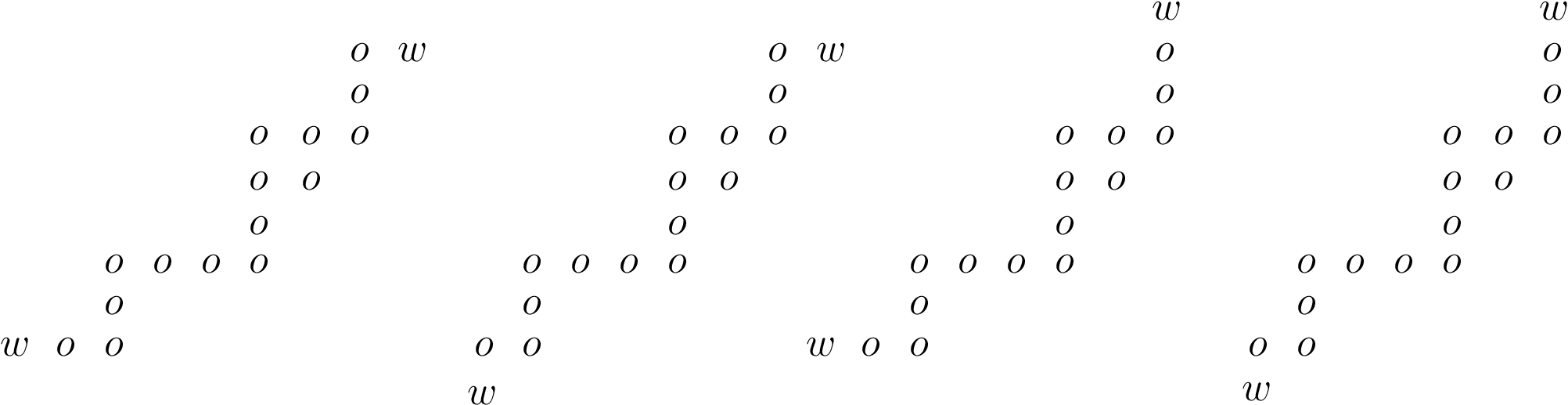}
			\caption{{\emma The four types of $2$-regular thickened ribbons with $\W=(1)$ are} $\W\rightarrow\O\rightarrow\W$, $\W\uparrow\O\rightarrow \W$, 
				$\W\rightarrow\O\uparrow \W$ and $\W\uparrow\O\uparrow \W$, where the symbols $w$ and $o$ denote boxes belonging to $\W$ and $\O$ respectively.\label{F:3}}
		\end{figure}
	\end{example}	
	If the bottommost (resp. topmost) copy of $\W$ lies to the left (resp. right) of the southwesternmost (resp. northeasternmost) box of $\O$, then we say {\emma that} the lower (resp. upper) copy of $\W$ is {\em horizontally attached to} $\O$, denoted by $\W\rightarrow \O$ (resp. $\O \rightarrow \W$). Similarly, if the bottommost (resp. topmost) copy $\W$ of $\F$ is {\emma directly} below (resp. right above) the southwesternmost (resp. northeasternmost) box of $\O$, then we say the lower (resp. upper) copy of $\W$ is {\em vertically attached to} $\O$, denoted by $\W\uparrow \O$ (resp. $\O \uparrow \W$). {\emma Thus}, there are four {\em types} for thickened ribbons (and in particular for ribbons): 
	\begin{align*}
	\W\rightarrow \O\rightarrow \W,\quad \W\uparrow \O\rightarrow \W,\quad \W\rightarrow \O\uparrow \W, \quad\mbox{ and }\quad\W\uparrow \O\uparrow \W; 
	\end{align*}
	see Figure \ref{F:3} for examples.
	
	Given two ribbons $\CMcal{S}$ and $\T$, we define the composition $\CMcal{S} {\circ}_\W\T$ with $\W=m-1$ as follows. Each box $d$ of $\CMcal{S}$ will contribute a copy of $\T$, denoted by $\T_d$, in the plane. We {\emma arrange the} copies $\T_d$ according to the positions of $d$ in $\CMcal{S}$.
	
	For two neighbouring boxes $d,d'$ of $\CMcal{S}$, if $d$ is one position west of $d'$, then we combine $\T_d$ and $\T_{d'}$ by identifying the upper copy of $\W$ in $\T_d$ with the lower copy of $\W$ in $\T_{d'}$, denoted $\T_{d}\amalg \T_{d'}$.
	If $d$ is one position south of $d'$, {\emma the construction differs as follows}:
	\begin{itemize}
		\item If $\T=\W\rightarrow\O\rightarrow\W$ (resp. $\T=\W\uparrow\O\uparrow\W$), then we position $\T_{d'}$ so that the lower copy of $\W$ in $\T_{d'}$ is one position northwest (resp. southeast) of the upper copy of $\W$ in $\T_d$.
		\item If $\T=\W\rightarrow\O\uparrow\W$ (resp. $\T=\W\uparrow\O\rightarrow\W$), then form $\T_{d}\amalg \T_{d'}$ and make an extra copy of $\W$ one position southeast (resp. northwest) from the intersection $\T_d\cap \T_{d'}$.	
	\end{itemize}
	\begin{example}
		In Figure \ref{F:2} we have $\W=1$, $\CMcal{S}=21$ and $\T=312$ of type $\W\rightarrow\O\rightarrow\W$. The resulting composition $\CMcal{S}\circ_{1}\T$ {\emma combines} the two ribbons $31412$ and $312$ in such a way that a $2\times 2$ block is created in between.
	\end{example}
	\begin{figure}[ht]
		\centering
		\includegraphics[scale=0.45]{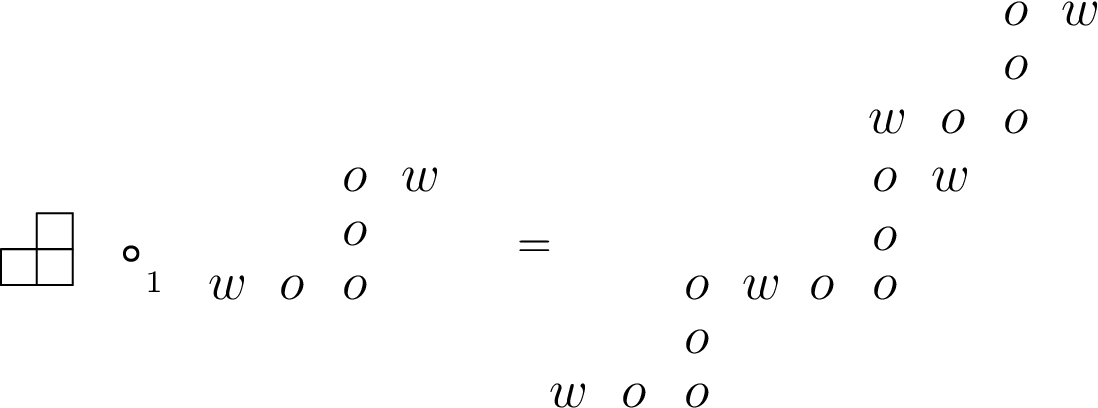}
		\caption{A composition of two ribbons: $21\circ_{1}312=31412\boxdot 312$ where $w$ and $o$ denote boxes of $\W$ and $\O$ respectively.\label{F:2}}
	\end{figure}
	\begin{remark}
		A given thickened ribbon {\emma may admit more than one} nontrivial factorization. For instance, when $m=2$, $\F=31412\boxdot 31412=11\circ_1 31412=22\circ_1 312$.
	\end{remark}
	The types of $\T$ and $\F$ are the same whenever $\F=\CMcal{S}\circ_\W\T$. Furthermore, the number of rows of $\CMcal{S}$ is one more than the number of $2\times m$ blocks in the $m$-regular thickened ribbon $\F$. In particular, $\CMcal{S}$ has exactly two rows if $\F$ has only one $2\times m$ block.
	
	The sufficient condition for {\emma equivalence of two} thickened ribbons was proved in \cite{MvW}, and also appears as a special case of \cite[Theorem 5.3.9]{OH:24}.
	
	\begin{theorem}\label{thm:MvW}(A special case of \cite[Theorem 3.31]{MvW} and \cite[Theorem 5.3.9]{OH:24})
		Let $\F=\alpha\boxdot\beta$ be an $m$-regular thickened ribbon with exactly one $2\times m$ block. If $\F$ has a nontrivial factorization $\F=\CMcal{S}\circ_{\W}\T$ with $\CMcal{S}\ne 1$, $\T\neq 1$ and $\W=m-1$, then
		\begin{align*}
			\F\sim \CMcal{S}^*\circ_\W \T\sim \CMcal{S}\circ_\W\T^*\sim\CMcal{S}^*\circ_\W\T^*.
		\end{align*}
		Similarly $\F^t=\CMcal{S}^*\circ_{\W^t}\T^t$, as a conjugate $m$-regular thickened ribbon with exactly one $m\times 2$ block, satisfies (with $\W^t=1^{m-1}$)
		\begin{align*}
			\F^t\sim \CMcal{S}\circ_{\W^t} \T^t\sim \CMcal{S}\circ_{\W^t}(\T^t)^*\sim\CMcal{S}^*\circ_{\W^t}(\T^t)^*.
		\end{align*}
	\end{theorem}
	
	\begin{example}
		The $2$-regular thickened ribbon $\F=31412\boxdot 312$ satisfies
		$$\F= 21 \circ_1 312\sim 12\circ_1 312\sim 21\circ_1 213\sim 12\circ_1 213,$$
		where $12\circ_1 312=312\boxdot 31412$, $ 21\circ_1 213=21413\boxdot 213$ and $12\circ_1 213=213\boxdot 21413$; see Figure \ref{F:4}.
		\begin{figure}[ht]
			\centering
			\includegraphics[scale=0.8]{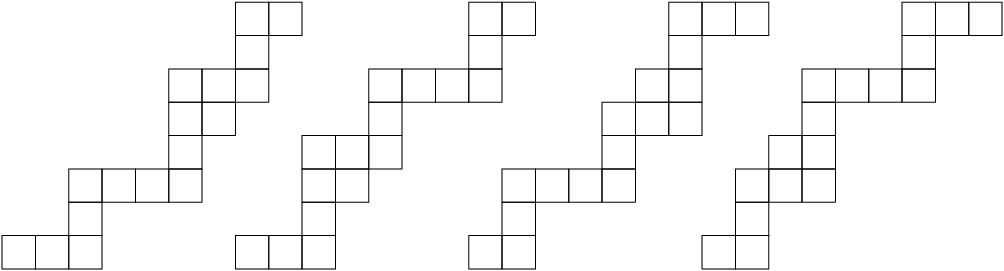}
			\caption{Four equivalent thickened ribbons: $31412\boxdot 312$, $312\boxdot 31412$, $21413\boxdot 213$ and $213\boxdot 21413$ from left to right.\label{F:4}}
		\end{figure}
	\end{example}
	
	%Let us recall that $\F=\alpha\boxdot\beta$, then $\alpha_1=\T_{1}$ and $\beta_{\ell(\beta)}=\T_{\ell(\T)}$ so that (\ref{E:con1}) holds if and only if $\T_1+\T_{\ell(\T)}\ge 3$.
	
	\section{A classification of positive integers}\label{S:31}
	
	Throughout Sections~\ref{S:31}--\ref{S:4}, let $\F=\alpha\boxdot\beta$ be an $m$-regular thickened ribbon with exactly one $2\times m$ block. Set
	\begin{align}\label{E:nr}
	n=|\alpha|+|\beta|\quad \mbox{ and } \quad r=\gcd(|\alpha|-m+1,|\beta|-m+1).
	\end{align}
	For such a diagram, formula ({\ref{E:JT}) simplifies to
	\begin{align}\label{eq:JT}
		s_{\F}=\sum_{(\CMcal{S})\succeq(\F)}
		(-1)^{\ell(\F)+\ell(\CMcal{S})}h_{\lambda(\CMcal{S})}-\sum_{(\gamma;1)\succeq(\F)}
		(-1)^{\ell(\F)+\ell(\gamma)+1}h_{\lambda(\gamma)}h_{m-1}.
	\end{align}
	The purpose of this section is to convert the {\emma row-size sequence $\alpha(\F)$ of $\F$ into signs and types} that can be compared across the equivalence class $[\F]$.

	\begin{definition}\label{Def:type}
		Let $n=|\alpha|+|\beta|$. A function 
		$$\bh_{\F}:{\emma [1,n-m]}\to\{+,-\}$$ 
		is defined as follows: For $1\le x\le n-m$, set $\bh_{\F}(x)=+$ if $(x,n-m+1-x;1)\succeq (\F)$, and $\bh_{\F}(x)=-$ otherwise. 
		
		We say {\emma that} $x$ is of {type} $i$ if exactly $i$ of the two equalities $\bh_\F(x)=+$ and $\bh_\F(n-m+1-x)=+$ hold. For $i\in \{0,1,2\}$, the set of elements of type $i$ is denoted by $A_i$. In particular, if $x=n-m+1-x$, then $x\in A_0\cup A_2$.
	\end{definition}
	
	{\emma Intuitively, $\bh_\F(x)=+$ (resp. $-$) means that some initial consecutive rows of $\F$ can be merged so that $x$ boxes can (resp. cannot) be peeled off from the bottom. %with the unique overlap accounted for by the $m-1$ correction. %The reflected value $\bh_\F(n-m+1-x)$ asks the analogous question from the top. 
		Thus, the type of $x$ records whether this peeling is possible from neither end, exactly one end, or both ends. 
		
		As in the ribbon case by Billera, Thomas and van Willigenburg \cite{BTvW}, the positive breakpoints encode the associated compositions; here, this reconstruction is made explicit by the sets $S_\F$ and $T_\F$ in the proof of Lemma \ref{lem:canonical2}.}
	
	%The function $\bh_{\F}$ is inherited from the ribbon invariant used by Billera, Thomas and van Willigenburg \cite{BTvW}. For ribbons, the analogous value of $h$ records whether one can peel off a specified number of boxes by removing entire rows from the bottom. Here $\bh_{\F}(x)=+$ has the same flavour: it says that the coarsening process can isolate a bottom contribution of size $x$ after accounting for the single overlap. The type of $x$ asks this question from both ends at once, comparing $x$ with its complementary value $n-m+1-x$.
	
	%\begin{definition}
	%A function $\bh_{\F}:\{1,2,\Fots,|\alpha|+|\beta|-2\}\to\{+,-\}$ is defined as follows: $\bh_{\F}(x)=+$ if $(x,|\alpha|+|\beta|-1-x;1)\succeq (\F)$; $\bh_{\F}(x)=-$ if $(|\alpha|+|\beta|-1-x,x;1)\succeq (\F)$. We say $x$ is of type $i$ if exactly $i$ of the two equalities $\bh_\F(x)=+$, $\bh_\F(|\alpha|+|\beta|-1-x)=+$ hold. For $i=0,1,2$, the set of elements of type $i$ is denoted by $A_i$.
	%\end{definition}
	
	\begin{example}\label{Example:D}
		Let $\F=132\boxdot 42122$ be a $2$-regular thickened ribbon, then the types and the function $\bh_{\F}$ are given as below:
		\begin{center}
			\begin{tabular}{c|ccccccccccccccc}
				$x$ & $1$ & $2$ & $3$ & $4$ & $5$ & $6$ & $7$ & $8$ & $9$ & $10$ & $11$ & $12$ & $13$ & $14$ & $15$ \\
				\hline \\[-1.5ex]
				type & $1$ & $1$ & $0$ & $2$ & $1$ & $0$ & $1$ & $0$ & $1$ & $0$ & $1$ & $2$ & $0$ & $1$ & $1$ \\
				$\bh_{\F}(x)$ & $+$ & $-$ & $-$ & $+$ & $-$ & $-$ & $-$ & $-$ & $+$ & $-$ & $+$ & $+$ & $-$ & $+$ & $-$ \\
			\end{tabular}
		\end{center}
	\end{example}
	
	Fix $\F=\alpha\boxdot\beta$, we define an equivalence class $[i]$ for each positive integer $1\leq i\leq |\alpha|+|\beta|-m$, so that for any $x\in [i]$, the value $\bh_{\F}(x)$ is determined by $\bh_{\F}(i)$ and the $h$-expansion of $s_{\F}$. 
	
	\begin{definition}\label{def:equiv}
		For an $m$-regular thickened ribbon $\F=\alpha\boxdot\beta$ {\emma of size $n=|\alpha|+|\beta|$}, {\emma we define an} equivalence relation $\sim$ {\emma on the set $[1,n-m]$} as the closure of the following relations:
		\begin{enumerate}
			\item For $1\leq i\leq n-m$, we have $i\sim n-m+1-i$.
			\item For $1\leq i\leq|\alpha|-1$, we have $i\sim |\alpha|-i$.
			\item For $|\alpha|-m+2\leq i\leq n-m$, we have $i\sim |\alpha|+n-2m+2-i$.
		\end{enumerate}
		%For $i\ge n-1$, $i\sim j$ if $i \mod (n-2)$ is equivalent to $j \mod (n-2)$. The equivalence class containing $i$ is denoted by $[i]$.
	\end{definition}
	
	%\begin{lemma}\label{lem:symmetry}
	%For any $1\leq i\leq |\alpha|+|\beta|-2-r$, $i$ is equivalent to $i+r$. Consequently, any equivalence class $[i]$ satisfies $[i]\cap\{1,2,\Fots,r\}=\{j,r+1-j\}$ with $j=\min[i]$.
	%\end{lemma}
	\begin{remark}
		The integer equivalence classes $[i]$ are new to the thickened ribbon setting, whereas $\bh_\F$ extends the function $h$ used for ribbons in~\cite{BTvW}. In the ribbon case, the signs of $h$ mark the breakpoints from which the composition can be constructed, and periodicity of this sign pattern detects a nontrivial composition factor. The present function {\emma $\bh_{\F}$ plays} the same role, except that the single overlap introduces the correction term $m-1$ and possible cancellations in the Jacobi-Trudi expansion.
	\end{remark}
	
	In the next lemma, the elements of each equivalence class $[i]$ are characterized.
	
	\begin{lemma}\label{lem:symmetry}
		For any $1\le i,j\le n-m$, $i$ and $j$ are equivalent if and only if $i\equiv j \mod r$ or $i\equiv m-1-j \mod r$,  where {\emma $n$ and $r$ are defined in (\ref{E:nr}).}
		
		Consequently, any equivalence class $[i]$ satisfies that {\emma $[i]\cap[1,r]$} contains exactly two elements unless $i\equiv m-1-i\mod r$, in which case {\emma $[i]\cap [1,r]=\{\min[i]\}$.}
	\end{lemma}
	\begin{proof}
		Suppose that $j\sim i$ through one of the relations $(1)$--$(3)$ in Definition \ref{def:equiv}, namely, $j=n-m+1-i$, $j=|\alpha|-i$ or $j=|\alpha|+n-2m+2-i$. In each case, we have $j\equiv m-1-i\mod r$. Since the map $i\mapsto m-1-i$ is an involution, transitivity of {\emma the} equivalence relation {\emma implies} that for any $i\sim j$,
		\begin{align*}
		j\equiv m-1-i\mod r\quad \mbox{ or }\quad j\equiv i \mod r.
	    \end{align*}
	
		Conversely, we {\emma will} show that $j\in [i]$ {\emma whenever} $j\equiv i\mod r$ or $j\equiv r+m-1-i \mod r$. For this, it suffices to prove that $i+r\in [i]$ and $r+m-1-i\in [i]$ for all $i$ such that $1\le i,i+r,r+m-1-i\le n-m$.
		
		For $1\leq i\leq |\beta|-m+1$,  relation (1) of Definition \ref{def:equiv} gives 
		$$n-m+1-i \sim i.$$
		Since $|\alpha|\le n-m+1-i\le n-m$, relation (3) yields
		$$i+|\alpha|-m+1 \sim n-m+1-i,$$
		hence  
		\begin{align}\label{E:ia}
			i+|\alpha|-m+1\in [i]. 
		\end{align}
		For $|\beta|-m+2\leq i\leq n-m$, relation (1) again gives $n-m+1-i\sim i$. Since $1\leq n-m+1-i \leq |\alpha|-1$, relation (2) implies $$n-m+1-i\sim i-|\beta|+m-1$$ and consequently
		\begin{align}\label{E:ib}
			i-|\beta|+m-1\in [i].
		\end{align}
		Without loss of generality assume $|\alpha|\le |\beta|$.  Since $r=\gcd(|\alpha|-m+1,|\beta|-m+1)$, the Euclidean algorithm {\emma gives}
		$\pm r=a(|\alpha|-m+1)-b(|\beta|-m+1)$ for some nonnegative integers $a,b$. Therefore, for any $1\leq i\leq n-r-m$, applying relations (\ref{E:ia}) $a$ times and (\ref{E:ib}) $b$ times alternatively yields $i+r\in[i]$, as desired.
		
		It remains to show that for $1\leq i\leq r$, we have $r+m-1-i\in [i]$. Since $|\alpha|-(r+m-1-i)\equiv i\mod r$, {\emma it follows from  (2) of Definition \ref{def:equiv} and the fact that $i+r\in [i]$ that} $$r+m-1-i\sim |\alpha|-(r+m-1-i)\sim i,$$ hence $r+m-1-i\in [i]$. In particular, $[i]\cap [1,r]$ contains exactly two elements only if $i\not\equiv m-1-i\mod r$. This completes the proof.
	\end{proof}
	\begin{example}
		We continue with the thickened ribbon $\F=132\boxdot 42122$ from Example \ref{Example:D}. Here  $\alpha=132$, $\beta=42122$, $m=2$, and $r=\gcd(|\alpha|-m+1,|\beta|-m+1)=5$. The equivalence class $[2]$ is $\{2,4,7,9,12,14\}$, so $[2]\cap [1,5]$ contains exactly two elements $2,4$. The equivalence class $[3]$ is $\{3,8,13\}$, so that $[3]\cap [1,5]$ has only one element $3$.
	\end{example}
	
	\section{When two ribbons are of equal size}\label{S:5}
	Throughout Section~5, $\F=\alpha\boxdot\beta$ denotes an $m$-regular thickened ribbon with exactly one $2\times m$ block and $|\alpha|=|\beta|$, {\emma unless otherwise stated.} %By Lemma~\ref{lem:reduction}, it suffices to study $m$-regular thickened ribbons rather than their conjugates. 
	{\emma Recall that Lemma~\ref{lem:reduction} reduces the proof to $m$-regular thickened ribbons.}
	The equal-size case is {\emma somewhat} simpler than the unequal-size case.
	
	The proof mainly {\emma consists of} the following steps: \vspace{1mm}
	\begin{itemize}
		\item The sufficient condition for $\F\sim \E$ follows directly from Theorem \ref{thm:MvW} and the simple observation that $\F\sim \F^*$.\vspace{1mm}
		\item For the necessity, we {\emma analyze} the types of elements.
		If all elements are of even type, i.e., $\F=\F^*$, then $[\F]=\{\F\}$. Otherwise, $\F$ has type-one elements, i.e., $\F\ne \F^*$. For any $\E$ such that $\E\sim \F$ and $\bh_{\F}(k)=\bh_{\E}(k)=+$ for a fixed $k\in A_1$, we {\emma show} that $\E=\F$ by Lemma \ref{lem:canonical2} and Theorem \ref{thm:a=b}. Consequently,  {\emma regardless of} whether $\F$ admits a nontrivial factorization, $[\F]=\{\F,\F^*\}$. 
	\end{itemize}
	The starting point of the proof is a relation between the equivalence class of an $m$-regular thickened ribbon $\F$ and the function $\bh_{\F}$.
	\begin{lemma}\label{lem:canonical2}
		Let $\F=\alpha\boxdot\beta$ with $|\alpha|\le |\beta|$. Then for any $\E=\alpha'\boxdot\beta'$ such that $\E\sim\F$, we have $|\alpha'|=|\alpha|$ and $|\beta'|=|\beta|$. {\emma Moreover}, 
		$$\bh_\E=\bh_\F \iff \E=\F.$$
	\end{lemma}
	
	\begin{proof}
		Consider the sets
		\begin{align*}
			S_{\F}&=\{1\leq x\leq |\alpha|-1:\bh_\F(x)=+\}=\mathrm{SET}(\alpha),\\
			T_{\F}&=\{1\leq x\leq |\alpha|-1:\bh_\F(|\alpha|-m+1+x)=+\}=\mathrm{SET}(\beta),
		\end{align*}
		that is, $\alpha=\mathrm{Comp}(S_{\F})$ and $\beta=\mathrm{Comp}(T_{\F})$. It follows that the equality $\E=\F$ is equivalent to $S_{\F}=S_{\E}$ and $T_{\F}=T_{\E}$, which in turn is equivalent to $\bh_\E=\bh_\F$. {\emma Hence the lemma follows.}
	\end{proof}
	
	Lemma \ref{lem:canonical2} {\emma implies} that when $|\alpha|=|\beta|$, the size of $[\F]$ equals the number of functions $\bh_{\E}$ for all $\E\sim \F$. {\emma As we shall see, there are only} two such functions $\bh_{\F}$ and $\bh_{\F^*}$. To set up the proof framework, we need to introduce {\emma several} concepts associated {\emma with} the function $\bh_{\E}$, all of which are independent of the choice of representative $\E$ in the equivalence class of $\F$.
	
	%\subsection{A classification of equivalence classes}
	\subsection{Invariants under equivalence}
	{\emma The purpose of this subsection is to identify the elements $x$ for which $\bh_{\F}(x)=\bh_{\E}(x)$ for all $\E\sim\F$, before we analyze the behavior of the sign function $\bh_{\F}$ on the remaining elements.}

		%isolate the information that is fixed by the skew Schur function before we try to reconstruct the sign function $\bh$. 
		%In the equal-size case, the type of every integer is invariant, although the individual sign $\bh_\F(x)$ at a type $1$ position can still vary among equivalent canonical thickened ribbons. This residual sign ambiguity is what the remainder of Section \ref{S:5} resolves.
	
	Let us begin by showing that the types of elements (see Definition \ref{Def:type}) are uniform {\emma within} any equivalence class.
	\begin{lemma}\label{lem:types2}
		For any $\E\sim \F$, element $x$ is of type $i$ in $\E$ if and only if it is of type $i$ in $\F$.
	\end{lemma}
	
	\begin{proof}
		For $x\in\{|\alpha|-m+1,|\alpha|\}$, we have $x\in A_0$ for all $\E$ with $\E\sim \F$. Now suppose $x\notin\{|\alpha|-m+1,|\alpha|\}$. Let $\lambda=\lambda(2|\alpha|-x-m+1,x,m-1)$ and set $d=|[h_\lambda]s_\F|$. It is {\emma straightforward} to see that $\{(\CMcal{S};0)\succeq(\F):\lambda(\CMcal{S})=\lambda\}=\varnothing$. Consequently, $d=|\{(\gamma;1)\succeq(\F):\lambda(\gamma)=\lambda(2|\alpha|-x-m+1,x)\}|$ which is invariant under replacing $\F$ by any $\E\sim \F$. By Definition \ref{Def:type}, it follows that $x\in A_d$ for every $\E$ with $\E\sim \F$.
	\end{proof}
	
	We continue by categorizing equivalence classes of positive integers in the {\emma equal-size case} $|\alpha|=|\beta|$, as defined in Definition \ref{Def:equi2}. {\emma Recall that} $r=\gcd(|\alpha|-m+1,|\beta|-m+1)=|\alpha|-m+1$ and let
	\begin{align}\label{E:Ri}
		\CMcal{R}_i=\{i,|\alpha|-i,|\alpha|+i-m+1,2|\alpha|-m+1-i\},
	\end{align}
	By Lemma \ref{lem:symmetry}, for $i\ge m-1$, with $i=\min [i]$, one finds that
	\begin{align}\label{E:Ri2}
		[i]=\CMcal{R}_i
	\end{align}
	with $1\leq i\leq |\alpha|/2$ (since $i=\min [i]$). For $1\le i<m-1$, we have
	$$[i]=\CMcal{R}_{i}\cup \CMcal{R}_{m-1-i},$$
	with $1\leq i\leq (m-1)/2$. 
	\begin{example}
		Let $\F=132 \boxdot 2112$ be a $2$-regular thickened ribbon. Then $\alpha=132$ and $\beta=2112$ both of size $6$. By (\ref{E:Ri2})re22., we have
		$\R_1=[1]=\{1,5,6,10\}$.
	\end{example}
	\begin{definition}\label{Def:equi2}
		The set of equivalence classes $\{[i]:1\leq i\leq |\alpha|/2\}$ is partitioned into three {\emma families} as follows. For an equivalence class $[i]$ with $i=\min[i]$: 
		\begin{enumerate}
			\item if $[i]\subseteq A_1$ and $\bh_\F(i)=-\bh_\F(|\alpha|-i)$, then we say $[i]\in\mathcal{A}$;
			\item if $[i]\subseteq A_1$ and $\bh_\F(i)=\bh_\F(|\alpha|-i)$, then we say $[i]\in\mathcal{B}$;
			\item if $[i]\cap (A_0\cup A_2)\neq\varnothing$ and $[i]\cap A_1\neq\varnothing$, then we say $[i]\in\mathcal{B}$;
			\item if $[i]\subseteq A_0\cup A_2$, then we say $[i]\in\mathcal{C}$.
		\end{enumerate}
	\end{definition}
	\begin{example}\label{Example:9}
		The types and the function $\bh_{\F}$ for $\F=132\boxdot 2112$ are given as below.
		\begin{center}
			\begin{tabular}{c|ccccccccccccccc}
				$x$ & $1$ & $2$ & $3$ & $4$ & $5$ & $6$ & $7$ & $8$ & $9$ & $10$ \\
				\hline \\[-1.5ex]
				type & $1$ & $1$ & $1$ & $2$ & $0$ & $0$ & $2$ & $1$ & $1$ & $1$ \\
				$\bh_{\F}(x)$ & $+$ & $-$ & $-$ & $+$ & $-$ & $-$ & $+$ & $+$ & $+$ & $-$ \\
			\end{tabular}
		\end{center}
	   \vspace{2mm}
		We have $[1]=\R_1=\{1,5,6,10\}$, $[2]=\R_2=\{2,4,7,9\}$, $[3]=\R_3=\{3,8\}$ and $[i]\in \mathcal{B}$ for $1\le i\le 3$ by Definition \ref{Def:equi2}.
	\end{example}
	
	\begin{lemma}\label{lem:class2}
		The class to which $[i]$ belongs is invariant under equivalence. Moreover, if $[i]\in\mathcal{A}$, then $m\leq \min[i] \leq |\alpha|-m$ and $[i]=\CMcal{R}_i$.
	\end{lemma}
	
	\begin{proof}
		Lemma \ref{lem:types2} ensures that the types of all $j\in[i]$ are uniform {\emma across} all members of the equivalence class of $\F$. If $[i]\cap (A_0\cup A_2)\neq\varnothing$, we distinguish the class of $[i]$ according to whether $[i]\cap A_1=\varnothing$ or not. 
		
		If $[i]\subseteq A_1$ with $i=\min[i]$, then $m\leq i\leq |\alpha|-m$, because $|\alpha|-x\in [x]\cap A_0$ for all $x\leq m-1$. Let $\lambda=\lambda(i,|\alpha|-i,i,|\alpha|-i)$ and set $d=|[h_\lambda]s_\F|$. Then,
		$$[i]\in\mathcal{A}\iff d=1,\quad \mbox{ and }\quad [i]\in\mathcal{B}\iff d=0.$$
		Since $d=|[h_\lambda]s_\E|$ {\emma is invariant} for any $\E\sim \F$, the class of $[i]$ when $[i]\subseteq A_1$ is determined by $d$, as desired.
	\end{proof}
	
	It should be noted that $\mathcal{A}=\mathcal{B}=\varnothing$ if and only if $i\in A_0\cup A_2$, which is equivalent to $\alpha=\beta^*$, or in other words, $\F=\F^*$. It follows that the equivalence class of $\F$ {\emma consists solely of} $\F$.
	
	If $\mathcal{A}\cup\mathcal{B}\ne \varnothing$, then $A_1\ne \varnothing$. We fix the sign $\bh_{\F}(k)$ for a chosen $k\in A_1$, and then study the properties of every $\E\in [\F]$ satisfying $\bh_{\E}(k)=\bh_{\F}(k)$.
	
	\begin{definition}\label{Def:can1}
		We say {\emma that} $\F=\alpha\boxdot\beta$ with $|\alpha|=|\beta|$ is canonical if one of the following mutually exclusive conditions holds:
		\begin{itemize}
			\item $\mathcal{B}\neq\varnothing$ and $\bh_\F(k)=+$, where $k$ is the smallest element such that $k\in A_1$ and $[k]\in\mathcal{B}$;
			\item $\mathcal{B}=\varnothing$, $\mathcal{A}\neq\varnothing$ and $\bh_\F(k)=+$, where $k$ is the smallest element such that $[k]\in\mathcal{A}$;
			\item $\mathcal{A}=\mathcal{B}=\varnothing$.
		\end{itemize}
	\end{definition}
	\begin{example}
		The thickened ribbon $\F=132\boxdot 2112$ in Example \ref{Example:9} is canonical by Definition \ref{Def:can1}: {\emma indeed}, $\mathcal{B}\ne \varnothing$, and $1$ is the smallest integer such that $1\in A_1$, $[1]\in \mathcal{B}$ and $\bh_{\F}(1)=+$.
	\end{example}
	
	Since exactly one of $\F,\F^*$ is canonical when $A_1\ne \varnothing$, and since $\bh_{\F}(x)=-\bh_{\F^*}(x)$ for all $x\in A_1$, it suffices to characterize the equivalence classes of canonical thickened ribbons. In the sequel, we {\emma therefore} assume that $\F$ is canonical.
	
	\subsection{Main result of Section \ref{S:5}}
	We state the main theorem {\emma (Theorem \ref{thm:a=b}) for} the equal-size case and explain how it completes Theorem \ref{T:1}. 
	
	The three {\emma follow-up subsections} establish Theorem \ref{thm:a=b}. More explicitly, part (I) shows that all undetermined elements must {\emma lie} in class $\mathcal{A}$, part (II) provides necessary lemmas {\emma analyzing} the relations and periodicity among the undetermined elements, and part (III) {\emma applies} the lemmas in (II) to {\emma exclude} every undetermined element.
	
	\begin{theorem}\label{thm:a=b}
		Suppose that $\F$ is canonical. Then the only canonical thickened ribbon in the equivalence class $[\F]$ is $\F$ itself. 
	\end{theorem}
	Theorem \ref{thm:a=b}, together with Lemma \ref{lem:canonical2}, completes the proof for the necessity of Theorem \ref{T:1} for $\F=\alpha\boxdot \beta$ when $|\alpha|=|\beta|$. 
	Before {\emma proving} Theorem \ref{thm:a=b}, we show how {\emma it implies} Theorem \ref{T:1} {\emma in the equal-size case}.\vspace{2mm}
	
	{\em Proof of Theorem \ref{T:1} for $\F=\alpha\boxdot \beta$ with $|\alpha|=|\beta|$}.
	If $\F=\alpha\boxdot\beta$ admits a nontrivial decomposition, say $\F=\CMcal{S}\circ_{m-1} \T$, then since $|\alpha|=|\beta|$, the ribbon $\CMcal{S}$ has two rows of equal size, {\emma hence} $\CMcal{S}=\CMcal{S}^*$. Taking $\E=\CMcal{S}\circ_{m-1} \T^*=\F^*$ or $\E=\CMcal{S}\circ_{m-1} \T=\F$, {\emma we obtain} $\E\sim \F$. If $\F$ has only {\emma the} trivial decomposition, then $\F=1\circ\F$; setting $\E=1\circ\F^*$ or $\E=1\circ\F$ again gives $\E\sim\F$.
	
	Conversely, suppose $\E\sim\F$, where $\F=\alpha\boxdot\beta$ and $|\alpha|=|\beta|$. It suffices to show that $\E\in\{\F,\F^*\}$. Assume $\F$ is canonical. If $\E$ is also canonical, then Lemma \ref{lem:canonical2} and Theorem \ref{thm:a=b} guarantee that $\E=\F$. Since exactly one of $\F$ or $\F^*$ is canonical (unless $\F=\F^*$), it follows that $\E\in\{\F,\F^*\}$.
	Equivalently, $[\F]=\{\F,\F^*\}$ contains $2^{\kappa}$ elements, where $\kappa=1$ if $\F\ne \F^*$, and $\kappa=0$ otherwise.
	
	\qed
	
	The {\emma remainder} of this section is {\emma devoted to} proving Theorem \ref{thm:a=b}. {\emma Equivalently}, we must show that $\bh_{\E}(x)=\bh_{\F}(x)$ for all $x\in A_1$ and every canonical $\E\in [\F]$. To this end, we introduce {\emma the notations of} determined and undetermined elements.
	\begin{definition}\label{Def:det}
		Suppose $\F=\alpha\boxdot\beta$ is canonical. For $1\leq x\leq |\alpha|+|\beta|-m$, we say that $x$ is {determined} if $\bh_\E(x)=\bh_\F(x)$ for all canonical $\E\in [\F]$. Otherwise, $x$ is called {undetermined}.
		
		Furthermore, a relation $\bh_{\F}(x)=\bh_{\F}(y)$ is said to be {determined} if $\bh_{\E}(x)=\bh_{\E}(y)$ holds for every canonical $\E\in [\F]$. Determined relations involving three or more terms are defined analogously. 
		
		%Further, a relation $\bh_{\F}(x)=\bh_{\F}(y)$ is said to be {\em determined} if $\bh_{\E}(x)=\bh_{\E}(y)$ holds for any canonical $\E,\F$ with $\E\sim \F$. %and the types of $x,y$ in $\F$ are the same as in $\E$.
	\end{definition}
	
	\subsection{Proof of Theorem \ref{thm:a=b}: Part I}
	Lemma \ref{lem:types2} guarantees that all elements of even type are determined. If $i$ is undetermined, then $i\in A_1$, {\emma and hence} $[i]\in\mathcal{A}\cup \mathcal{B}$. The {\emma following} lemma below {\emma rules out} the possibility that $[i]\in\mathcal{B}$ and asserts that all elements of {\emma such an equivalence class} are undetermined.
	
	\begin{lemma}\label{lem:sym2}
		If $i$ is undetermined, then $[i]\in\mathcal{A}$, and every $j\in [i]$ is also undetermined.
	\end{lemma}
	
	\begin{proof}
		It suffices to consider $i\le |\alpha|-1$, since $i$ is undetermined if and only if $2|\alpha|-m+1-i$ is undetermined.
		
		If $[i]\in\mathcal{A}$, then Lemma \ref{lem:class2} gives $[i]=\{i,|\alpha|-i,|\alpha|-m+1+i,2|\alpha|-m+1-i\}$ with $\bh_\F(i)=-\bh_\F(|\alpha|-i)=\bh_\F(|\alpha|-m+1+i)=-\bh_\F(2|\alpha|-m+1-i)$. Consequently, if any $j\in [i]$ is determined, then every $x\in[i]$ must be determined.
		
		Suppose {\emma now} that $i$ is undetermined and $[i]\in\mathcal{B}\cup\mathcal{C}$. {\emma Since all elements in $\mathcal{C}$ are of even type and hence determined by Lemma \ref{lem:types2}}, we must have $[i]\in\mathcal{B}$. We shall prove that all elements of $[i]$ are determined, {\emma which will yield} a contradiction. Since $i\le |\alpha|-1$, $[i]\in \mathcal{B}$ and $i\in A_1$, exactly one of the following holds:
		\begin{enumerate}
			\item $|\alpha|-i\in A_0\cup A_2$, or
			\item $|\alpha|-i\in A_1$ and $\bh_\F(i)=\bh_\F(|\alpha|-i)$.
		\end{enumerate}
		%In either case, if $i$ is determined, then all elements of $[i]$ are determined. 
		
		The condition $[i]\in \mathcal{B}$ {\emma implies} that $\mathcal{B}\ne\varnothing$. Let $k$ be the smallest integer such that $[k]\in\mathcal{B}$ and $k\in A_1$. Since $\F$ is canonical, $\bh_\F(k)=+$, and all elements of $\CMcal{R}_k$ are determined. If $i\in\CMcal{R}_k$, then $i$ is determined and we are done. Otherwise, let $a$ be the number of $+$ {\emma signs} in the multiset $\{\bh_\F(k),\bh_\F(|\alpha|-k)\}$, and let $b$ be the number of $+$ signs in the multiset $\{\bh_\F(|\alpha|-m+1+k),\bh_\F(2|\alpha|-m+1-k)\}$. It is {\emma straightforward} to check that $a$ and $b$ are determined, and that $a\neq b$.
		
		Let $\lambda=\lambda(k,|\alpha|-k,i,|\alpha|-i)$ and let $d=|[h_\lambda]s_\F|$. We {\emma now} distinguish cases {\emma based on} whether $m-1$ is a part of $\lambda$.
		\begin{enumerate}
			\item If $k\ne m-1$ and $i\ne m-1$, then 
			$\{(\gamma;1)\succeq(\F):\lambda(\gamma,m-1)=\lambda\}=\varnothing$, thus $d=|\{(\CMcal{S};0)\succ(\F):\lambda(\CMcal{S})=\lambda\}|$. {\emma The} possible cases are {\emma summarized in the following table}:
			\begin{center}
				\begin{tabular}{c|cc}
					Case (1)& $\bh_\F(i)$ & $d$ \\
					\hline \\[-1.5ex]
					$|\alpha|-i\in A_0$ & $+$ & $b$ \\
					& $-$ & $a$ \\
					\hline \\[-1.5ex]
					$|\alpha|-i\in A_1$ & $+$ & $2b$ \\
					& $-$ & $2a$ \\
					\hline \\[-1.5ex]
					$|\alpha|-i\in A_2$ & $+$ & $a+2b$ \\
					& $-$ & $2a+b$ \\
				\end{tabular}
			\end{center}
			\item If $k=m-1$ and $i\ne m-1$, then $\lambda=\lambda(|\alpha|-m+1,i,|\alpha|-i,m-1)$ and $(a,b)=(1,0)$. Here the set $\{(\gamma;1)\succeq(\F):\lambda(\gamma,m-1)=\lambda\}$ may be nonempty. The possible cases are:
			\begin{center}
				\begin{tabular}{c|cc}
				 Case (2) & $\bh_\F(i)$ & $d$ \\
					\hline \\[-1.5ex]
					$|\alpha|-i\in A_0\cup A_2$ & $+$ & $0$ \\
					& $-$ & $1$ \\
					\hline \\[-1.5ex]
					$|\alpha|-i\in A_1$ & $+$ & $0$ \\
					& $-$ & $2$ 
				\end{tabular}
			\end{center}
			\item  If $k\ne m-1$ and $i=m-1$,  then $k<m-1$ and $\lambda=\lambda(|\alpha|-m+1,k,|\alpha|-k,m-1)$. Since $|\alpha|-k>|\alpha|-m+1$, we have $|\alpha|-k\in A_0$. Hence the only possible coarsened composition corresponding to $\lambda$ is $(k,|\alpha|-k\boxdot |\alpha|-m+1,m-1)$. 
			Consequently, 
			$$\bh_{\F}(m-1)=+\iff d=0,\quad \mbox{ and }\quad \bh_{\F}(m-1)=-\iff d=1.$$
		\end{enumerate} 
		Since $a\ne b$ and $d=|[h_{\lambda}]s_{\E}|$ for every canonical $\E \in [\F]$, the tables {\emma above} show that both $\bh_\F(i)$ and $\bh_\F(|\alpha|-i)$ are determined from $d$. That is,
		the elements $i$ and $|\alpha|-i$ are determined in each of the cases (1)--(3), {\emma contradicting the assumption that $i$ is undetermined}.
		
		{\emma Therefore, $[i]\in\mathcal{A}$, and all elements of $[i]$ are undetermined, as claimed.}
		
		%As is shown in the tables above, $\bh_\F(i)$ can be determined from $d$. Since $d$ is invariant under equivalence, $i$ is determined and all elements of $[i]$ are determined, which contradicts the assumption that $i$ is undetermined. It follows that $[i]\in\A$ and this finishes the proof.
	\end{proof}
	
	\subsection{Proof of Theorem \ref{thm:a=b}: Part II}\label{ss:lemma1}
	We proceed with a sequence of lemmas {\emma that} specify {\emma the} relations between undetermined elements of $\F$. The essential ideas {\emma are drawn} from \cite{BTvW}, though several technical details must be {\emma adapted to the} thickened ribbon setting.
	
	\begin{lemma}\label{lem:BTvW4}%[Compare with Lemma 5.2 of \cite{BTvW}]
		Suppose $x,y,x+y\in A_1$. Then whether $$\bh_{\F}(x)=\bh_{\F}(y)=-\bh_{\F}(x+y)$$ holds can be determined.
	\end{lemma}
	
	\begin{proof}
		Let $\lambda=\lambda(x,y,2|\alpha|-x-y-m+1,m-1)$ and set $d=|[h_\lambda]s_\F|$. Note that $x,y,x+y\in A_1$ implies $\{x,y,x+y\}\cap\{|\alpha|-m+1,|\alpha|\}=\varnothing$, which {\emma in turn yields} $\{(\CMcal{S};0)\succeq(\CMcal{\F}):\lambda(\CMcal{S})=\lambda\}=\varnothing$. Therefore, by (\ref{eq:JT}), the multiset $\{(\gamma;1)\succeq(\F):\lambda(\gamma,m-1)=\lambda\}$ contains exactly $d$ elements. Moreover, $\bh_{\F}(x)=\bh_{\F}(y)=-\bh_{\F}(x+y)$ if and only if $d=0$.
	\end{proof}
	
	\begin{lemma}\label{lem:BTvW5}%[Compare with Lemma 5.3 of \cite{BTvW}]
		Suppose exactly two of $x,y,x+y$ lie in $A_1$, say $i,j\in A_1$. Then whether $$\bh_{\F}(i)=\bh_{\F}(j)$$ holds can be determined, unless $m-1\in A_1$ and either $x+y=|\alpha|-m+1$ or $|\alpha|\in\{x,y\}$.
	\end{lemma}
	
	\begin{proof}
		Let $\lambda=\lambda(x,y,2|\alpha|-x-y-m+1,m-1)$ and set $d=|[h_\lambda]s_\F|$. Note that $|\alpha|,|\alpha|-m+1\in A_0$. It is sufficient to consider the following cases (1)--(2) and their subcases:
		\begin{enumerate}
			\item $x,y\in A_1$:
			\begin{enumerate}
				\item $x+y=|\alpha|$;
				\item $x+y\notin\{|\alpha|-m+1,|\alpha|\}$;
				\item $m-1\in A_0\cup A_2$ and $x+y=|\alpha|-m+1$. 
			\end{enumerate}
			\item $x,x+y\in A_1$:
			\begin{enumerate}
				\item $y=|\alpha|-m+1$;
				\item $y\notin\{|\alpha|-m+1,|\alpha|\}$;
				\item $m-1\in A_0\cup A_2$ and $y=|\alpha|$.
			\end{enumerate}
		\end{enumerate}
	   {\emma Since the exceptional cases where $m-1\in A_1$ and either $x+y=|\alpha|-m+1$ or $|\alpha|\in\{x,y\}$ are not considered, case (1) excludes the subcase $m-1\in A_1$ and $x+y=|\alpha|-m+1$, and case (2) excludes the subcase $m-1\in A_1$ and $y=|\alpha|$. }
	
		%For case (1), only the case when $m-1\in A_1$ and $x+y=|\alpha|-m+1$ is not taken into account. For case (2), we do not consider the omitted subcase $m-1\in A_1$ and $y=|\alpha|$.
		
		Case (1)--(a): $x,y\in A_1$ and $x+y=|\alpha|$. Lemma \ref{lem:class2} and Definition \ref{Def:equi2} imply that $\bh_{\F}(x)=\bh_{\F}(y)$ (or $-\bh_\F(y)$) if and only if $[x]\in\mathcal{B}$ (or $\mathcal{A}$, respectively), and we are done.

		Case (2)--(a): $x,x+y\in A_1$ and $y=|\alpha|-m+1$. Lemma \ref{lem:class2} and Definition \ref{Def:equi2} imply $\bh_{\F}(x)=\bh_{\F}(x+y)$ (or $-\bh_\F(x+y)$) if and only if $[x]\in\mathcal{A}$  (or $\mathcal{B}$, respectively), and we are done.
		
		Case (1)--(b): Suppose $x,y\in A_1$ and $x+y\notin\{|\alpha|-m+1,|\alpha|\}$. Then $\{(\CMcal{S};0)\succeq(\CMcal{\F}):\lambda(\CMcal{S})=\lambda\}=\varnothing$. Hence, by (\ref{eq:JT}), the multiset $\{(\gamma;1)\succeq(\F):\lambda(\gamma,m-1)=\lambda\}$ has exactly $d$ elements. Furthermore, {\emma $x+y\in A_0\cup A_2$ and whether or not $\bh_\F(x)=\bh_\F(y)$ is determined.}
		%\begin{itemize}
		%	\item if $x+y\in A_0$, then $\bh_{\F}(x)=\bh_{\F}(y)$ (or $-\bh_{\F}(y)$) if and only if $d=0$ (or $1$, respectively);
		%	\item if $x+y\in A_2$, then $\bh_{\F}(x)=\bh_{\F}(y)$ (or $-\bh_{\F}(y)$) if and only if $d=2$ (or $3$, respectively).
		%\end{itemize}
	   {\emma 
	   	\begin{center}
	   	\begin{tabular}{c|cc}
	   		Case (1)--(b) & $\bh_\F(x)/\bh_\F(y)$ & $d$ \\
	   		\hline \\[-1.5ex]
	   		$x+y\in A_0$ & $1$ & $0$ \\
	   		& $-1$ & $1$ \\
	   		\hline \\[-1.5ex]
	   		$x+y\in A_2$ & $1$ & $2$ \\
	   		& $-1$ & $3$ 
	   	\end{tabular}
	   \end{center}}
		
		Case (2)--(b): Suppose $x,x+y\in A_1$ and $y\notin\{|\alpha|-m+1,|\alpha|\}$. Then $\{(\CMcal{S};0)\succeq(\CMcal{\F}):\lambda(\CMcal{S})=\lambda\}=\varnothing$. Hence, by (\ref{eq:JT}), the multiset $\{(\gamma;1)\succeq(\F):\lambda(\gamma,m-1)=\lambda\}$ has exactly $d$ elements. Furthermore,  {\emma $y\in A_0\cup A_2$ and whether or not $\bh_\F(x)=\bh_\F(x+y)$ is determined.}
		%\begin{itemize}
		%	\item if $y\in A_0$, then $\bh_{\F}(x)=\bh_{\F}(x+y)$ (or $-\bh_{\F}(x+y)$) if and only if $d=1$ (or $0$, respectively);
		%	\item if $y\in A_2$, then $\bh_{\F}(x)=\bh_{\F}(x+y)$ (or $-\bh_{\F}(x+y)$) if and only if $d=3$ (or $2$, respectively).
		%\end{itemize}
	
	    {\emma 
	    	\begin{center}
	    		\begin{tabular}{c|cc}
	    		Case (2)--(b)	& $\bh_\F(x)/\bh_\F(x+y)$ & $d$ \\
	    			\hline \\[-1.5ex]
	    			$y\in A_0$ & $1$ & $1$ \\
	    			& $-1$ & $0$ \\
	    			\hline \\[-1.5ex]
	    			$y\in A_2$ & $1$ & $3$ \\
	    			& $-1$ & $2$ 
	    		\end{tabular}
	    \end{center}}
		
		Case (1)--(c): $x,y\in A_1$, $m-1\in A_0\cup A_2$ and $x+y=|\alpha|-m+1$. 
		
		If $x,y$ are determined, then we are done. Otherwise, without loss of generality, assume $x$ is undetermined. Then, by Lemma \ref{lem:sym2}, we have $[x]\in\A$. Lemma \ref{lem:class2} then gives $m\le \min[x]\le |\alpha|-m$, which implies $m-1<|\alpha|-m+1$. {\emma This allows} us to apply Case (2)--(b) to the triple $(m-1,y,|\alpha|-x)$, {\emma which determines}
		whether $\bh_\F(|\alpha|-x)=\bh_\F(y)$ holds. Since $[x]\in\A$, we have $\bh_\F(|\alpha|-x)=-\bh_\F(x)$, {\emma and the desired statement follows}.
		
		Case (2)--(c): $x,x+y\in A_1$, $m-1\in A_0\cup A_2$ and $y=|\alpha|$. 
		
		If both $x$ and $x+y$ are determined, we are done. Otherwise, {\emma suppose} $x$ is undetermined. Then by Lemma \ref{lem:sym2}, we have $[x]\in\A$ and $x+|\alpha|-m+1\in A_1$. As in Case (1)--(c), we have
		$m-1<|\alpha|-m+1$ and the triple $(m-1,x+|\alpha|-m+1,x+y)$ satisfies the conditions of Case (2)--(b). Hence, it is determined whether $\bh_\F(x+|\alpha|-m+1)=\bh_\F(x+y)$. Since $[x]\in\A$, we have $\bh_\F(x)=\bh_\F(x+|\alpha|-m+1)$, {\emma and the desired conclusion follows.}
		
		If {\emma instead} $x+y$ is undetermined, then by Lemma \ref{lem:sym2}, we have $[x+y]\in\A$ and $x+m-1\in A_1$. The same argument {\emma applies to} the triple $(m-1,x,x+m-1)$, so we omit the details.
	\end{proof}
	
	We shall discuss the periodicity of {\emma the} undetermined elements of $\F$ {\emma using} the approach from \cite{BTvW}. We first introduce a sequence of functions $g_t$, and then examine the antisymmetry and periodicity of $g_t(x)$ with the help of Lemma \ref{lem:BTvW3}. Recall that $r=\gcd(|\alpha|-m+1,|\beta|-m+1)=|\alpha|-m+1$.
	
	\begin{definition}\label{Def:gt}
		Let $M$ be the set of all undetermined elements, and let $m_1=\min(M)$. Suppose $\{i\in M:1\leq i\leq m_1+r \}=\{m_1,m_2,\dots,m_z\}$ with $m_1<m_2<\cdots<m_z$. For $1\leq t\leq z$, let $$r_t=\gcd(m_1,m_2,\dots,m_t).$$ 
		For $2\leq t\leq z$, we define a function $g_t:[1,m_t-1]\backslash\{|\alpha|-m+1,|\alpha|\}\to\{0,1,-1,*\}$ as follows:
		$$g_t(x)=\begin{cases}
			* & \text{ if } r_{t-1}|x,\\
			0 & \text{ if }r_{t-1}\nmid x \text{ and } x\in A_0\cup A_2,\\
			1 & \text{ if }r_{t-1}\nmid x \text{ and } x\in A_1 \text{ and }\bh_\F(x)=+,\\
			-1 & \text{ if }r_{t-1}\nmid x \text{ and } x\in A_1 \text{ and }\bh_\F(x)=-.
		\end{cases}$$
		For $x\in \{|\alpha|-m+1,|\alpha|\}$, we set $g_z(x)=0$.
	\end{definition}
	
	The functions $g_t$ retain {\emma precisely} the information needed for the periodicity argument. A value {\emma of} $1$ or $-1$ records the sign of a determined type-one position; $0$ records an even-type position; and the symbol $*$ masks multiples of $r_{t-1}$, where unresolved behavior may remain. As $t$ increases, {\emma the value of $r_t$} decreases. {\emma The periodicity and antisymmetry of $g_t$} force the masked set to shrink until a hypothetical undetermined position would have to be of even type, which is impossible.
	
	\begin{lemma}\label{lem:BTvW3}[Lemma 5.5 of \cite{BTvW}]
		Let $f$ be a function on {\emma $[1,d-1]$} which takes values $0,1,-1,*$, and suppose that there is some $c<d$, such that for all $x$ for which both sides are well-defined
		\begin{enumerate}
			\item $f(x)=-f(d-x)$,
			\item $f(x)=f(c+x)$,
			\item $f(x)=-f(c-x)$
		\end{enumerate}
		except that if either side equals $*$, the equation is not required to hold. Further, we require that the points where $f$ takes the value $*$ are either exactly the multiples of $c$ less than $d$, or else no point at all. Let $r_0=\gcd(c,d)$. Then on multiples of $r_0$, $f$ takes on only the values $0$ and $*$. On nonmultiples of $r_0$, $f$ is $r_0$-periodic, and $f$ is antisymmetric on {\emma $[1,r_0-1]$.}
	\end{lemma}
	
	\begin{lemma}\label{lem:D2}
		The function $g_t$ satisfies the assumptions of Lemma \ref{lem:BTvW3} {\emma in} the following two cases:
		\begin{enumerate}
			\item $m-1\in A_0\cup A_2$, and for each $2\leq t\leq z$, take $d=m_t$ and $c=r_{t-1}$. 
			\item $m-1\in A_1$, and for each $2\leq t\leq z-1$, take $d=m_t$ and $c=r_{t-1}$. 
		\end{enumerate}
		Consequently, if we set $s=z$ for (1) and $s=z-1$ for (2), then
		$g_s$ is $r_s$-periodic on {\emma $[1,m_{s}-1]$} except {\emma at} multiples of $r_s$, and $g_s$ is antisymmetric on {\emma $[1,r_s-1]$.} Moreover, if $x<m_1$ and $r_s\mid x$, then $x\in A_0\cup A_2$.
	\end{lemma}
	
	\begin{proof}
		Note that if $m-1\in A_1$, then $d\leq m_{z-1}< |\alpha|-m+1$, which allows us to apply Lemmas  \ref{lem:BTvW4} and \ref{lem:BTvW5}. Our proof consists of three steps. In the first two steps, we show that $g_t$ {\emma satisfies} properties (1) and (2) of Lemma \ref{lem:BTvW3} directly. In the last step we prove that property (3) also holds for $g_t$ by induction on $t$. 
		
		Fix an arbitrary $1\leq x\leq m_t-1$ such that $g_t(x)\neq *$. 
		
		(1) Suppose $g_t(x)=0$; that is $x\in A_0\cup A_2$. From the triple $(x,d-x,d)$, Lemma \ref{lem:BTvW5} asserts that either $d-x\in A_0\cup A_2$ or $d-x$ is undetermined i.e. $g_t(d-x)=0$ or $*$. 
		
		Suppose $|g_t(x)|=1$, thus $x\in A_1$ and $x$ is determined. From the triple $(x,d-x,d)$, Lemmas \ref{lem:BTvW4} and \ref{lem:BTvW5} assert that either $d-x\in A_1$ with $\bh_\F(d-x)=-\bh_\F(x)$, or $d-x$ is undetermined i.e. $g_t(d-x)=-g_t(x)$ or $*$. 
		
		This shows that $g_t$ satisfies condition (1) of Lemma \ref{lem:BTvW3}.
		
		(2) Suppose $g_t(x)=0$, thus $x\in A_0\cup A_2$. For any $1\leq s< t$ such that $x+m_s< m_t$, since $m_s$ is undetermined, Lemma \ref{lem:BTvW5} {\emma applied to} the triple $(x,m_s,x+m_s)$ gives that either $x+m_s\in A_0\cup A_2$ or $x+m_s$ is undetermined. But if $x+m_s$ is undetermined, then $r_{t-1}|x$ and $g_t(x)=*$, contradicting to $g_t(x)=0$. {\emma Hence} $x+m_s\in A_0\cup A_2$ i.e. $g_t(x+m_s)=g_t(x)=0$.
		
		Suppose $|g_t(x)|=1$, thus $x\in A_1$ and $x$ is determined. For any $1\leq s< t$ such that $x+m_s<m_t$, since $m_s$ is undetermined, Lemmas \ref{lem:BTvW4} and \ref{lem:BTvW5} {\emma applied to} the triple $(x,m_s,x+m_s)$ assert that either $x+m_s\in A_1$ with $\bh_\F(x+m_s)=\bh_\F(x)$, or $x+m_s$ is undetermined. If {\emma the latter holds}, then $r_{t-1}|x$ and $g_t(x)=*$, contradicting to $|g_t(x)|=1$. {\emma Therefore}, $x+m_s\in A_1$ and $g_t(x+m_s)=g_t(x)$.
		
		It follows that for all $1\leq s<t$, $g_t$ is $m_s$-periodic except {\emma at} multiples of $r_{t-1}$. {\emma By} the Euclidean algorithm, this implies that $g_t$ is $r_{t-1}$-periodic except at multiples of $r_{t-1}$. This verifies condition (2) of Lemma \ref{lem:BTvW3}.
		
		(3) We {\emma now} prove {\emma by induction on $t$} that equation (3) of Lemma \ref{lem:BTvW3} holds for $f=g_t$. {\emma For} $t=2$, we have $r_{t-1}=m_1$, and condition (3) reduces to the case (1) above for $t=1$; hence the base case is true.
		
		{\emma Assume} condition (3) holds for some $t$. Then $g_t$ satisfies the assumptions of Lemma \ref{lem:BTvW3}, so by that lemma $g_t$ is antisymmetric on $[1,r_t-1]$, where $r_t=\gcd(r_{t-1},m_t)$.  Consequently, 
		\begin{align*}
			g_{t+1}(x)=g_t(x)=-g_t(r_t-x)=-g_{t+1}(r_t-x)
		\end{align*}
		for $1\leq x\leq r_t-1$. Thus equation (3) also holds for $t+1$,  {\emma completing the induction}.
		
		We have {\emma shown} that $g_t$ satisfies the conditions of Lemma \ref{lem:BTvW3}. Now suppose $x<m_1$ and $r_s\mid x$. Let $t$ be the smallest integer such that $r_t\mid x$. Then $2\leq t\leq s$, and Lemma \ref{lem:BTvW3} implies that $g_t(x)=0$, since $r_{t-1}\nmid x$ while $r_t\mid x$.
	\end{proof}
	
	\subsection{Proof of Theorem \ref{thm:a=b}: Part III}
	%\subsection{Proof of Theorem \ref{T:1} for the case $|\alpha|=|\beta|$}
	This part uses {\emma results} from part II to show {\emma that} all elements are determined. We prove Theorem \ref{thm:a=b} by distinguishing the cases $m-1\in A_1$ {\emma and} $m-1\in A_0\cup A_2$, which are treated in Propositions \ref{lem:final2} and \ref{lem:final3}, respectively. 
	
	%The key ingredient is to show that all elements of a canonical $\F$ are determined in the follow-up two lemmas, which requires a careful case analysis.
	\begin{proposition}\label{lem:final2}
		Theorem \ref{thm:a=b} holds for $m-1\in A_1$.
	\end{proposition}
	
	\begin{proof}
		If $m-1\in A_1$, then $[m-1]\in\mathcal{B}$ and $m-1$ is determined. 
		
		Suppose that there exists an undetermined element of $\F$, let $x$ be the smallest undetermined element, then by Lemmas \ref{lem:symmetry}, \ref{lem:class2} and \ref{lem:sym2} one sees that $m\le x<|\alpha|-x$, that is, $$m\le x<\frac{|\alpha|}{2},$$ which {\emma justifies} the use of Lemmas \ref{lem:BTvW4} and \ref{lem:BTvW5}. 
		
		Lemma \ref{lem:D2} proves that $g_{z-1}(u)$ is $r_{z-1}$-periodic on $[1,m_{z-1}-1]$ except at multiples of $r_{z-1}$, where $r_{z-1}\vert \gcd(x,|\alpha|-x)$, thus $r_{z-1}\le x<|\alpha|/2$. 
		
		Since $m-1\in A_1$ is determined and $m-1<x=m_1$, Lemma \ref{lem:D2} implies $r_{z-1}\nmid (m-1)$. {\emma Write} 
		\begin{align*}
			m-1\equiv a\mod r_{z-1}
		\end{align*}
		with $a\ne 0$. Then, by the antisymmetry and periodicity of $g_{z-1}$, 
		\begin{align}\label{E:gz}
			g_{z-1}(a)=g_{z-1}(m-1)=-g_{z-1}(r_{z-1}-a)=\pm 1.
		\end{align}
		We {\emma now aim to prove} that $g_{z-1}$ is $(m-1)$-periodic except at multiples of $r_{z-1}$, i.e.,
		\begin{align}\label{E:gpe}
			g_{z-1}(u)=g_{z-1}(u+m-1)
		\end{align}
		for $1\le u\le m_{z-1}-m$ such that $r_{z-1}\nmid u$ and $r_{z-1}\nmid (u+m-1)$. 
		
		Suppose that (\ref{E:gpe}) is proved. Let $\sigma=\gcd(r_{z-1},m-1)$, then $g_{z-1}(u)$ is also $\sigma$-periodic except at multiples of $r_{z-1}$, with $1\le \sigma<r_{z-1}\ne 1$ (since $r_{z-1}\nmid (m-1)$). {\emma Consequently}, $g_{z-1}(a)=g_{z-1}(r_{z-1}-a)$ {\emma because} both $a$ and $r_{z-1}-a$ are multiples of $\sigma$ but not of $r_{z-1}$. {\emma This contradicts (\ref{E:gz}). Therefore all elements are determined, completing the proof.}
		
		{\emma It remains to prove (\ref{E:gpe})}. Fix $1\le u\le m_{z-1}-m$ with $r_{z-1}\nmid u$ and $r_{z-1}\nmid (u+m-1)$. Since $m_{z-1}$ is undetermined, we have $m_{z-1}\leq |\alpha|-m$, hence $u+m-1<|\alpha|-m$. 
		
		We claim that it suffices to prove (\ref{E:gpe}) for $u\in A_1$. {\emma We prove it by showing that
		\begin{align}\label{E:u02}
			u\in A_0\cup A_2\,\Longrightarrow u+m-1\in A_0\cup A_2.
		\end{align}
        That is, $g_{z-1}(u)=0\,\Longrightarrow g_{z-1}(u+m-1)=0$.} If $u\in A_0\cup A_2$, then Lemma \ref{lem:BTvW5} {\emma applied to} the triple $(u,m-1,u+m-1)$, asserts that either $u+m-1\in A_0\cup A_2$ or $u+m-1\in A_1$ is determined. {\emma Assuming} (\ref{E:gpe}) {\emma holds} for all $u\in A_1$, the latter case would imply $|g_{z-1}(u)|=1$, {\emma contradicting} $g_{z-1}(u)=0$. {\emma Hence} (\ref{E:u02}) is proved, as desired.
		
		We now prove (\ref{E:gpe}) for $u\in A_1$ by repeatedly applying Lemmas \ref{lem:BTvW4} and \ref{lem:BTvW5}. {\emma Equivalently, we establish that}
		\begin{align}\label{E:u1}
			u\in A_1\,\Longrightarrow u+m-1\in A_1 \,\mbox{ is determined and }\,\bh_{\F}(u+m-1)=\bh_{\F}(u).
	\end{align}
		The conditions $r_{z-1}\nmid u$ and $r_{z-1}\nmid (u+m-1)$ imply that both $u$ and $u+m-1$ are determined. Since $r-u\in [u+m-1]$ is determined, consider the triples 
		\begin{align*}
			(x,u,x+u)\quad \mbox{ and }\quad(r-u,x+u,r+x).
		\end{align*}
		Lemmas \ref{lem:BTvW4} and \ref{lem:BTvW5}, together with $r_{z-1}\nmid (x+u)$, imply that $x+u\in A_1$ is determined and $\bh_{\F}(x+u)=\bh_{\F}(u)$. Furthermore, $r-u\in A_1$ is determined and satisfies $\bh_{\F}(r-u)=-\bh_{\F}(x+u)$. 
		
		Recall that $r_{z-1}\mid  \gcd(x,|\alpha|-x) \mid |\alpha|$. Since $r_{z-1}\nmid (u+m-1)$, we also have
		\begin{align*}
			r_{z-1} \nmid (|\alpha|-r+u), \quad r_{z-1} \nmid (r-u)\quad \mbox{ and }\quad r_{z-1}\nmid (r-u-x).
		\end{align*}
		If $x+u<r$, then from the triple $(r-u-x,x,r-u)$ and {\emma the fact} $r_{z-1}\nmid (r-u-x)$, we {\emma obtain} that $r-u-x\in A_1$ is determined and $\bh_{\F}(r-u-x)=\bh_{\F}(r-u)=-\bh_{\F}(u)$. {\emma Applying} the triple $(u+m-1,r-u-x,|\alpha|-x)$, {\emma then yields} that $u+m-1\in A_1$ is determined with $$\bh_{\F}(u+m-1)=-\bh_{\F}(r-u-x)=\bh_{\F}(u),$$ as desired. If $x+u>r$, then from the triple $(x+u-r, r-u, x)$ and $r_{z-1}\nmid (x+u-r)$, we get $x+u-r\in A_1$ determined and $\bh_{\F}(x+u-r)=-\bh_{\F}(r-u)=\bh_{\F}(u)$. Finally, applying the triple $(x+u-r, |\alpha|-x,u+m-1)$ with $r_{z-1}\nmid (u+m-1)$ {\emma gives} that $u+m-1\in A_1$ is determined and $$\bh_{\F}(u+m-1)=\bh_{\F}(x+u-r)=\bh_{\F}(u).$$ This completes the proof of (\ref{E:u1}) and (\ref{E:gpe}), as wished.
	\end{proof}
	
	\begin{proposition}\label{lem:final3}
		Theorem \ref{thm:a=b} holds for $m-1\in A_0\cup A_2$.
	\end{proposition}
	
	\begin{proof}
		The statement holds trivially for canonical $\F$ satisfying $\F=\F^*$, so we focus on canonical $\F$ such that $\F\ne \F^*$, i.e., $A_1\ne\varnothing$.
		We claim that 
		\begin{align}\label{E:mindet}
		x =\min(A_1)\,\mbox{ is determined}.
	\end{align}
		Suppose {\emma not}. Then we must have $\mathcal{B}\neq\varnothing$. {\emma Indeed, if}  $\mathcal{B}=\varnothing$, then by Lemma \ref{lem:sym2}, $[x]\in \A\ne \varnothing$, and since $\F$ is canonical, $x=\min(A_1)$ is determined, {\emma contradicting the assumption that} $x$ is undetermined. {\emma Hence} $\mathcal{B}\ne \varnothing$. 
		
		Let $k$ be the smallest integer such that $k\in A_1$ and $[k]\in\mathcal{B}$. Since $\F$ is canonical, $\bh_{\F}(k)=+$. Moreover, $m\leq x<k<r$, $k$ is determined, and one of the following holds:
		\begin{enumerate}
			\item $k+r\in A_0\cup A_2$,
			\item $k+r\in A_1$ and $\bh_\F(k)=-\bh_\F(k+r)$.
		\end{enumerate}
		
		{\emma By} Lemma \ref{lem:sym2}, $|\alpha|-x$ is also undetermined. From the triple $(|\alpha|-x,k,|\alpha|-x+k)$, Lemmas \ref{lem:BTvW4} and \ref{lem:BTvW5} imply that $|\alpha|-x+k\in A_1$. We {\emma now consider} whether $|\alpha|-x+k$ is determined or not.
		
		If $|\alpha|-x+k\in A_1$ is undetermined, then $k-x+m-1\in [|\alpha|-x+k]$ is undetermined by Lemma \ref{lem:sym2}. {\emma Applying} Lemmas \ref{lem:BTvW4} and \ref{lem:BTvW5} to the triples
		\begin{align*}
			(x-m+1,k-x+m-1,k)\quad\mbox{ and }\quad (x-m+1,|\alpha|-x+k,k+r),
		\end{align*}
		 together with properties (1) or (2), yields that $x-m+1$ is undetermined, contradicting the minimality of $x$.
		
		{\emma Now} suppose $|\alpha|-x+k\in A_1$ is determined. Since $m-1\in A_0\cup A_2$, Lemma \ref{lem:BTvW5}, applied to the triple $(m-1,k-x+r,|\alpha|-x+k)$,  implies that $k-x+r$ is determined, hence $k-x$ is determined {\emma regardless of whether} $k-x\in A_0\cup A_2$ or $k-x\in A_1$. {\emma However}, from the triples 
		\begin{align*}
			(k-x,x,k) \quad \mbox{ and }\quad (k-x,x+r,k+r),
		\end{align*}
		Lemmas \ref{lem:BTvW4} and \ref{lem:BTvW5}, combined with properties (1) or (2), imply that $k-x$ must be undetermined, a contradiction. Therefore, {\emma the claim (\ref{E:mindet}) is proved.}
		
		We {\emma now} follow the notation from Definition \ref{Def:gt} and Lemma \ref{lem:D2}. Suppose there exists an undetermined element. %i.e. $M\neq\varnothing$. 
		Since $x=\min(A_1)$, we have $x<m_1$. Note that $$r_z\mid \gcd(m_1,|\alpha|-m_1,r+m_1)\mid (m-1),$$ in particular, $r_z\leq m-1$. Since $g_z(u)=0$ for $|\alpha|-m+1\leq u\leq |\alpha|$ (i.e., $m$ consecutive zeros), and $g_z$ is $r_z$-periodic on $[1,m_1+r-1]$, we obtain $g_z(u)\neq 0$ only if $r_z\mid u$. {\emma In particular}, $r_z\mid x$ {\emma since} $x\in A_1$. But Lemma \ref{lem:D2} implies that $x$ is of even type, a contradiction. {\emma Hence} there can be no undetermined element. This completes the proof.
		
		%Let $t$ be the smallest integer such that $r_t|x$. Since $x<m_1$ and $r_{z}\vert x$, we get $2\leq t\leq z$. Take $d=m_t$ and $c=r_{t-1}$, Lemma \ref{lem:D2} implies that $g_t$ satisfies the assumptions of Lemma \ref{lem:BTvW3} and hence $g_t(x)=0$ as a result of $r_{t-1}\nmid x$ and $r_t\mid x$, i.e., $x\in A_0\cup A_2$, which is against $x\in A_1$. We therefore conclude that all elements are determined.
	\end{proof}
	
	%Now we are able to establish Theorem \ref{T:1} for thickened ribbons $\F=\alpha\boxdot\beta$ with $|\alpha|=|\beta|$.
	%\begin{lemma}\label{L:dem2}
	%If $\F=\alpha\boxdot\beta$ with $|\alpha|=|\beta|$ admits a nontrivial decomposition, that is $\F=\CMcal{S}\circ_\W \T$ with $\CMcal{S}\neq \W=1$ and $\T\neq 1$, then $\CMcal{S}$ is a ribbon with two rows of the same size, i.e., $\CMcal{S}=\CMcal{S}^*$.
	%\end{lemma}
	%\begin{proof}
	%The proof follows exactly the same line as the one for Lemma \ref{lem:period}, for which we omit the details.
	%\end{proof}

	\section{When two ribbons are of unequal sizes}\label{S:4}
	Throughout Section~6, $\F=\alpha\boxdot\beta$ denotes an $m$-regular thickened ribbon with exactly one $2\times m$ block and $|\alpha|\ne|\beta|$. {\emma Following} Definition \ref{Def:a<b} we orient the two ribbon pieces so that $|\alpha|<|\beta|$, and {\emma call such a thickened ribbon $\F$ \emph{canonical}}. This canonical hypothesis remains in {\emma effect} unless explicitly stated otherwise.\vspace{2mm}
	
	The proof is divided into the following steps:
	\begin{itemize}
		\item The sufficient condition is established in the same way as in the case $|\alpha|= |\beta|$.
		\item For the necessary condition, we distinguish two cases:
		\begin{itemize}
			\item if $\F=\CMcal{S}\circ_{m-1}\T$ is a nontrivial factorization, then when $\T=\T^*$, we prove that $[\F]=\{\F,\F^*\}$ (see Theorem \ref{thm:a<b} (1)--(a)). %(see Lemma \ref{lem:canonical2} and \ref{lem:sym}).
			If, on the other hand, $\T\ne\T^*$, we show that exactly four thickened ribbons are equivalent to $\F$ (see Theorem \ref{thm:a<b} (1)--(b));%(see Lemma \ref{lem:period},  Proposition \ref{lem:main} and Theorem \ref{thm:MvW});
			\item otherwise, when $\F$ only has a trivial factorization, we prove that $[\F]=\{\F,\F^*\}$ (see Theorem \ref{thm:a<b} (2)).%(see Lemma \ref{lem:period}, Proposition \ref{lem:B1} and \ref{lem:B2}).
		\end{itemize}
	\end{itemize}
	%Let us recall that it suffices to focus on the case (\ref{E:con1}), that is, the bottommost row and the topmost row of $\F$ have in total at least three boxes.
	%Throughout this section, $\F=\alpha\boxdot\beta$ will be a thickened ribbon with $|\alpha|\neq|\beta|$ and $\alpha_1+\beta_{\ell(\beta)}\geq 3$. 

	For $|\alpha|\ne |\beta|$, canonical thickened ribbons {\emma are defined} differently (compare Definitions \ref{Def:can1} and \ref{Def:a<b}), but the {\emma notion} of determined elements {\emma remains unchanged} (see Definition \ref{Def:det}).
	
	\begin{definition}\label{Def:a<b}
		We say that $\F=\alpha\boxdot\beta$ is \emph{canonical} if $|\alpha|<|\beta|$.
	\end{definition}
	The starting point of the proof is again Lemma \ref{lem:canonical2}. It is important to note that Lemma \ref{lem:canonical2} {\emma does not hold} if $\E$ is not canonical. For example, let $\F=312\boxdot 31412$ and $\E=31412\boxdot 312$, then $\bh_{\F}=\bh_{\E}$ but $\F\ne \E$.

	%\begin{lemma}\label{lem:canonical}
	%	Suppose $\F=\alpha\boxdot\beta$ is canonical, then for any canonical $\E$ with $\E\sim \F$, we have $\bh_\E=\bh_\F$ if and only if $\E=\F$.
	%\end{lemma}
	
	As a result of $\F\sim\F^*$ and Corollary \ref{lem:basic}, it is sufficient to characterize the equivalence class of $\F$ among canonical thickened ribbons consisting of ribbons of sizes $|\alpha|$ and $|\beta|$. %Unless otherwise specified, $\F$ is always assumed to be canonical.

	\subsection{Invariants under equivalence}
	{\emma In contrast to the equal-size case}, 
	the unequal-size case introduces four {\emma critical positions $|\alpha|-m+1$, $|\alpha|$, $|\beta|-m+1$, and $|\beta|$, where} cancellation can {\emma alter the} type. Away from {\emma these} positions, {\emma the} type remains invariant; at each critical position, the additional hypothesis on $\bh_\F(m-1)$ restores invariance. This subsection records {\emma precisely} what survives equivalence and {\emma organizes} the resulting behavior into four classes $\mathcal{A}$ to $\mathcal{D}$.
	
	%We start with some properties of the equivalence class of $\F$. 
	Recall that $r=\gcd(|\alpha|-m+1,|\beta|-m+1)$.
	Corollary \ref{lem:basic} guarantees that $|\alpha|,|\beta|,r$ are invariant under equivalence. %meaning that any $\E=\alpha'\boxdot\beta'\sim \F$ satisfies $\{|\alpha'|,|\beta'|\}=\{|\alpha|,|\beta|\}$ so that $\gcd(|\alpha'|-m+1,|\beta'|-m+1)=r$. %Furthermore, since $[m-1]=[|\alpha|-m+1]$, we always assume that $m-1\le |\alpha|-m+1$.

	%In the rest of this section, we only consider canonical thickened ribbons 

	%\begin{definition}\label{Def:det2}
	%	\lsx{Suppose $\F=\alpha\boxdot\beta$ is canonical. For $1\leq x\leq |\alpha|+|\beta|-2$, $x$ is said to be \emph{determined} if for all canonical $\E$ with $\E\sim \F$, we always have $\bh_\E(x)=\bh_\F(x)$. Otherwise, we say $x$ is \emph{undetermined}.}
	
	%Further, a relation $\bh_{\F}(x)=\bh_{\F}(y)$ is said to be {\em determined} if $\bh_{\E}(x)=\bh_{\E}(y)$ holds for any canonical $\E,\F$ with $\E\sim \F$. %and the types of $x,y$ in $\F$ are the same as in $\E$.
	%\end{definition}
	
	%Lemma \ref{lem:canonical} indicates that when nontrivial equivalence exists among canonical ones, namely there exists an $\E$ satisfying $\F\ne \E\sim\F$ and $\F,\E$ are canonical, we must have $\bh_\E(x)\ne \bh_\F(x)$ for some $x$. Such $x$ is defined as undetermined element; see Definition \ref{Def:det}.
	
	%This motivates the follow-up definitions: the first one was introduced in \cite{BTvW}, while the second one is new, tailored for the equivalence classes of positive integers in Definition \ref{def:equiv}.
	
	%\begin{definition}
	%	Suppose $\F=\alpha\boxdot\beta$ is canonical. For $1\leq x\leq|\alpha|+|\beta|-2$, $x$ is said to be \emph{determined} if for all canonical $\E\sim\F$, we always have $\bh_\E(x)=\bh_\F(x)$. Otherwise, we say $x$ is \emph{undetermined}.
	%\end{definition}

	The lemma below discusses the type of elements in the equivalence class of $\F$. Note that the cancellation of contributions from coarsening compositions could change the type of elements. 
	As a case in point, let $\F=122\boxdot 3$ and $\E=12\boxdot 32$ (see Figure \ref{F:e2}). By Definition \ref{Def:type}, we have $4\in A_1$ in $\F$, whereas $4\in A_0$ in $\E$. This kind of scenario is excluded from Lemma \ref{lem:types}.
	
	\begin{lemma}\label{lem:types}
		Suppose that $\E\sim \F=\alpha\boxdot \beta$. Then:
		\begin{enumerate}
			\item for $x\notin\{|\alpha|-m+1,|\alpha|,|\beta|-m+1,|\beta|\}$, the element $x$ is of type $i$ in $\E$ if and only if it is of type $i$ in $\F$;
			\item for $x\in\{|\alpha|-m+1,|\alpha|,|\beta|-m+1,|\beta|\}$, if $\F$ and $\E$ are canonical and $\bh_\F(m-1)=\bh_\E(m-1)$, then $\bh_\F(x)=\bh_\E(x)$. In particular, the type of $x$ in $\F$ and $\E$ is the same.
		\end{enumerate}	
	\end{lemma}
	\begin{proof}
		Let $n=|\alpha|+|\beta|$, we first prove $(1)$. Fix  $x\notin\{|\alpha|-m+1,|\alpha|,|\beta|-m+1,|\beta|\}$, consider the partition $\lambda=\lambda(n-x-m+1,x,m-1)$ and let $d=|[h_\lambda]s_\F|$. We have $\{(\CMcal{S};0)\succeq(\F):\lambda(\CMcal{S})=\lambda\}=\varnothing$ so that $d=\left|\{(\gamma;1)\succeq(\F):\lambda(\gamma)=\lambda(n-x-m+1,x)\}\right|$. Consequently,
		{\emma \begin{align*}
			x=(n-m+1)/2 \iff x\in A_{2d}\quad \mbox{ and }\quad x\ne (n-m+1)/2 \iff x\in A_{d}.
		\end{align*}}
	   Since $|[h_\lambda]s_\E|=|[h_\lambda]s_\F|$, the type of $x$ in $\E$ must be the same as the type of $x$ in $\F$.
		
		We next establish $(2)$. Since $\F,\E$ are canonical, it is clear that 
		$$\bh_\F(x)=\bh_\E(x)=-$$ 
		for $x\in \{|\alpha|-m+1,|\alpha|\}$. Combined with (1), {\emma this implies} that the {\emma type} of $m-1$ is the same in $\F$ and $\E$. {\emma It remains} to prove that $\bh_\F(x)=\bh_\E(x)$ for $x\in \{|\beta|-m+1,|\beta|\}$. Since $\bh_\F(m-1)=\bh_\E(m-1)$, we must have $\bh_\F(n-2m+2)=\bh_\E(n-2m+2)$.
		Let $\lambda=(|\beta|,|\alpha|-m+1,m-1)$. Then 
		{\emma \begin{align*}
			\bh_\F(m-1)=\bh_\F(|\beta|) \iff [h_\lambda]s_\F=0,
		\end{align*}
		which gives $\bh_\F(|\beta|)=\bh_\E(|\beta|)$.}
		
		Similarly, let $\lambda=(|\beta|-m+1,|\alpha|,m-1)$. Then 
       \begin{align*}
	{\emma 	\bh_\F(n-2m+2)=\bh_\F(|\beta|-m+1) \iff [h_\lambda]s_\F=0,}
	\end{align*}
	leading to $\bh_\F(|\beta|-m+1)=\bh_\E(|\beta|-m+1)$. For $i\in\{|\beta|-m+1,|\beta|\}$, since $i$ is of type $1$ in $\F$ if and only if $\bh_{\F}(i)=+$, the types of $x$ in $\F$ and $\E$ agree for all $x\in\{|\alpha|-m+1,|\alpha|,|\beta|-m+1,|\beta|\}$. This completes the proof.
	\end{proof}
	\begin{example}\label{Example:11}
		Let $\F=312\boxdot 31412$ and $\E=213\boxdot 21413$ be two canonical thickened ribbons. The types and the $\bh$-functions are given as below:
		\begin{center}
			\begin{tabular}{c|ccccccccccccccc}
				$x$ & $1$ & $2$ & $3$ & $4$ & $5$ & $6$ & $7$ & $8$ & $9$ & $10$ & $11$ & $12$ & $13$ & $14$ & $15$ \\
					\hline \\[-1.5ex]
				{type} & $0$ & $1$ & $2$ & $1$ & $0$ & $0$ & $1$ & $2$ & $1$ & $0$ & $0$ & $1$ & $2$ & $1$ & $0$ \\
				$\bh_{\F}(x)$ & $-$ & $-$ & $+$ & $+$ & $-$ & $-$ & $-$ & $+$ & $+$ & $-$ & $-$ & $-$ & $+$ & $+$ & $-$ \\
				$\bh_{\E}(x)$ & $-$ & $+$ & $+$ & $-$ & $-$ & $-$ & $+$ & $+$ & $-$ & $-$ & $-$ & $+$ & $+$ & $-$ & $-$ \\
			\end{tabular}
		\end{center}
        \vspace{2mm}	
		For $x\in \{5,6,10,11\}$, we have $\bh_{\F}(x)=\bh_{\E}(x)$, and the types of any element in $\F$ and $\E$ are the same, as asserted by Lemma  \ref{lem:types} (2).
	\end{example}

	We next examine the types of elements in each equivalence class $[i]$ in order to identify undetermined elements and establish relations {\emma among} them. 
	
	\begin{definition}\label{Def:7}
		We partition the equivalence classes $\{[i]:1\leq i<r\,\mbox{ and }\, i\ne m-1\}$ into four classes.
		\begin{enumerate}
			\item If $[i]\subseteq A_1$ and for all $j\in[i]$ with $1\leq j\leq|\alpha|+|\beta|-m-r$, we have $\bh_\F(j)=\bh_\F(j+r)$, then we say $[i]\in\mathcal{A}$.
			\item If $[i]\cap A_1\neq\varnothing$ and for some $j\in[i]$ with $1\leq j\leq|\alpha|+|\beta|-m-r$, we have $\bh_\F(j)\neq \bh_\F(j+r)$, then we say $[i]\in\mathcal{B}$.
			\item If $[i]\subseteq A_0$ or $[i]\subseteq A_2$, then we say $[i]\in\mathcal{C}$.
			\item If $[i]\subseteq A_0\cup A_2$, $[i]\cap A_0\neq\varnothing$ and $[i]\cap A_2\neq\varnothing$, then we say $[i]\in\mathcal{D}$.
		\end{enumerate}
		
		If $m-1\in A_0$, then $[m-1]$ belongs to one of the classes as defined above. 
		
		If $m-1\in A_1$, then $[m-1]\in\mathcal{A}$ if and only if the following conditions hold:
		\begin{enumerate}[label=(\roman*)]
			\item $\bh_\F(j)=\bh_\F(m-1)$ for all $j\ne |\alpha|$ with $j\equiv m-1\mod r$;
			\item $\bh_\F(j)=\bh_\F(|\alpha|+|\beta|-2m+2)$ for all $j\ne |\alpha|-m+1$ with $ j\equiv 0\mod r$.
		\end{enumerate}
		Otherwise, we say $[m-1]\in\mathcal{B}$. 
		
		For $m-1\in A_2$, if $$\left( [m-1]\cap A_1\right)\setminus \{|\alpha|-m+1,|\alpha|,|\beta|-m+1,|\beta|\}\neq\varnothing,$$ then $[m-1]\in\mathcal{B}$. Otherwise,  $[m-1]\in\mathcal{C}$ if both (i) and (ii) hold; and $[m-1]\in\mathcal{D}$ if one of (i) and (ii) {\emma fails}.
	\end{definition}
	\begin{example}
		Let $\F=312\boxdot 31412$, where the types and $\bh_{\F}$ are shown in Example \ref{Example:11}. Here $m=2$ and $r=\gcd(6-1,11-1)=5$. Since $[1]=\{1,5,6,11\}\subseteq A_0$, we have $[1]\in \mathcal{C}$ by Definition \ref{Def:7} (3). Since $$[2]=\{2,4,7,9,12,14\}\subseteq A_1$$ and the corresponding $\bh_{\F}(x)$ is $\{-,+,-,+,-,+\}$, thus $[2]\in\mathcal{A}$ by Definition \ref{Def:7} (1). 
		
		In fact, all equivalence classes $[i]$ in $\F$ lie in $\mathcal{A}$ or $\mathcal{C}$. This is not a coincidence, as we {\emma will prove} in Lemma \ref{lem:period}.
	\end{example}

	{\emma The class $[m-1]$ requires separate consideration, since $\bh_{\F}(|\alpha|)=\bh_{\F}(|\alpha| - m+1)=-$ always holds. The crucial observation is that if $[i]\cap A_1\neq\varnothing$ and $\bh_\F(j)=\bh_\F(j+r)$ for all $j\in[i]$, then $[i]\subseteq A_1$. We state the condition $[i]\subseteq A_1$ in (1) and justify it briefly in Remark \ref{re:3}.}

	\begin{remark}\label{re:3}
		{\emma Suppose} $[i]\cap A_1\neq\varnothing$, and choose $i\in [i]\cap A_1$, so that $\bh_{\F}(i)=-\bh_{\F}(n-m+1-i)$, where $n=|\alpha|+|\beta|$. We will show that $j\in A_1$ for every $j\in [i]$. By Lemma \ref{lem:symmetry}, either 
		$$j\equiv i\mod r\quad \mbox{ or }\quad j\equiv m-1-i\mod r.$$ 
		{\emma For} the former case, {\emma the periodicity of $\bh_{\F}$} gives
		$$\bh_{\F}(j)=\bh_{\F}(i)=-\bh_{\F}(n-m+1-i)=-\bh_{\F}(n-m+1-j),$$ 
		hence $j\in A_1$. {\emma For} the latter case, we have $$j\equiv n-m+1-i\mod r$$ and {\emma therefore} $$\bh_{\F}(j)=\bh_{\F}(n-m+1-i)=-\bh_{\F}(i)=-\bh_{\F}(n-m+1-j),$$ {\emma again yielding} $j\in A_1$. Thus every $j\in [i]$ is of type $1$. 
	\end{remark}
	
	\begin{remark}\label{remark:1}
		If $r\vert (m-1)$ and $m-1\in A_1$, then $[m-1]\in\mathcal{B}$. {\emma Indeed, if} $[m-1]\in\mathcal{A}$, then since $m-1\in A_1$, we would have $\bh_\F(m-1)=\bh_\F(|\alpha|+|\beta|-2m+2)$, which is a contradiction.
	\end{remark}
	
	%We proceed to present the features of these four classes among all equivalent thickened ribbons of $\F$. 
	
	The lemma below {\emma shows} that the classification of equivalence classes $[i]$ is well-defined, in the sense that it is {\emma consistent across} all members of the equivalence class $[\F]$.
	
	\begin{lemma}\label{lem:class}
		The class to which $[i]$ belongs is invariant under equivalence of thickened ribbons. Moreover, if $[i]\in\mathcal{B}$, then there exists some $j\in[i]\setminus\{|\alpha|-m+1,|\alpha|\}$ that is determined.
	\end{lemma}
	\begin{proof}
		We prove it by distinguishing two cases: 
		\begin{enumerate}[label=(\Roman*)]
			\item $m-1\notin[i]$ or $i=m-1\in A_0\cup A_2$;
			\item $i=m-1\in A_1$.
		\end{enumerate}	
		For case (I): we have either $[i]\cap\{|\alpha|-m+1,|\alpha|,|\beta|-m+1,|\beta|\}=\varnothing$ or $\bh_{\F}(m-1)=\bh_{\E}(m-1)$ for all canonical equivalent thickened ribbons $\F,\E$ with $\F\sim \E$. By Lemma \ref{lem:types}, the type of any $j\in[i]$ is agreed in $[\F]$. Hence, by definition, if $[i]\not\subseteq A_1$, the class of $[i]$ can be distinguished.
		
		If $[i]\subseteq A_1$, we must determine whether $[i]\in\A$ or $[i]\in\mathcal{B}$. To this end, we establish the following assertion:
		\begin{itemize}
			%\item If $[i]\subseteq A_1$ and $(r+1)/2\in[i]$, then $[i]\in\mathcal{B}$.
			\item  If $[i]\subseteq A_1$, %and $(r+1)/2\notin[i]$, 
			then $[i]\in\mathcal{A}$ if and only if the following hold:
			\begin{enumerate}
				\item For every $1\le  j\le |\alpha|-1$ with $j\in[i]$, let $\lambda=\lambda(j,|\alpha|-j,|\beta|)$. Then $|[h_\lambda]s_\F|=1$.
				\item For every $1\le  j\le |\beta|-1$ and $j\in[i]$, let $\lambda=\lambda(j,|\beta|-j,|\alpha|)$. Then $|[h_\lambda]s_\F|=1$.
			\end{enumerate}
		\end{itemize}
		%Let $j=(r+1)/2$. Suppose that $[i]\not\in\mathcal{B}$, i.e., $[i]\in\A$, we must have $\bh_{\F}(j)=-\bh_{\F}(r+1-j)$ for the reason that $\bh_{\F}(j)=-\bh_{\F}(|\alpha|+|\beta|-1-j)=-\bh_{\F}(r+1-j)$. However this is contrary to the assumption $\bh_{\F}(j)=\bh_{\F}(r+1-j)$ as $j=r+1-j$. We thus conclude that $[i]\in\mathcal{B}$ if $[i]\subseteq A_1$ and $(r+1)/2\in[i]$. 
		This assertion shows that $[i]\in \A$ is invariant under equivalence of $\F$.
		
		If $[i]\subseteq A_1$, %and $(r+1)/2\notin[i]$, 
		conditions (1) and (2) above are equivalent to the identities
		\begin{align}\label{E:j1}
			\bh_{\F}(j)&=\bh_{\F}(|\beta|+j-m+1), \mbox{ for }\, 1\le j\le  |\alpha|-1\,\mbox{ and }\, j\in [i]\\
			\label{E:j2}\bh_{\F}(j)&=\bh_{\F}(|\alpha|+j-m+1), \mbox{ for }\, 1\le  j\le |\beta|-1 \,\mbox{ and }\, j\in [i].
		\end{align}
		{\emma It remains} to show that $[i]\in\A$ if and only if (\ref{E:j1})--(\ref{E:j2}) hold. That is, $$\bh_{\F}(j)=\bh_{\F}(j+r)$$ holds for all $1\le j\le |\alpha|+|\beta|-m-r$ if and only if  (\ref{E:j1})--(\ref{E:j2}) hold. 
		
		{\emma Suppose} $\bh_{\F}(j)=\bh_{\F}(j+r)$ is true, then $\bh_{\F}$ is $r$-periodic, thus (\ref{E:j1})--(\ref{E:j2}) must be true. {\emma Conversely, assuming}  (\ref{E:j1}) and (\ref{E:j2}), the same argument as in the proof of Lemma \ref{lem:symmetry} (via the Euclidean algorithm) yields $\bh_{\F}(j)=\bh_{\F}(j+r)$. \vspace{2mm}

		%Since there is a positive integer $q_1$ such that $|\beta|-1=q_1(|\alpha|-1)+r_1$ and $0\le r_1<|\alpha|-1$, we obtain $\bh_{\F}(j)=\bh_{\F}(|\beta|+j-1)=\bh_{\F}(j+r_1)$ according to (\ref{E:j1})--(\ref{E:j2}). It follows that $\bh_{\F}(j)=\bh_{\F}(j+r)$ by continuing to apply the Euclidean algorithm on the integers $|\alpha|-1$ and $r_1$ when $r_1\ne 0$. 
		
		An equivalent {\emma formulation gives a criterion} for $[i]\in\mathcal{B}$:
		\begin{itemize}
			%\item If $[i]\subseteq A_1$ and $(r+1)/2\in[i]$, then $[i]\in\mathcal{B}$.
			\item  If $[i]\subseteq A_1$, %and $(r+1)/2\notin[i]$, 
			then $[i]\in\mathcal{B}$ if and only if one of the following holds:
			\begin{enumerate}
				\item For some $1\le j\le |\alpha|-1$ with $j\in[i]$, let $\lambda=\lambda(j,|\alpha|-j,|\beta|)$. Then $|[h_\lambda]s_\F|=0$ or $2$. {\emma Specifically,}
				\begin{align*}
					\bh_\F(j)=+ \iff |[h_\lambda]s_\F|=2\quad \mbox{ and }\quad 
					\bh_\F(j)=- \iff |[h_\lambda]s_\F|=0.
				\end{align*}
				%Then, $\bh_\F(j)=+$ (or $-$) if and only if $|[h_\lambda]s_\F|=2$ (or 0, respectively).
				\item For some $1\le j\le |\beta|-1$ and $j\in[i]$, let $\lambda=\lambda(j,|\beta|-j,|\alpha|)$, we have $|[h_\lambda]s_\F|=0$ or $2$. {\emma Specifically,}
					\begin{align*}
					\bh_\F(|\alpha|-m+1+j)&=+ \iff |[h_\lambda]s_\F|=2,\\
					\bh_\F(|\alpha|-m+1+j)&=- \iff |[h_\lambda]s_\F|=0.
				\end{align*}
				%Then, $\bh_\F(|\alpha|-m+1+j)=+$ (or $-$) if and only if $|[h_\lambda]s_\F|=2$ (or 0, respectively).
			\end{enumerate}
		\end{itemize}
		In both cases, there exists some $j\in[i]$ that is determined.
		
		For case (II): $i=m-1\in A_1$. If $\{|\alpha|/2,|\beta|/2\}\cap[m-1]\neq\varnothing$, we claim that $r\vert (m-1)$ and $[m-1]\in\mathcal{B}$. Indeed, assuming $|\alpha|/2\in [m-1]$, Lemma \ref{lem:symmetry} gives either $|\alpha|/2\equiv 0 \mod r$ or $|\alpha|/2\equiv m-1 \mod r$, which implies $r\mid \gcd(|\alpha|/2,|\alpha|-m+1)\mid (m-1)$ or  $r\mid \gcd(|\alpha|/2-m+1,|\alpha|-m+1)\mid (m-1)$. That is, $r\vert (m-1)$ and we know that $[m-1]\in\mathcal{B}$ by Remark \ref{remark:1}.
		
		Otherwise, $|\alpha|/2\notin[m-1]$ and $|\beta|/2\notin[m-1]$. Then  $[m-1]\in\mathcal{A}$ if and only if the following conditions {\emma hold}:
		\begin{enumerate}
			\item For any $j\in[m-1]\setminus\{|\alpha|-m+1,|\alpha|,|\beta|-m+1,|\beta|\}$, we have $j\in A_1$.
			\item For any $m\le j\le |\alpha|-m$ and $j\in[m-1]$, let $\lambda=\lambda(j,|\alpha|-j,|\beta|)$. Then $|[h_\lambda]s_\F|=1$.
			\item For any $m\le j\le |\beta|-m$ and $j\in[m-1]
			$, let $\lambda=\lambda(j,|\beta|-j,|\alpha|)$. Then $|[h_\lambda]s_\F|=1$.
			\item Let $\lambda=(|\beta|,|\alpha|-m+1,m-1)$ and $\mu=(|\beta|-m+1,|\alpha|,m-1)$. Then $[h_\lambda]s_\F=[h_\mu]s_\F=0$.
		\end{enumerate}
		Again, this implies that $[i]\in \A$ is invariant {\emma under} equivalence.
		
		We now {\emma justify} the claim. Condition (4) {\emma is equivalent to}
		\begin{align*}
			\bh_{\F}(m-1)=\bh_{\F}(|\beta|)\quad\mbox{ and }\quad \bh_{\F}(|\alpha|+|\beta|-2m+2)=\bh_{\F}(|\beta|-m+1).
		\end{align*}
		Conditions (1)-(3) hold if and only if (\ref{E:j1}) and (\ref{E:j2}) are true for $m\le j\le |\alpha|-m$ and $m\le j\le |\beta|-m$, {\emma excluding} $|\alpha|-m+1,|\alpha|$, respectively. {\emma Equivalently}, conditions (1)--(4) are equivalent to (\ref{E:j1}) for $m-1\le j\le |\alpha|-m$ and (\ref{E:j2}) for $m\le j\le |\beta|-m+1$ with $j\not\in\{|\alpha|-m+1,|\alpha|\}$. By the Euclidean algorithm, this is interchangeable with 
		\begin{align*}
			\bh_\F(j)&=\bh_\F(m-1)\,\,\mbox{ for all }\, j\equiv m-1\!\!\mod r\,\mbox{ and }\,j\ne |\alpha|,\\
			\bh_\F(j)&=\bh_\F(|\alpha|+|\beta|-2m+2)\,\mbox{ for all }\, j\equiv 0\mod r \,\mbox{ and }\,j\ne |\alpha|-m+1,
		\end{align*}
		{\emma as desired}. This also yields a criterion for $[i]\in\mathcal{B}$ in this case:
		\begin{itemize}
			\item If $i=m-1\in A_1$, then $[i]\in\mathcal{B}$ if and only if one of the following is true.
			\begin{enumerate}
				\item For some $j\in[m-1]\setminus\{|\alpha|-m+1,|\alpha|,|\beta|-m+1,|\beta|\}$, we have $j\in A_0\cup A_2$. %Then $j$ must be determined.
				\item For some $m\le j\le |\alpha|-m$, let $\lambda=\lambda(j,|\alpha|-j,|\beta|)$, we have $|[h_\lambda]s_\F|=0$ or $2$. That is, 
				\begin{align*}
					\bh_\F(j)&=+ \iff |[h_\lambda]s_\F|=2,\\
					\bh_\F(j)&=- \iff |[h_\lambda]s_\F|=0.
				\end{align*}
				\item For some $m\le j\le |\beta|-m$, let $\lambda=\lambda(j,|\beta|-j,|\alpha|)$, we have $|[h_\lambda]s_\F|=0$ or $2$. That is,
				\begin{align*}
					\bh_\F(|\alpha|-m+1+j)&=+ \iff |[h_\lambda]s_\F|=2,\\
					\bh_\F(|\alpha|-m+1+j)&=- \iff |[h_\lambda]s_\F|=0.
				\end{align*}
				%$\bh_\F(|\alpha|-m+1+j)=+$ (or $-$) if and only if $|[h_\lambda]s_\F|=2$ (or $0$, respectively).
				\item Let $\lambda=(|\beta|,|\alpha|-m+1,m-1)$ and $\mu=(|\beta|-m+1,|\alpha|,m-1)$, then at least one of $|[h_\lambda]s_\F|$ and $|[h_\mu]s_\F|$ equals $1$. {\emma Specifically, if $|[h_\lambda]s_\F|=1$, then} 
				\begin{align*}
					\bh_\F(m-1)&=+ \iff [h_\lambda]s_\F=(-1)^{\ell(\F)-1},\\
					\bh_\F(m-1)&=- \iff [h_\lambda]s_\F=(-1)^{\ell(\F)}.
				\end{align*}
				If $|[h_\mu]s_\F|=1$, then
			    \begin{align*}
			    	\bh_\F(m-1)&=+ \iff [h_\lambda]s_\F=(-1)^{\ell(\F)},\\
			    	\bh_\F(m-1)&=- \iff [h_\lambda]s_\F=(-1)^{\ell(\F)-1}.
			    \end{align*}
			\end{enumerate}
		\end{itemize}
		It follows that if $[i]\in\mathcal{B}$, then there is some $j\in[i]\setminus\{|\alpha|-m+1,|\alpha|\}$ that is determined. This completes the proof.
	\end{proof}
	\subsection{Main result of Section \ref{S:4}}
	The necessity of Theorem \ref{T:1} for {\emma the unequal-size case} is established by the following theorem, which is the main result of this section.
	\begin{theorem}\label{thm:a<b}
		Suppose that $\F$ is canonical. 
		\begin{enumerate}
			\item
			If $\mathcal{B}=\mathcal{D}=\varnothing$, then $\F$ has a nontrivial factorization, say $\F=\CMcal{S}\circ_{m-1}\T$. 
			\begin{enumerate}
				\item
				If $\T=\T^*$, then the unique canonical thickened ribbon in the equivalence class of $\F$ is $\F$ itself.
				\item Otherwise, $\T\ne \T^*$, then the canonical thickened ribbons in the equivalence class of $\F$ are $\F$ and $\CMcal{S}\circ_{m-1}\T^*$.
			\end{enumerate}
			\item
			If $\mathcal{B}\cup\mathcal{D}\ne \varnothing$, then the only canonical thickened ribbon in the equivalence class of $\F$ is $\F$ itself. 
		\end{enumerate} 
	\end{theorem}
	
	We first {\emma illustrate how Theorem \ref{thm:a<b}  proves} Theorem \ref{T:1} in the case $\F=\alpha\boxdot \beta$ with $|\alpha|\ne |\beta|$, and then proceed to prove Theorem \ref{thm:a<b} itself.
	
	{\em Proof of Theorem \ref{T:1} for $\F=\alpha\boxdot \beta$ with $|\alpha|\ne |\beta|$}. It remains to show the necessity. If $\F=\CMcal{S}\circ_{m-1}\T$, we observe that $|\alpha|\neq|\beta|$ if and only if $\CMcal{S}\neq\CMcal{S}^*$. Otherwise, $\F=\CMcal{S}\circ\T$ with $\CMcal{S}=1$ and $\T=\F$. Since exactly one of $\F$ and $\F^*$ is canonical, Theorem \ref{thm:a<b} implies that the equivalence class of $\F$ has exactly $2^{\kappa}$ elements, where $\kappa$ equals the number of occurrences of $\CMcal{S}\ne \CMcal{S}^*$ and $\T\ne \T^*$, as desired.
	
	\qed
	\begin{remark}
		The concept of {\em irreducible} factorization appears in Conjecture \ref{Conj:1}. For the special case of $m$-regular thickened ribbons with only {\emma a single} $2\times m$ block, any nontrivial factorization of $\F$, if {\emma it} exists, must be of the form $\CMcal{S}\circ_{m-1}\T$ by Theorem \ref{T:1}. {\emma Moreover, such a factorization} yields the same number of occurrences of $\CMcal{S}\ne \CMcal{S}^*$ and $\T\ne \T^*$, regardless of irreducibility. It is indeed for this reason that we do not introduce irreducible factorizations. 
	\end{remark}
	
	%\begin{remark}
	%If $[i]\subseteq A_1$ and $(r+m-1)/2\in[i]$, then $[i]\in\mathcal{B}$. Let $j=(r+m-1)/2$ and suppose that $[i]\not\in\mathcal{B}$, that is, $[i]\in\A$, we have $\bh_{\F}(j)=-\bh_{\F}(r+m-1-j)$ because of $\bh_{\F}(j)=-\bh_{\F}(|\alpha|+|\beta|-m+1-j)=-\bh_{\F}(r+m-1-j)$. However this is contrary to the assumption $\bh_{\F}(j)=\bh_{\F}(r+m-1-j)$ as $j=r+m-1-j$. %We conclude that $[i]\in\mathcal{B}$ provided that $[i]\subseteq A_1$ and $(r+m-1)/2\in[i]$. 
	%\end{remark}
	
	\subsection{Proof of Theorem \ref{thm:a<b} (1) and (a)}\label{subsection:6.3}
	
	This subsection proves the {\emma non-trivial} factorization criterion {\emma for the subcase $\T=\T^*$}. The absence of $\mathcal B$ and $\mathcal D$ is equivalent to {\emma the} periodicity of $\bh_\F$; one period of $\bh_{\F}$ reconstructs the inner ribbon $\T$, while the quotient lengths determine the two-row outer ribbon $\CMcal{S}$.
	
	The statement (1) of Theorem \ref{thm:a<b} follows from a relation between nontrivial factorizations of $\F$ and equivalence classes of integers in Lemma \ref{lem:period}.
	\begin{lemma}\label{lem:period}
		An $m$-regular thickened ribbon $\F=\alpha\boxdot \beta$ with exactly one $2\times m$ block admits a nontrivial decomposition, i.e., $$\F=\CMcal{S}\circ_\W \T$$ with $\T\ne 1$, $\W=m-1$, and $\CMcal{S}=(p,q)$ with $p<q$, if and only if $\mathcal{B}=\mathcal{D}=\varnothing$.
	\end{lemma}
	\begin{proof}
		First note that $\mathcal{B}=\mathcal{D}=\varnothing$ if and only if $\bh_\F$ is periodic with period $r$,  except that $\bh_\F(|\alpha|-m+1)=\bh_\F(|\alpha|)=-$. 
		
		If $\F=\CMcal{S}\circ_{m-1} \T$ with $\CMcal{S}\ne 1$ and $\T\ne 1$, then since $\F$ contains exactly one $2\times m$ block, we must have $\CMcal{S}=(p,q)$, and $\bh_\F$ is periodic with period $|\T|-m+1$,  except at $|\alpha|-m+1$ and $|\alpha|$. {\emma Moreover}, 
		\begin{align*}
			|\T|=\frac{|\alpha|-m+1}{p}+m-1=\frac{|\beta|-m+1}{q}+m-1,
		\end{align*}
		which implies {\emma that} $|\T|-m+1$ divides $r$. {\emma Hence} $\bh_\F$ is periodic with period $r$, except at $|\alpha|-m+1$ and $|\alpha|$.
		
		{\emma Conversely}, suppose $\bh_\F$ is periodic with period $r$, except at $|\alpha|-m+1$ and $|\alpha|$. Let $\CMcal{M}=\{1\le x\le r-1:\bh_{\F}(x)=+\}$, then $\mathrm{Comp}(\CMcal{M})\vDash r$. Let $n=|\alpha|+|\beta|$. {\emma The ribbon} $\T\vDash (r+m-1)$ is constructed from $\mathrm{Comp}(\CMcal{M})$ as follows:
		\begin{itemize}
			\item If $\bh_{\F}(n-2m+2)=-$, then $\F$ {\emma is of} type $\W\rightarrow \O\rightarrow \W$ or $\W\uparrow \O\rightarrow \W$, and we take $\T$ to be {\emma the} composition obtained by adding $(m-1)$ to the last part of $\mathrm{Comp}(\CMcal{M})$.
			\item If $\bh_{\F}(n-2m+2)=+$, then $\F$ {\emma is of} type $\W\rightarrow \O\uparrow \W$ or $\W\uparrow \O\uparrow \W$, and we take $\T$ to be {\emma the} composition obtained by adding a new part $(m-1)$ at the end of $\mathrm{Comp}(\CMcal{M})$.
		\end{itemize}
		It follows that $$\F=\left((|\alpha|-m+1)/r,(|\beta|-m+1)/r\right)\circ_{m-1}\T$$ is a nontrivial factorization, as desired.
	\end{proof}
	Since $\F=\CMcal{S}\circ_{m-1}\T$ with $\T=\T^*$, all elements $x$ with $1\le x\le |\alpha|-1$ {\emma are of even type}, which implies that $\mathcal{A}=\varnothing$. By Lemma \ref{lem:canonical2} and the lemma below, {\emma it follows} that $\F$ is the unique canonical thickened ribbon in its equivalence class. This proves Theorem \ref{thm:a<b} (1)--(a).
	
	\begin{lemma}\label{lem:sym}
		If $i$ is undetermined, then $[i]\in\mathcal{A}$ and all elements of $[i]\setminus\{|\alpha|-m+1,|\alpha|\}$ are undetermined. In particular, $m-1\le \min[i]\le |\alpha|-m$.
	\end{lemma}
	\begin{proof}
		Suppose $[i]\in\mathcal{A}$ and there exists some $j\in[i]\setminus\{|\alpha|-m+1,|\alpha|\}$ that is determined. Then, for any $x\in[i]\setminus\{|\alpha|-m+1,|\alpha|\}$, since $[i]\in\mathcal{A}$, we have $$\bh_\F(x)=\bh_\F(j)$$ whenever $x\equiv j\mod r$. On the other hand, because $x\in A_1$, $\bh_\F(x)=-\bh_\F(|\alpha|+|\beta|-m+1-x)$ and consequently $$\bh_\F(x)=-\bh_\F(j)$$ if $x\equiv m-1-j\mod r$. Since $j$ is determined, {\emma this implies that} $x$ is also determined, a contradiction. {\emma Therefore,} if $[i]\in\mathcal{A}$, then all elements of $[i]$ other than $|\alpha|-m+1$ and $|\alpha|$ must be undetermined.
		
		%(or $-\bh_\F(j)$) if and only if $x\equiv j\mod r$ (or $1-j\mod r$, respectively). In particular, $x$ is determined. The statement follows from the fact that $|\alpha|-1$ and $|\alpha|$ are always determined.
		
	    It remains to prove that $[i]\in\mathcal{A}$.
		If $[i]\in\mathcal{C}\cup\mathcal{D}$, then either $i\in A_0\cup A_2$ that must be determined, or $m-1\in A_2$ and $i\in\{|\alpha|-m+1,|\alpha|,|\beta|-m+1,|\beta|\}$ that are also determined. {\emma Hence} $[i]\not\in \mathcal{C}\cup\mathcal{D}$. 
		
		Now suppose $[i]\in\mathcal{B}$. By Lemma \ref{lem:class}, there exists $j\in[i]\setminus\{|\alpha|-m+1,|\alpha|\}$ that is determined. {\emma We will show} that any element of $[j]$ is determined, which contradicts the assumption that $i\in [j]$ is undetermined. Therefore, we must have $[i]\in\mathcal{A}$.
		
		For this purpose, we will show that for every $x\in[j]\backslash \{|\alpha|-m+1,|\alpha|\}$, if $x$ is determined, then so is $y$, where $y$ is equivalent to $x$ via one of the three relations in Definition \ref{def:equiv}:
		\begin{align*}
			y=n-m+1-x, \quad y=|\alpha|-x\quad \mbox{ or }\quad y=|\alpha|+n-2m+2-x,
		\end{align*}
		where $n=|\alpha|+|\beta|$.
		\begin{enumerate}
			\item If $y=n-m+1-x$, then 
			\begin{itemize}
				\item if $x\in A_0\cup A_2$, we have $\bh_{\F}(y)=\bh_{\F}(x)$.
				\item If $x\in\{|\beta|-m+1,|\beta|\}$, then $\bh_{\F}(y)=-$.
				\item  If $x\in A_1$ and $x\not\in \{|\alpha|-m+1,|\alpha|,|\beta|-m+1,|\beta|\}$, we must have $\bh_{\F}(y)=-\bh_{\F}(x)$.   
			\end{itemize}
			 %\item If $x\in A_0\cup A_2$ or $x\in\{|\beta|-1,|\beta|\}$, then $|\alpha|+|\beta|-1-x$ must be determined. Otherwise, $\bh_\F(|\alpha|+|\beta|-1-x)=-\bh_\F(x)$ i.e. $|\alpha|+|\beta|-1-x$ is determined.
			\item If $y=|\alpha|-x$, then 
			\begin{itemize}
				\item if $x=m-1$, then $|\alpha|-m+1$ is determined because $\bh_{\F}(|\alpha|-m+1)=-$. 	
				\item If $1<x<|\alpha|-1$ and $x\not\in \{m-1,|\alpha|-m+1\}$, let $\lambda=\lambda(x,|\alpha|-x,|\beta|)$. {\emma Then $y$ is determined by the following table.}
			\end{itemize}
			%\begin{itemize}
			%	\item if $\bh_\F(x)=+$, then $\bh_\F(y)=+$ (or $-$) if and only if $d=2$ (or $1$, respectively).
			%	\item if $\bh_\F(x)=-$, then $\bh_\F(y)=+$ (or $-$) if and only if $d=1$ (or $0$, respectively).
			%\end{itemize}
		    {\emma 
		    	\begin{center}
		    		\begin{tabular}{c|cc}
		    			$(\bh_\F(x),\bh_\F(y))$ & $|[h_\lambda]s_\F|$\\
		    			\hline \\[-1.5ex]
		    			$(+,+)$ &  $2$\\
		    			$(+,-)$ or $(-,+)$ & $1$\\
		    			$(-,-)$ &  $0$
		    		\end{tabular}
		    \end{center}}
		    \vspace{2mm}
		
			%If $x=|\alpha|-1$, then $y=1$ is determined because $\bh_{\F}(|\beta|)=+$ if $|\alpha|-1\in A_0$ is determined by the type of $|\alpha|-1$ and $\bh_{\F}(1)$ is determined by $\bh_{\F}(|\beta|)$ according to the relation that $\bh_\F(1)=\bh_\F(|\beta|)$ if and only if $|[h_\lambda]s_\F|=0$ where $\lambda=(|\beta|,|\alpha|-1,1)$.
			\item If $y=|\alpha|+n-2m+2-x$, then 
			\begin{itemize}
				\item if $|\alpha|< x<n-m$, let $\lambda=\lambda(|\alpha|,x-|\alpha|+m-1,n-x-m+1)$. {\emma Then $y$ is also determined by the previous table.} 
				\item If $x=n-m$, then $y=|\alpha|-m+2$ is determined as $\bh_{\F}(|\alpha|-m+2)=-$.
			\end{itemize}
			
			%and $d=|[h_{\lambda}]|s_{\F}$, then 
			%\begin{itemize}
			%	\item if $\bh_\F(x)=+$, then $\bh_\F(y)=+$ (or $-$) if and only if $d=2$ (or $1$, respectively).
			%	\item if $\bh_\F(x)=-$, then $\bh_\F(y)=+$ (or $-$) if and only if $d=1$ (or $0$, respectively).
			%\end{itemize}
			%\item Let $\lambda=(|\beta|,|\alpha|-1,1)$, then $\bh_\F(1)=\bh_\F(|\beta|)$ if and only if $|[h_\lambda]s_\F|=0$. Similarly, let $\lambda=(|\beta|-1,|\alpha|,1)$, then $\bh_\F(|\beta|-1)=\bh_\F(|\alpha|+|\beta|-2)$ if and only if $|[h_\lambda]s_\F|=0$.
		\end{enumerate}
		As a consequence, if $[i]\in\mathcal{B}$, then all elements of $[i]\backslash \{|\alpha|-m+1,|\alpha|\}$ are determined. Hence, if $i$ is undetermined, we must have $[i]\in\A$. 
		
		For $1\le z<m-1$, the element $z$ must be determined. Indeed, if $z$ were undetermined, then $[z]\in\mathcal{A}$, but $\bh_{\F}(z)=-\bh_{\F}(|\alpha|-z)=+$ would be determined, a contradiction. {\emma Therefore, the smallest undetermined element must be at least $m-1$.}
		
		 Consequently, if $i=\min[i]$ is undetermined, then $|\alpha|-i$ is also undetermined. In particular, $|\alpha|-i\ge m-1$ and $i\ge m-1$, that is, $m-1\le i\le |\alpha|-m$ since $|\alpha|-m+1$ is determined with $\bh_{\F}(|\alpha|-m+1)=-$.
	\end{proof}
	
	\subsection{Proof of Theorem \ref{thm:a<b} (b)}
	
	This subsection {\emma deals with} the case $\T\ne\T^*$. The next two lemmas are the unequal-size {\emma counterparts} of Lemmas \ref{lem:BTvW4} and \ref{lem:BTvW5}: they {\emma employ} the same coefficient tests away from the four {\emma critical values $|\alpha|-m+1$, $|\alpha|$, $|\beta|-m+1$, and $|\beta|$}, {\emma with} the additional hypotheses {\emma addressing} cancellations at {\emma these values}. The resulting propagation argument {\emma then} shows that reversing the inner factor {\emma yields} the only second canonical representative.
	
	\begin{lemma}\label{lem:BTvW1}%[Compare with Lemma 5.2 of \cite{BTvW}]
		Suppose $x,y,x+y\in A_1$. Then whether or not $$\bh_{\F}(x)=\bh_{\F}(y)=-\bh_{\F}(x+y)$$ holds can be determined, unless $x+y\in\{|\alpha|-m+1,|\beta|-m+1\}$ or $\{x,y\}\cap\{|\alpha|,|\beta|\}\neq\varnothing$.  
	\end{lemma}
	\begin{proof}
		The elements $x,y$ satisfying $x+y\notin\{|\alpha|-m+1,|\beta|-m+1\}$ and $\{x,y\}\cap\{|\alpha|,|\beta|\}=\varnothing$ {\emma fall into} exactly one of the following cases:
		\begin{enumerate}
			\item $x+y\in\{|\alpha|,|\beta|\}$;
			\item $\{x,y\}\cap\{|\alpha|-m+1,|\beta|-m+1\}\neq\varnothing$;
			\item $\{x,y,x+y\}\cap\{|\alpha|-m+1,|\alpha|,|\beta|-m+1,|\beta|\}=\varnothing$.
		\end{enumerate}
		
		Case (1): The element $x+y\in A_1$ must be determined, namely, $\bh_{\F}(|\alpha|)=-\quad\mbox{and}\quad \bh_{\F}(|\beta|)=+$.
		Since $x\equiv m-1-y\mod r$ and $x,y\in A_1$, Lemma \ref{lem:symmetry} gives $[x]=[y]$, and hence $[x]\in\A\cup \mathcal{B}$. If $[x]\in\mathcal{B}$, then both $x$ and $y$ are determined by Lemma \ref{lem:sym}, and we are done.  
		If $[x]\in\mathcal{A}$, then we obtain $\bh_\F(x)=-\bh_\F(y)$.
		
		Case (2): This case is similar to Case (1). Suppose $x\in \{|\alpha|-m+1,|\beta|-m+1\}$. Then $x$ is clearly determined as $\bh_{\F}(|\alpha|-m+1)=-\quad\mbox{and}\quad\bh_{\F}(|\beta|-m+1)=+$.
        Since $y\equiv x+y\mod r$ and $x,y\in A_1$, we get $[y]=[x+y]\in \mathcal{A}\cup \mathcal{B}$ by Lemma \ref{lem:symmetry}. If $[y]\in \mathcal{B}$, then both $y$ and $x+y$ are determined by Lemma \ref{lem:sym}; otherwise,  $[y]\in\mathcal{A}$ and $\bh_{\F}(y)=\bh_{\F}(x+y)$.
		
		Case (3): Let $\lambda=\lambda(x,y,|\alpha|+|\beta|-x-y-m+1,m-1)$ and $d=[h_\lambda]s_\F$.  Then $\{(\CMcal{S};0)\succeq(\CMcal{\F}):\lambda(\CMcal{S})=\lambda\}=\varnothing$. Hence, by (\ref{eq:JT}), the multiset $\{(\gamma;1)\succeq(\F):\lambda(\gamma,m-1)=\lambda\}$ contains exactly $d$ elements and $\bh_{\F}(x)=\bh_{\F}(y)=-\bh_{\F}(x+y)$ if and only if $d=0$.
	\end{proof}	
	\begin{lemma}\label{lem:BTvW2}%[Compare with Lemma 5.3 of \cite{BTvW}]
		Suppose exactly two of $x,y,x+y$ belong to $A_1$, say $i,j\in A_1$. Then whether or not $$\bh_{\F}(i)=\bh_{\F}(j)$$ {\emma holds} can be determined, unless  $x+y \in\{|\alpha|-m+1,|\beta|-m+1\}$ or $\{x,y\}\cap\{|\alpha|,|\beta|\}\ne \varnothing$. Moreover, if $m-1\in A_0$ and $\mathcal{B}=\varnothing$, 
		%$\{|\alpha|-m+1,|\alpha|\}\subseteq A_0$, 
		the statement is always true. 
		%unless $x+y\in\{|\alpha|-m+1,|\beta|-m+1\}$ or $\{x,y\}\cap\{|\alpha|,|\beta|\}\neq\varnothing$. Moreover, if $m-1\in A_0\cup A_2$ and $\mathcal{B}=\varnothing$, then 
	\end{lemma}
	
	\begin{proof}
		Let $\lambda=\lambda(x,y,|\alpha|+|\beta|-x-y-m+1,m-1)$ and set $d=|[h_\lambda]s_\F|$.
		Without loss of generality, we discuss {\emma only} the cases {\emma where} either $x,y\in A_1$ {\emma or} $x,x+y\in A_1$ subject to the condition $x+y\not\in\{|\alpha|-m+1,|\beta|-m+1\}$ and $\{x,y\}\cap\{|\alpha|,|\beta|\}=\varnothing$. If $m-1\in A_0$ and $\mathcal{B}=\varnothing$, we will show that the conclusion {\emma holds} without {\emma these restrictions}. \vspace{2mm}
		
		Case (I): $x,y\in A_1$ and $x+y\not\in A_1$:
		\begin{enumerate}
			\item If $x+y\in\{|\alpha|,|\beta|\}$, 
			then $\bh_{\F}(x+y)=-$. By Lemma \ref{lem:symmetry}, we have $[x]=[y]$. If $[x]\in\mathcal{B}$, then $x$ and $y$ are determined, and we are done. If $[x]\in\mathcal{A}$, then $\bh_\F(x)=-\bh_\F(y)$.
			\item If $x+y\notin\{|\alpha|-m+1,|\alpha|,|\beta|-m+1,|\beta|\}$ and $\{x,y\}\cap\{|\alpha|-m+1,|\beta|-m+1\}\neq\varnothing$, assume without loss of generality that $x\in\{|\alpha|-m+1,|\beta|-m+1\}$. Then $x$ must be determined. Moreover, $[y]=[x+y]$ by Lemma \ref{lem:symmetry}. Since $x+y\not\in A_1$, we find $[x+y]\notin\mathcal{A}$; together with Lemma \ref{lem:sym}, this implies that both $y$ and $x+y$ are determined.
			\item If $\{x,y,x+y\}\cap\{|\alpha|-m+1,|\alpha|,|\beta|-m+1,|\beta|\}=\varnothing$, then  $\{(\CMcal{S};0)\succeq(\CMcal{\F}):\lambda(\CMcal{S})=\lambda\}=\varnothing$. Hence, by (\ref{eq:JT}), the multiset $\{(\gamma;1)\succeq(\F):\lambda(\gamma,m-1)=\lambda\}$ contains exactly $d$ elements. Furthermore, {\emma whether $\bh_\F(x)=\bh_\F(y)$ or $\bh_\F(x)=-\bh_\F(y)$ is determined according to the following table:}
			 {\emma 
				\begin{center}
					\begin{tabular}{c|cc}
						Case (I)--(3)& $\bh_\F(x)/\bh_\F(y)$ & $d$ \\
						\hline \\[-1.5ex]
						$x+y\in A_0$ & $1$ & $0$ \\
						& $-1$ & $1$ \\
						\hline \\[-1.5ex]
						$x+y\in A_2$ & $1$ & $2$ \\
						& $-1$ & $3$ 
					\end{tabular}
			\end{center}}
			%\begin{itemize}
			%	\item if $x+y\in A_0$, then $\bh_{\F}(x)=\bh_{\F}(y)$ (or $-\bh_\F(y)$) if and only if $d=0$ (or $1$, respectively);
			%	\item if $x+y\in A_2$, then $\bh_{\F}(x)=\bh_{\F}(y)$ (or $-\bh_\F(y)$) if and only if $d=2$ (or $3$, respectively).
			%\end{itemize}
		\end{enumerate}
		
		Case (II): $x,x+y\in A_1$ and $y\not\in A_1$:
		\begin{enumerate}
			\item $\{x,y\}\cap\{|\alpha|-m+1,|\beta|-m+1\}\neq\varnothing$:\\
			If $y\in\{|\alpha|-m+1,|\beta|-m+1\}$, then $y$ is determined, since $y\notin A_1$. In addition, $[x+y]=[x]$. If $[x]\in\mathcal{B}$, then $x$ and $x+y$ are determined, and we are done. If $[x]\in\mathcal{A}$, then $\bh_\F(x)=\bh_\F(x+y)$.
			\\
			Otherwise, if $x\in\{|\alpha|-m+1,|\beta|-m+1\}$, then $x$ is determined, since $\bh_{\F}(|\alpha|-m+1)=-$ and $\bh_{\F}(|\beta|-m+1)=+$. Moreover, $[x+y]=[y]$,  and {\emma since} $y\not\in A_1$, we have $[y]\notin\mathcal{A}$. By Lemma \ref{lem:sym}, both $x+y$ and $y$ are determined.
			\item $\{x,y\}\cap\{|\alpha|-m+1,|\alpha|,|\beta|-m+1,|\beta|\}=\varnothing$ and $x+y\in\{|\alpha|,|\beta|\}$:\\ Then $\bh_{\F}(|\alpha|)=-$ and $\bh_{\F}(|\beta|)=+$. Moreover, $[x]=[y]$ by Lemma \ref{lem:symmetry}, and we have $[y]\notin\mathcal{A}$, since $y\not\in A_1$. By Lemma \ref{lem:sym}, we know that $x,y$ are also determined.
			\item  $\{x,y,x+y\}\cap\{|\alpha|-m+1,|\alpha|,|\beta|-m+1,|\beta|\}=\varnothing$: \\
			Then $\{(\CMcal{S};0)\succeq(\CMcal{\F}):\lambda(\CMcal{S})=\lambda\}=\varnothing$. Hence, the multiset $\{(\gamma;1)\succeq(\F):\lambda(\gamma,m-1)=\lambda\}$ contains exactly $d$ elements by (\ref{eq:JT}). Furthermore, {\emma whether $\bh_\F(x)=\bh_\F(x+y)$ or $\bh_\F(x)=-\bh_\F(x+y)$ is determined according to the following table:}
			{\emma 
				\begin{center}
					\begin{tabular}{c|cc}
					 Case (II)--(3)	& $\bh_\F(x)/\bh_\F(x+y)$ & $d$ \\
						\hline \\[-1.5ex]
						$y\in A_0$ & $1$ & $1$ \\
						& $-1$ & $0$ \\
						\hline \\[-1.5ex]
						$y\in A_2$ & $1$ & $3$ \\
						& $-1$ & $2$ 
					\end{tabular}
			\end{center}
		    \vspace{2mm}}

			%\begin{itemize}
			%	\item if $y\in A_0$, then $\bh_{\F}(x)=\bh_{\F}(x+y)$ (or $-\bh_\F(x+y)$) if and only if $d=1$ (or $0$, respectively);
			%	\item if $y\in A_2$, then $\bh_{\F}(x)=\bh_{\F}(x+y)$ (or $-\bh_\F(x+y)$) if and only if $d=3$ (or $2$, respectively).
			%\end{itemize}
		\end{enumerate}
		
		Case (III): $m-1\in A_0$, $\mathcal{B}=\varnothing$ and 
		$x+y\in\{|\alpha|-m+1,|\beta|-m+1\}$: \\
		Observe that $x+y\in A_0$; otherwise, if $x+y\in A_1$, then $[m-1]\in\mathcal{B}$, {\emma contradicting} the assumption $\mathcal{B}=\varnothing$. {\emma Hence}
		$x,y\in A_1$, so $[x],[y]\in\mathcal{A}$, which {\emma gives} $\bh_\F(x)=-\bh_\F(y+m-1)$ and $y+m-1\in A_1$. {\emma Applying} Case (II)--(3) to {\emma the triple $(m-1,y,y+m-1)$}, we find that whether $\bh_\F(y)=\bh_\F(y+m-1)$ {\emma holds} can be determined. {\emma Consequently,} we can determine whether $\bh_\F(x)=\bh_\F(y)$.
		
		Case (IV): $m-1\in A_0$, $\mathcal{B}=\varnothing$ and 
		$\{x,y\}\cap\{|\alpha|,|\beta|\}\ne \varnothing$:\\ 
		Suppose $y\in\{|\alpha|,|\beta|\}$. {\emma Then} $y\in A_0$; otherwise, if $y\in A_1$, then $[y]=[m-1]$ and $[m-1]\in\mathcal{B}$, a contradiction. {\emma Hence},
		$x,x+y\in A_1$, so $[x],[x+y]\in\mathcal{A}$, which yields $\bh_\F(x+y)=\bh_\F(x+m-1)$ and $x+m-1\in A_1$. {\emma Applying} Case (II)--(3) {\emma to the triple $(m-1,x,x+m-1)$}, we obtain that whether $\bh_\F(x)=\bh_\F(x+m-1)$ holds can be determined. Hence, we can determine whether  $\bh_\F(x)=\bh_\F(x+y)$.
		
	\end{proof}
	
	The arguments for the next two lemmas are completely analogous to Lemma \ref{lem:D2} and Proposition \ref{lem:final2}, respectively, which are consequences of Lemmas \ref{lem:BTvW1} and \ref{lem:BTvW2}. We therefore omit the details.
	
	\begin{lemma}\label{lem:D}
		The function $g_t$ satisfies the assumptions of Lemma \ref{lem:BTvW3} {\emma in} the following two cases:
		\begin{enumerate}
			\item $m-1\in A_0$ and $\mathcal{B}=\varnothing$. For each $2\leq t\leq z$, take $d=m_t$ and $c=r_{t-1}$. 
			\item $m-1\in A_1\cup A_2$ or $\mathcal{B}\ne \varnothing$. For each $2\leq t\leq z-1$, take $d=m_t$ and $c=r_{t-1}$. 
		\end{enumerate}
		Consequently, if we set $s=z$ for (1) and $s=z-1$ for (2), then
		$g_s$ is $r_s$-periodic on {\emma  $[1,m_{s}-1]$} except {\emma at} multiples of $r_s$, and $g_s$ is antisymmetric on {\emma $[1,r_s-1]$}. If $x<m_1$ and $r_s\mid x$, then $x\in A_0\cup A_2$.
	\end{lemma}
	\begin{lemma}\label{lem:final}
		If $\F$ is canonical and $m-1\in A_1$ is determined, then {\emma every element is} determined.
	\end{lemma}
	We are now ready to prove Theorem \ref{thm:a<b} (b), with the help of Proposition \ref{lem:main}.
	\begin{proposition}\label{lem:main}
		Suppose $\F$ is canonical and $\mathcal{B}=\varnothing$. Let $k$ be the smallest integer such that $[k]\in \mathcal{A}$. If $k$ is determined, then all elements are determined.
	\end{proposition}
	
	\begin{proof}
		%If such a $k$ does not exist, then any element $u$ is determined because  $[u]\in\mathcal{C}\cup\mathcal{D}$. In consequence, we only prove the case when $k$ exists. 
		
		Define a function $g:[1,|\alpha|+|\beta|-2m+2]\backslash\{|\alpha|-m+1,|\alpha|\}\to\{0,1,-1\}$ by
		$$g(x)=\begin{cases}
			0 & \text{ if } x\in A_0\cup A_2,\\
			1 & \text{ if } x\in A_1 \text{ and }\bh_\F(x)=+,\\
			-1 & \text{ if }x\in A_1 \text{ and }\bh_\F(x)=-.
		\end{cases}$$
		For $x\in \{|\alpha|-m+1,|\alpha|\}$, set $g(x)=0$. {\emma  We claim that} for all $1\le u<m-1$, we have 
		\begin{align}\label{E:g}
		g(u)\in \{0,1\}. 
		\end{align}
		Since $\mathcal{B}=\varnothing$, we have $[u]\in\mathcal{A}\cup \mathcal{C}\cup\mathcal{D}$. If $[u]\in \mathcal{A}$, then $g(u)=1$ {\emma because} $\bh_{\F}(u)=-\bh_{\F}(|\alpha|-u)=+$. If $[u]\in \mathcal{C}\cup\mathcal{D}$, then $g(u)=0$.
		
		%We will discuss the type of $m-1$. If $m-1\in A_1$, then $[m-1]\in\mathcal{A}$ and $k\le m-1$; otherwise $m-1\in A_0\cup A_2$ and $k\ne m-1$. 
		
		The condition $\mathcal{B}=\varnothing$ implies that $g$ is $r$-periodic, except at $\{|\alpha|-m+1,|\alpha|\}$. {\emma That is}, 
		\begin{align}\label{E:gpe2}
			g(u)=g(u+r)
		\end{align}
		for $1\le u<r$ such that $\{u,u+r\}\cap \{|\alpha|-m+1,|\alpha|\}=\varnothing$. {\emma Indeed}, 
		if $u\in A_0\cup A_2$, then $u+r\in A_0\cup A_2$, since $\mathcal{B}=\varnothing$. Hence
		$g(u)=g(r+u)=0$. If $u\in A_1$, then $[u]=[u+r]\in \mathcal{A}$, thus $g(u)=g(r+u)$.
		
		Since $g(u)=-g(|\alpha|+|\beta|-m+1-u)$ for all $u$, we have $g(u+m-1)=-g(|\alpha|+|\beta|-2m+2-u)$. In particular, for $1\leq u< r$, (\ref{E:gpe2}) implies $$g(u+m-1)=-g(|\alpha|+|\beta|-2m+2-u)=-g(r-u).$$
		
		Suppose on the contrary that there exists an undetermined element. We now follow the notation from Definition \ref{Def:gt} and Lemma \ref{lem:D}. Recall that $r_{z-1}\mid \gcd(m_1,r+m-1-m_1)$, and in particular, $r_{z-1}\mid (r+m-1)$. 
		
		Fix $1\leq u\leq m_{z-1}-m+1$ such that $r_{z-1}\nmid u$ and $r_{z-1}\nmid (u+m-1)$, we have $r_{z-1} \nmid (r-u)$. In particular, $u$ and $r-u$ are determined.
		%Recall that $g_{z-1}$ is $r_{z-1}$-periodic on $\{1,\ldots,m_{z-1}-1\}$ except multiples of $r_{z-1}$ by Lemma \ref{lem:D}.  We are going to prove that $g(u+m+1)=g(u)$ for such $u$ satisfying $r_s\nmid u$ and $r_s\nmid (u+m-1)$. 
		Then, $r_{z-1} \nmid (m_1+u)$ i.e. $m_1+u$ is determined. {\emma Applying Lemma \ref{lem:BTvW1} and \ref{lem:BTvW2} to the triples} 
		$$(m_1,u,m_1+u)\quad\mbox{ and }\quad(r-u,m_1+u,m_1+r),$$
		we obtain
		$g(r-u)=-g(m_1+u)=-g(u)$, since $u$, $r-u$ and $m_1+u$ are determined. {\emma Consequently},  $$g(u+m-1)=g(u)$$ for all $1\le u<m_{z-1}-m+1$ such that $r_{z-1}\nmid u$ and $r_{z-1}\nmid (u+m-1)$.
		
		By minimality of $k$, we have $k<m_1$. Recall {\emma from Lemma \ref{lem:D}} that $g_{z-1}$ is $r_{z-1}$-periodic on $[1,m_{z-1}-1]$, except at multiples of $r_{z-1}$. Let $$\sigma=\gcd(r_{z-1},m-1).$$ Then $g_{z-1}(u)=g(u)$, and $g_{z-1}$ is $\sigma$-periodic on the interval $[1,m_{z-1}-1]$, except at multiples of $r_{z-1}$. Since $k\in A_1$ and $k<m_1$, Lemma \ref{lem:D} guarantees that $r_{z-1}\nmid k$; in particular, $r_{z-1}>1$. We write $k\equiv a \mod r_{z-1}$ {\emma with} $a\ne 0$. By the antisymmetry of $g_{z-1}$, we have $g_{z-1}(r_{z-1}-a)=-g_{z-1}(a)=\pm 1$. Since $g_{z-1}$ is $\sigma$-periodic, both $1$ and $-1$ {\emma must} appear in the set $\{g(u):1\le u\le \sigma\}$. This contradicts the fact (\ref{E:g}).

	\end{proof}
	
	{\em Proof of Theorem \ref{thm:a<b} (b)}. Lemma \ref{lem:period} and Proposition \ref{lem:main} assert that the equivalence class of $\F$ contains at most two canonical thickened ribbons. On the other hand, {\emma  when $\T\ne \T^*$}, Theorem \ref{thm:MvW} {\emma provides} two equivalent canonical thickened ribbons. {\emma Hence} the equivalence class $[\F]$ {\emma contains} exactly two canonical thickened ribbons, namely, $\F$ and $\CMcal{S}\circ_{m-1}\T^*$. {\emma Equivalently}, the equivalence class $[\F]$ {\emma consists of} exactly four elements.  
	
	\qed	
	
	\subsection{Proof of Theorem \ref{thm:a<b} (2)}
	We now turn to the {\emma non-periodic} case $\mathcal B\cup\mathcal D\ne\varnothing$. The generalized triple lemmas below apply the same sign-detection mechanism separately to the ribbons $\alpha$ and $\beta$. Propositions \ref{lem:B1} and \ref{lem:B2} then rule out every remaining undetermined position.
	
	The next two lemmas, in parallel with Lemma \ref{lem:BTvW1} and \ref{lem:BTvW2}, make significant use of the condition $|\alpha|\ne |\beta|$ to derive determined relations of $\bh_{\F}$.
	
	\begin{lemma}\label{lem:GeneralBTvW1}
		Suppose that $\F$ is canonical. Then the statements below are true:
		\begin{enumerate}
			\item For $1\leq x,y,x+y\le |\alpha|-1$ satisfying $\bh_\F(i)=-\bh_\F(|\alpha|-i)$ with $i=x,y,x+y$. Then whether or not
			$$\bh_\F(x)=\bh_\F(y)=\bh_\F(|\alpha|-x-y)$$ can be determined, unless $m- 1\in\{x,y,|\alpha|-x-y\}$, $m-1\in A_2$ and $|\beta|\in A_1$. %and $\{x,y,x+y\}\backslash\{m-1\}\subseteq A_1$.
			\item For $1\leq x,y,x+y\le |\alpha|-1$ satisfying $\bh_\F(k)=-\bh_\F(|\alpha|-k)$ with $k=x,y,x+y$. Then whether or not
			$$\bh_\F(i)=\bh_\F(j)$$ can be determined, if $m-1\in\{x,y,|\alpha|-x-y\}$,  $\{i,j\}=\{x,y,|\alpha|-x-y\}\backslash\{m-1\}\not\subseteq A_1$, 
			$m-1\in A_2$ and $|\beta|\in A_1$;		 
			\item For $1\leq x,y,x+y\le |\beta|-1$ such that $\bh_\F(|\alpha|-m+1+i)=-\bh_\F(|\alpha|+|\beta|-m+1-i)$ with $i=x,y,x+y$. Then whether or not $$\bh_\F(|\alpha|-m+1+x)=\bh_\F(|\alpha|-m+1+y)=\bh_\F(|\alpha|+|\beta|-m+1-x-y)$$ can be determined, unless $m-1\in\{x,y,|\beta|-x-y\}$, $m-1\in A_2$ and $|\alpha|\in A_1$;
			\item For $1\leq x,y,x+y\le |\beta|-1$ such that $\bh_\F(|\alpha|-m+1+k)=-\bh_\F(|\alpha|+|\beta|-m+1-k)$ with $k=x,y,x+y$. Then whether or not $$\bh_\F(|\alpha|-m+1+i)=\bh_\F(|\alpha|-m+1+j)$$ can be determined, if $m-1\in\{x,y,|\beta|-x-y\}$, $\{i,j\}=\{x,y,|\beta|-x-y\}\backslash\{m-1\}\not\subseteq A_1$, $m-1\in A_2$ and $|\alpha|\in A_1$.
		\end{enumerate}
	\end{lemma}
	\begin{proof}
		We only prove (1)--(2) because (3)--(4) follow analogously. Consider the following cases: 
		\begin{enumerate}[label=(\alph*)]
			\item $m-1\not \in \{x,y,|\alpha|-x-y\}$;
			\item $m-1 \in \{x,y,|\alpha|-x-y\}$ and $m-1\not \in A_2$;
			\item $m- 1\in\{x,y,|\alpha|-x-y\}$, $m-1\in A_2$ and $|\beta|\in A_0$;
			\item $m- 1\in\{x,y,|\alpha|-x-y\}$, $m-1\in A_2$, $|\beta|\in A_1$ and $\{x,y,|\alpha|-x-y\}\backslash\{m-1\}\not\subseteq A_1$ %$\{x,y,|\alpha|-x-y\}\backslash\{m-1\}\cap (A_0\cup A_2)\ne \varnothing$.
		\end{enumerate}
		For the last case, say $\{i,j\}=\{x,y,|\alpha|-x-y\}\backslash\{m-1\}$, we will show that whether or not $\bh_{\F}(i)=\bh_{\F}(j)$ can be determined. 
		Let $$\lambda=\lambda(x,y,|\alpha|-x-y,|\beta|)$$ and set $d=[h_\lambda]s_\F$.
		
		For (a): Since {\emma none of the parts of} $\lambda$ equals $m-1$, we have $\{(\gamma;1)\succeq(\F):(\lambda(\gamma),m-1)=\lambda\}=\varnothing$. {\emma Hence, by} (\ref{eq:JT}), the multiset $\{(\CMcal{S};0)\succeq(\F):\lambda(\CMcal{S})=\lambda\}$ contains exactly $d$ elements, thus obtaining that 
		\begin{align*}
			{\emma \bh_\F(x)=\bh_\F(y)=\bh_\F(|\alpha|-x-y) \iff d=0}.
		\end{align*}
		
		For (b) and (c): If $m-1\in\{x,y,|\alpha|-x-y\}$, then $\bh_{\F}(m-1)=-\bh_{\F}(|\alpha|-m+1)=+$. That is, $m-1\in A_1\cup A_2$. If $m-1\in A_1$, then $m-1$ is determined. By Lemma \ref{lem:final}, $x,y,x+y$ are all determined, and we are done. 
		
		Otherwise, $m-1\in A_2$. Assume without loss of generality that $y=m-1$. Then $|\beta|\in A_0$, $\bh_\F(x)=-\bh_\F(|\alpha|-x)$ and $\bh_\F(x+m-1)=-\bh_\F(|\alpha|-x-m+1)$. Accordingly, the multiset $\{(\CMcal{S})\succeq(\F):\lambda(\CMcal{S})=\lambda\}\cup \{(\gamma;1)\succeq(\F):(\lambda(\gamma),m-1)=\lambda\}$ contains at most six elements, namely:
		\begin{itemize}
			\item $(x,m-1,|\alpha|-x-m+1\boxdot |\beta|)$,
			\item $(m-1,x,|\alpha|-x-m+1\boxdot |\beta|)$,
			\item $(m-1,|\alpha|-x-m+1,x\boxdot |\beta|)$,
			\item $(|\alpha|-x-m+1,m-1,x\boxdot |\beta|)$,
			\item $(x,|\beta|,|\alpha|-x-m+1;1)$,
			\item $(|\alpha|-x-m+1,|\beta|,x;1)$.
		\end{itemize}
		If $\{[x],[|\alpha|-x-m+1]\}\cap\mathcal{A}=\varnothing$, then $x$ and $x+m-1$ are determined, and we are done. {\emma Now suppose} $\{[x],[|\alpha|-x-m+1]\}\cap\mathcal{A}\neq\varnothing$. Without loss of generality, assume $[|\alpha|-x-m+1]\in\mathcal{A}$. 
		\begin{itemize}
			\item If $x\in A_0$, then $\bh_\F(x)\neq \bh_\F(y)$.
			%\item If $x\in A_0$ or $|\alpha|-x-m+1\in A_0$, then $\bh_\F(x)\neq \bh_\F(y)$ or $\bh_\F(|\alpha|-x-m+1)\neq \bh_\F(y)$, respectively.
			\item If $x\in A_2$, then if $x\ne m-1$, we have 
			\begin{align*}
				\bh_\F(|\alpha|-x-m+1)&=-\iff d=(-1)^{\ell(\F)},\\
				\bh_\F(|\alpha|-x-m+1)&=+\iff d=(-1)^{\ell(\F)-1}.
			\end{align*}
			%$\bh_\F(|\alpha|-x-m+1)=-$ (or $+$) if and only if $d=(-1)^{\ell(\F)}$ (or $(-1)^{\ell(\F)-1}$, respectively). 
			If $x=m-1$, we have 
			\begin{align*}
				\bh_\F(|\alpha|-x-m+1)&=-\iff d=0,\\
				\bh_\F(|\alpha|-x-m+1)&=+\iff |d|=1.
			\end{align*}
		    %$\bh_\F(|\alpha|-x-m+1)=-$ (or $+$) if and only if $d=0$ (or $|d|=1$).
			%\item the case $|\alpha|-x-m+1\in A_2$ and $[x]\in\mathcal{A}$ can be done similarly and we omit it.
			%\item If $|\alpha|-x-m+1\in A_2$ and $[x]\in\mathcal{A}$, then $\bh_\F(x)=-$ (or $+$) if and only if $d=(-1)^{\ell(\F)}$ (or $(-1)^{\ell(\F)-1}$, respectively).
			\item If $x\in A_1$, then $\bh_\F(x)=\bh_\F(|\alpha|-x-m+1)=+$ if and only if $d=0$.
			%(1) if $\bh_\F(x)=\bh_\F(|\alpha|-x-m-1)=+$, then $d=0$; (2) if $\bh_\F(x)=\bh_\F(|\alpha|-x-m-1)=-$, then $|d|=2$; (3) if $\bh_\F(x)=+$ and $\bh_\F(|\alpha|-x-m-1)=-$, then $|d|=1$; (4) if $\bh_\F(x)=-$ and $\bh_\F(|\alpha|-x-m-1)=+$, then $|d|=1$.
		\end{itemize}
		
		In all {\emma subcases}, we can determine whether $\bh_\F(x)=\bh_\F(y)=\bh_\F(|\alpha|-x-y)$ holds, and we are done. The cases $x=m-1$ and $x+y=m-1$ follow {\emma by} the same argument, so we omit the details.
		
		For (d): Suppose $y=m-1\in A_2$, with $|\beta|\in A_1$ and $\{x,|\alpha|-x-m+1\}\cap(A_0\cup A_2)\neq\varnothing$. Accordingly, the multiset $\{(\CMcal{S})\succeq(\F):\lambda(\CMcal{S})=\lambda\}\cup \{(\gamma;1)\succeq(\F):(\lambda(\gamma),m-1)=\lambda\}$ contains at most eight elements: the six {\emma listed} above, together with the following two {\emma elements}:
		\begin{itemize}
			%\item $(x,m-1,|\alpha|-x-m+1\boxdot |\beta|)$,
			%\item $(m-1,x,|\alpha|-x-m+1\boxdot |\beta|)$,
			%\item $(m-1,|\alpha|-x-m+1,x\boxdot |\beta|)$,
			%\item $(|\alpha|-x-m+1,m-1,x\boxdot |\beta|)$,
			%\item $(x,|\beta|,|\alpha|-x-m+1;1)$,
			%\item $(|\alpha|-x-m+1,|\beta|,x;1)$,
			\item $(|\beta|,|\alpha|-x-m+1,x;1)$,
			\item $(|\beta|,x,|\alpha|-x-m+1;1)$.
		\end{itemize}
		If $\{[x],[|\alpha|-x-m+1]\}\cap\mathcal{A}=\varnothing$, then $x,x+m-1$ are determined, and we are done. {\emma Now suppose} $\{[x],[|\alpha|-x-m+1]\}\cap\mathcal{A}\neq\varnothing$. Without loss of generality, assume that $[x+m-1]\in\mathcal{A}$ and that $m-1\ne x\in A_0\cup A_2$. Then {\emma whether or not $\bh_\F(x)=\bh_\F(m-1)=\bh_\F(|\alpha|-x-m+1)$ is determined by the following table:}
		%\begin{itemize}
		%	\item If $x\in A_0$, then $\bh_\F(x)=\bh_\F(|\alpha|-x-m+1)$ (or $\ne$)  
		%	if $|d|=1$ (or $|d|=2$);
		%	\item If $x\in A_2$, then $\bh_\F(x)=\bh_\F(|\alpha|-x-m+1)$ (or $\ne$)  
		%	if $|d|=2$ (or $|d|=1$).
		%\end{itemize}

        	{\emma 
        	\begin{center}
        		\begin{tabular}{c|cc}
        		Subcase in (d) & $\bh_\F(x)/\bh_\F(|\alpha|-x-m+1)$ & $|d|$ \\
        			\hline \\[-1.5ex]
        			$x\in A_0$ & $1$ & $1$ \\
        			& $-1$ & $2$ \\
        			\hline \\[-1.5ex]
        			$x\in A_2$ & $1$ & $2$ \\
        			& $-1$ & $1$ 
        		\end{tabular}
        	\end{center}
        	\vspace{2mm}}

		The cases $x=m-1$ and $x+y=m-1$ follow by the same idea, so we omit {\emma the} details.
	\end{proof}

	\begin{lemma}\label{lem:GeneralBTvW2}
		Suppose that $\F$ is canonical. Then the following statements are true:
		\begin{enumerate}
			\item Suppose $1\leq x,y,x+y\le |\alpha|-1$ and exactly two of $x,y,x+y$, say $i,j$ satisfy $\bh_\F(i)=-\bh_\F(|\alpha|-i)$ and $\bh_\F(j)=-\bh_\F(|\alpha|-j)$, then whether or not $$\bh_\F(i)=\bh_\F(j)$$ 
			can be determined, unless $m-1\in\{x,y,|\alpha|-x-y\}$ and either $|\beta|\in A_1$ or $\{i,j\}\subseteq A_1$.
			\item Suppose $1\leq x,y,x+y\le |\beta|-1$ and exactly two of $x,y,x+y$, say $i,j$ satisfy $\bh_\F(|\alpha|-m+1+i)=-\bh_\F(|\alpha|+|\beta|-m+1-i)$ and $\bh_\F(|\alpha|-m+1+j)=-\bh_\F(|\alpha|+|\beta|-m+1-j)$, then whether or not $$\bh_\F(|\alpha|-m+1+i)=\bh_\F(|\alpha|-m+1+j)$$ 
			can be determined, unless $m-1\in\{x,y,|\alpha|-x-y\}$ and either $|\alpha|\in A_1$ or $\{i,j\}\subseteq A_1$.
		\end{enumerate}
	\end{lemma}
	\begin{proof}
		To prove (1), we consider the following cases:
		\begin{enumerate}[label=(\alph*)]
			\item $m-1\not \in \{x,y,|\alpha|-x-y\}$;
			\item $m- 1\in\{x,y,|\alpha|-x-y\}$, $|\beta|\in A_0$ and $\{i,j\}\cap (A_0\cup A_2)\ne \varnothing$.
		\end{enumerate}
		Let $\lambda=\lambda(x,y,|\alpha|-x-y,|\beta|)$ and $d=|[h_\lambda]s_\F|$. For any canonical $\E\sim \F$, we have $d=|[h_\lambda]s_\E|$. We {\emma now analyze the possible} values of $d$.
		
		For (a): if $m-1\notin\{x,y,|\alpha|-x-y\}$, then $\{(\gamma;1)\succeq(\F):(\lambda(\gamma),m-1)=\lambda\}=\varnothing$. Hence, by (\ref{eq:JT}), the multiset $\{(\CMcal{S};0)\succeq(\F):\lambda(\CMcal{S})=\lambda\}$ contains exactly $d$ elements. Then,
	
		\begin{enumerate}
			\item if $\bh_\F(x+y)=\bh_\F(|\alpha|-x-y)$, then {\emma whether or not $\bh_{\F}(x)=\bh_{\F}(y)$ is determined by $d$, namely,}
            	{\emma 
            	\begin{center}
            		\begin{tabular}{c|cc}
            	Cases $(a)$--$(1)$ & $\bh_{\F}(x)/\bh_{\F}(y)$	& $d$ \\
            			\hline \\[-1.5ex]
            	$\bh_\F(x+y)=\bh_\F(|\alpha|-x-y)=-$	&	$1$ & $0$ \\
            		&	$-1$ & $1$\\
            		\hline \\[-1.5ex]
            	$\bh_\F(x+y)=\bh_\F(|\alpha|-x-y)=+$	&	$1$ & $2$ \\
             	& $-1$ & $3$  \\
            		\end{tabular}
            	\end{center}
            	}	
			
			%$\bh_\F(x)=\bh_\F(y)$ (or $-\bh_\F(y)$) if and only if $d=0$ (or $1$, respectively);

			%\item if $\bh_\F(x+y)=\bh_\F(|\alpha|-x-y)=+$, then {\emma whether or not $\bh_{\F}(x)=\bh_{\F}(y)$ or not is determined by $d$, namely,}
			
			%$\bh_\F(x)=\bh_\F(y)$ (or $-\bh_\F(y)$) if and only if $d=2$ (or $3$, respectively);
			\item if $\bh_\F(y)=\bh_\F(|\alpha|-y)$, then {\emma whether or not $\bh_{\F}(x)=\bh_{\F}(x+y)$ is determined by $d$, namely,}
				{\emma 
				\begin{center}
					\begin{tabular}{c|cc}
						Cases $(a)$--$(2)$ & $\bh_{\F}(x)/\bh_{\F}(x+y)$	& $d$ \\
						\hline \\[-1.5ex]
						$\bh_\F(y)=\bh_\F(|\alpha|-y)=-$	&	$1$ & $1$ \\
						&	$-1$ & $0$\\
						\hline \\[-1.5ex]
						$\bh_\F(y)=\bh_\F(|\alpha|-y)=+$	&	$1$ & $3$ \\
						& $-1$ & $2$  \\
					\end{tabular}
				\end{center}
				}

			%$\bh_\F(x)=\bh_\F(x+y)$ (or $-\bh_\F(x+y)$) if and only if $d=1$ (or $0$, respectively);

			%$\bh_\F(x)=\bh_\F(x+y)$ (or $-\bh_\F(x+y)$) if and only if $d=3$ (or $2$, respectively).
		\end{enumerate} 
		The case $\bh_\F(x)=\bh_\F(|\alpha|-x)$ can be handled similarly to (1)--(2) and we omit the details.
		
		For (b): Suppose $y=m-1$. If $m-1\in A_1$, then $m-1$ {\emma is} determined. If not, then $[m-1]\in \A$,  and by Definition \ref{Def:7} (together with $|\beta|\in A_0$), {\emma we obtain} $\bh_{\F}(m-1)=\bh_{\F}(|\beta|)=-$, contradicting that $m-1$ is undetermined. It follows that $m-1\in A_1$ is determined, {\emma and} Lemma \ref{lem:final} {\emma completes the argument}. 
		
		If $y=m-1\in A_0$ {\emma and} $|\beta|\in A_0$, then we have $\bh_\F(x)=-\bh_\F(|\alpha|-x)$ and $\bh_\F(x+m-1)=-\bh_\F(|\alpha|-x-m+1)$.  {\emma Accordingly}, the multiset $\{(\CMcal{S})\succeq(\F):\lambda(\CMcal{S})=\lambda\}\cup \{(\gamma;1)\succeq(\F):(\lambda(\gamma),m-1)=\lambda\}$ contains at most four elements:
		\begin{itemize}
			\item $(x,m-1,|\alpha|-x-m+1\boxdot |\beta|)$,
			\item $(|\alpha|-x-m+1,m-1,x\boxdot |\beta|)$,
			\item $(x,|\beta|,|\alpha|-x-m+1;1)$,
			\item $(|\alpha|-x-m+1,|\beta|,x;1)$.
		\end{itemize}
		If $\{[x],[x+m-1]\}\cap\mathcal{A}=\varnothing$, then $x,x+m-1$ are determined, and we are done. {\emma Now} suppose $\{[x],[x+m-1]\}\cap\mathcal{A}\neq\varnothing$. Since $\{x,x+m-1\}\cap (A_0\cup A_2)\ne \varnothing$, {\emma the following case discussion shows that both $x$ and $x+m-1$ are determined.}
		
			{\emma 
			\begin{center}
				\begin{tabular}{c|cc}
					Subcase in $(b)$ & $\bh_\F$	& $d$ \\
					\hline \\[-1.5ex]
					$x\in A_0\cup A_2$ and $[x+m-1]\in\mathcal{A}$	&	$\bh_\F(x+m-1)=+$ & $0$ \\
					&	$\bh_\F(x+m-1)=-$ & $1$\\
					\hline \\[-1.5ex]
					$x+m-1\in A_0\cup A_2$ and $[x]\in\mathcal{A}$	&	$\bh_\F(x)=+$ & $1$ \\
					& $\bh_\F(x)=-$ & $0$  \\
				\end{tabular}
			\end{center}
			}	
		
		%\begin{enumerate}
		%	\item if $x\in A_0\cup A_2$ and $[x+m-1]\in\mathcal{A}$, then $\bh_\F(x+m-1)=+$ (or $-$) if and only if $d=0$ (or $1$, respectively);
		%	\item if $x+m-1\in A_0\cup A_2$ and $[x]\in\mathcal{A}$, then $\bh_\F(x)=+$ (or $-$) if and only if $d=1$ (or $0$, respectively).
		%\end{enumerate}
		%Therefore, both $x$ and $x+m-1$ are determined and we are done. 
		
		If $m-1\in A_2$ and $|\beta|\in A_0$, then we have $\bh_\F(m-1)=-\bh_\F(|\alpha|-m+1)$, so exactly one of $x,x+m-1$ satisfies 
		$\bh_\F(i)=-\bh_\F(|\alpha|-i)$. Suppose $i=x$. Then: 
		\begin{itemize}
			\item If $\{[x],[x+m-1]\}\cap \mathcal{A}=\varnothing$, then $x,x+m-1$ are determined, and we are done.
			\item Otherwise, $[x]\in\mathcal{A}$ and $\bh_\F(x+m-1)=\bh_\F(|\alpha|-x-m+1)$. In particular, $|\alpha|-x-m+1$ is determined. If $x$ is determined, then whether $\bh_{\F}(x)=\bh_{\F}(m-1)=+$ {\emma holds} is determined. Otherwise, $x$ is undetermined. {\emma Applying Lemma \ref{lem:BTvW2}} to the triple $(m-1,|\alpha|-x-m+1,|\alpha|-x)$, we obtain that $|\alpha|-x-m+1\in A_0\cup A_2$. By considering the multiset $\{(\CMcal{S})\succeq(\F):\lambda(\CMcal{S})=\lambda\}\cup \{(\gamma;1)\succeq(\F):(\lambda(\gamma),m-1)=\lambda\}$ {\emma as} before, $x$ is determined by the following table:
		\end{itemize}
		
		{\emma 
			\begin{center}
				\begin{tabular}{c|cc}
					Subcase in $(b)$ & $\bh_\F$	& $d$ \\
					\hline \\[-1.5ex]
					$\bh_{\F}(x+m-1)=+$ &	$+$ & $1$ \\
					&	$-$ & $2$\\
					\hline \\[-1.5ex]
					$\bh_{\F}(x+m-1)=-$	&	$+$ & $0$ \\
					& $-$ & $1$  \\
				\end{tabular}
			\end{center}
			}	
		
		%The multiset $\{(\CMcal{S})\succeq(\F):\lambda(\CMcal{S})=\lambda\}\cup \{(\gamma;1)\succeq(\F):(\lambda(\gamma),m-1)=\lambda\}$ contains at most six elements, namely,
		%\begin{itemize}
		%	\item $(m-1,x,|\alpha|-x-m+1\boxdot |\beta|)$,
		%	\item $(m-1,|\alpha|-x-m+1,x\boxdot |\beta|)$,
		%	\item $(x,m-1,|\alpha|-x-m+1\boxdot |\beta|)$,
		%	\item $(|\alpha|-x-m+1,m-1,x\boxdot |\beta|)$,
		%	\item $(x,|\beta|,|\alpha|-x-m+1;1)$,
		%	\item $(|\alpha|-x-m+1,|\beta|,x;1)$.
		%\end{itemize}
		%In consequence, we derive that
		%\begin{enumerate}
			%\item if $x+m-1\in A_0$, then $\bh_\F(x)=+$ (or $-$) if and only if $d=0$ (or $1$, respectively);
			%\item if $\bh_{\F}(x+m-1)=+$, then $\bh_\F(x)=+$ (or $-$) if and only if $d=1$ (or $2$, respectively); 
			%\item if $\bh_{\F}(x+m-1)=-$, then $\bh_\F(x)=+$ (or $-$) if and only if $d=0$ (or $1$, respectively).
			%\item if $x+m-1\in A_2$, then $\bh_\F(x)=+$ (or $-$) if and only if $d=1$ (or $2$, respectively).
		%\end{enumerate}
		The case $i=x+m-1$ {\emma is handled} similarly. The cases $x=m-1$ {\emma and} $x+y=m-1$ {\emma follow by the same argument}, so we omit the details.
		
		The proof of (2) of Lemma \ref{lem:GeneralBTvW2} {\emma proceeds along} the same line as {\emma that} of (1), with $|\alpha|$ {\emma replaced} by $|\beta|$. We therefore omit the details.
	\end{proof}

	Theorem \ref{thm:a<b} (2) is divided into Propositions \ref{lem:B1} and \ref{lem:B2}.
	
	\begin{proposition}\label{lem:B1}
		Suppose $\F$ is canonical. If $\mathcal{B}\neq\varnothing$, then all elements are determined.
	\end{proposition}
	
	\begin{proof}
		Suppose otherwise, {\emma and} let $x$ be the smallest undetermined element. In particular, we have $[x]\in\mathcal{A}$ and $m-1\leq x\leq (r+m-1)/2$ by Lemma \ref{lem:sym} and Lemma \ref{lem:symmetry}, respectively. 
		
		{\emma To show that all elements must be determined, we derive a contradiction in each of the following cases:}
		\begin{enumerate}
			\item $x=m-1$;
			\item $\mathcal{B}=\{[m-1]\}$ (which implies $x>m-1$)
			\item $x>m-1$ and there exists $[k] \in\mathcal{B}$ with $[k]\ne [m-1]$.
		\end{enumerate}
	    {\emma The contradiction is obtained by examining the coefficients $d=[h_\lambda]s_\F$ for various partitions $\lambda$. For a given partition $\lambda$, define}
	    \begin{align*}
	    	{\emma G_{\lambda}=\{(\CMcal{S})\succeq(\F):\lambda(\CMcal{S})=\lambda\}\cup \{(\gamma;1)\succeq(\F):(\lambda(\gamma),m-1)=\lambda\}}.
	    \end{align*}
		Case (1): $x=m-1$. Since $\mathcal{B}\neq\varnothing$, there exists $[k]\in\mathcal{B}$ such that one of the following holds:
		\begin{enumerate}[label*=(\alph*)]
			\item $1\leq k\leq|\alpha|-1$, $k\in A_1$ and $|\alpha|-k\in A_0\cup A_2$.
			\item $1\leq k\leq|\alpha|-1$, $k\in A_1$, $|\alpha|-k\in A_1$ and $\bh_\F(k)=\bh_\F(|\alpha|-k)$.
			\item $|\alpha|\leq k\leq|\alpha|+|\beta|-m$, $k\in A_1$ and $2|\alpha|+|\beta|-2m+2-k\in A_0\cup A_2$.
			\item $|\alpha|\leq k\leq|\alpha|+|\beta|-m$, $k\in A_1$, $2|\alpha|+|\beta|-2m+2-k\in A_1$ and $\bh_\F(k)=\bh_\F(2|\alpha|+|\beta|-2m+2-k)$.
		\end{enumerate}
		If we are in cases (a) or (b), let $\lambda=\lambda(k,|\alpha|-k,|\beta|-m+1,m-1)$ and $d=|[h_\lambda]s_\F|$. The set $G_{\lambda}$ contains at most six elements, namely: 
		\begin{itemize}
			\item $(k,|\alpha|-k\boxdot |\beta|-m+1,m-1)$,
			\item $(|\alpha|-k,k\boxdot|\beta|-m+1,m-1)$,
			\item $(|\beta|-m+1,k,|\alpha|-k;1)$,
			\item $(|\beta|-m+1,|\alpha|-k,k;1)$,
			\item $(k,|\beta|-m+1,|\alpha|-k;1)$,
			\item $(|\alpha|-k,|\beta|-m+1,k;1)$.
		\end{itemize}
		Since $[m-1]\in\mathcal{A}$ and $m-1\in A_1$, we have $\bh_\F(m-1)=-\bh_{\F}(|\beta|-m+1)$. We {\emma now} list all possibilities {\emma for} case (a) and (b).
		
		\begin{center}
			\begin{tabular}{c|cc}
				Cases (a)--(b)	& $\bh_\F(m-1)$ & $d$ \\
				\hline\\[-1.5ex]
				$k\in A_1$, $|\alpha|-k\in A_0$ & $+$ & $0$ \\
				& $-$ & $1$ \\
				\hline\\[-1.5ex]
				$k\in A_1$, $|\alpha|-k\in A_2$ & $+$ & $1$ \\
				& $-$ & $0$ or $2$ \\
				\hline\\[-1.5ex]
				$k,|\alpha|-k\in A_1$, $\bh_\F(k)=\bh_\F(|\alpha|-k)$ & $+$ & $0$ \\
				& $-$ & $2$ \\
			\end{tabular}
		\end{center}
		
		As is shown in the table above, $\bh_\F(m-1)$ {\emma is} determined from $d$, contradicting {\emma the assumption} that $x=m-1$ is undetermined. Cases (c) and (d) {\emma are handled} similarly, and we omit the details.
		
		Case (2): $\mathcal{B}=\{[m-1]\}$, which implies $x\ge m$. Note that $m-1\in A_0\cup A_2$. Indeed, if $m-1\in A_1$ {\emma were} determined, then Lemma \ref{lem:final} {\emma would imply} that all elements are determined, {\emma contradicting} the existence of {\emma an} undetermined $x$. If, {\emma on the other hand}, $m-1\in A_1$ {\emma were} undetermined, then $[m-1]\in\A$, {\emma contradicting the assumption} that $[m-1]\in\mathcal{B}$. It follows that 
		\begin{align}\label{E:meven}
		m-1\in A_0\cup A_2.  
		\end{align}
		%The proof is similar to Proposition \ref{lem:main} and we only point out the differences.
		
		%Since $\mathcal{B}=\{[m-1]\}$, the function $g$ is $r$-periodic except multiples of $r_0$ and the elements from $[m-1]$, that is, $g(u)=g(u+r)$ for $1\le u\le |\alpha|-m-r$ such that $r_0\nmid u$, $r_0\nmid (u+r)$ and $[u]\ne [m-1]$. 
		
		%Let $s=\gcd(r_0,r)=\gcd(r_0,m-1)$, it follows that $g(u)$ is also $s$-periodic except multiples of $r_0$ and $[u]=[m-1]$. 
		Consider the subcases below:
		\begin{itemize}
			\item $[m-1]\cap A_1\not\subseteq\{|\alpha|-m+1,|\alpha|,|\beta|-m+1,|\beta|\}$;
			\item $[m-1]\cap A_1\subseteq\{|\alpha|-m+1,|\alpha|,|\beta|-m+1,|\beta|\}$.
		\end{itemize}
		For the first subcase, there exists $k\in [m-1]\cap A_1$ such that 
		$$k\equiv m-1\mod r\quad \mbox{ and }\quad k\notin\{|\alpha|,|\beta|\}.$$ 
		Since $x\ge m$, we have $x-m+1\ge 1$.
		From the triple $(x-m+1,k,x+k-m+1)$, Lemmas \ref{lem:BTvW1} and \ref{lem:BTvW2} imply that $x-m+1\in A_1$ is determined, because $x+k-m+1\in [x]$ is undetermined and $k\in A_1$ is determined. On the other hand, {\emma applying the same reasoning to} the triple $(m-1,x-m+1,x)$, we find {\emma that} $m-1\in A_1$, {\emma contradicting (\ref{E:meven}).}
		
		For the second subcase, $m-1\in A_0$ by Definition \ref{Def:7}. We begin by showing that $x\ne 2m-2$, i.e., $x-m+1\ne m-1$ {\emma regardless of whether} $|\alpha|\in A_1$ or $|\beta|\in A_1$. Let $$\lambda=\lambda(|\alpha|,|\beta|-2m+2,m-1,m-1)$$ and set $d_1=|[h_\lambda]s_\F|$. 
		%If $m-1\in A_0$ 
		From Lemma \ref{lem:sym}, we have $m-1<|\alpha|-m+1$, i.e., $2m-2<|\alpha|$. 
		If $|\alpha|\in A_1$, then
		the set $G_{\lambda}$ contains at most one {\emma element}, namely, $(|\beta|-2m+2,m-1,|\alpha|;1)$. {\emma Moreover}, 
		\begin{align*}
			\bh_\F(|\beta|-2m+2)&=+ \iff d_1=1,\\
			\bh_\F(|\beta|-2m+2)&=- \iff d_1=0.
		\end{align*}
		Otherwise, if $|\alpha|\not \in A_1$ and  $|\beta|\in A_1$, let $$\mu=\lambda(|\alpha|-2m+2,m-1,|\beta|,m-1)$$ and set $d_2=|[h_\mu]s_\F|$. The set $G_{\mu}$ contains at most one {\emma element}, namely, $(|\beta|, m-1,|\alpha|-2m+2;1)$. Hence, 
		\begin{align*}
			\bh_\F(|\beta|+m-1)&=+ \iff d_2=1,\\
			\bh_\F(|\beta|+m-1)&=- \iff d_2=0.
		\end{align*}
		
		%For $m-1\in A_2$, since $x-m+1<x$ is determined, from the triple $(m-1,x-m+1,x)$ and Lemma \ref{lem:BTvW2}, we have $x-m+1\in A_0\cup A_2$. If $\bh_{\F}(x-m+1)=\bh_{\F}(|\alpha|-x+m-1)$, then $x-m+1\ne m-1$ because
		%$\bh_{\F}(m-1)=-\bh_{\F}(|\alpha|-m+1)$. Otherwise $\bh_{\F}(x-m+1)=-\bh_{\F}(|\alpha|-x+m-1)$, then applying Lemma \ref{lem:GeneralBTvW1} on the triple $(m-1,x-m+1,x)$ gives $\bh_{\F}(x-m+1)=-\bh_{\F}(m-1)=-$ so that $x$ can be undetermined. This also implies that $x-m+1\ne m-1$, that is, $x\ne 2m-2$.
		
		%if $|\alpha|\not\in A_1$ and $|\alpha|-m+1\in A_1$, then $\bh_\F(|\beta-2m+2)=+$ (or $-$) if and only if $d_1=1$ (or $0$, respectively).  If $|\alpha|\in A_1$ and $|\alpha|-m+1\not \in A_1$, then $\bh_\F(2m-2)=+$ (or $-$) if and only if $d_2=0$ (or $1$, respectively). If $|\alpha|\in A_1$ and $|\alpha|-m+1\in A_1$, consider the triple $(|\alpha|-m+1,|\beta|-|\alpha|-m+1, |\beta|-2m+2)$, it is easy to verify that $|\alpha|>2m-2$ because otherwise $m-1\not\in A_2$ and $r-m+1\ne m-1$ for that $r\mid (m-1)$ leads to no undetermined element. 
		
		In particular, $2m-2\in [|\beta|-2m+2]=[|\beta|+m-1]$ is determined, so $x\ne 2m-2$. {\emma This} allows us to apply Lemmas \ref{lem:GeneralBTvW1} and \ref{lem:GeneralBTvW2} {\emma to} triples containing $x-m+1$.
		We {\emma now} divide the case $[m-1]\cap A_1\subseteq\{|\alpha|-m+1,|\alpha|,|\beta|-m+1,|\beta|\}$ into {\emma the following} four subcases:
		\begin{itemize} 
			\item $|\beta|-m+1\in A_1$, i.e., $|\alpha|\in A_1$:
			\begin{enumerate}[label=(\roman*)]
				\item $\bh_{\F}(2|\alpha|-m+1)=-$, (ii) $\bh_{\F}(2|\alpha|-m+1)=+$. %and $m-1\in A_0\cup A_2$;
				%\item $\bh_{\F}(2|\alpha|-m+1)=+$. %and $m-1\in A_0$;
				%\item $\bh_{\F}(2|\alpha|-m+1)=+$ and $m-1\in A_2$;
			\end{enumerate}
			\item  $|\beta|-m+1\not\in A_1$ and $|\beta|\in A_1$:
			\begin{enumerate}[label=(\roman*)]
				\setcounter{enumi}{2}
				\item $\bh_{\F}(2|\alpha|-2m+2)=-$, (iv) $\bh_{\F}(2|\alpha|-2m+2)=+$. %and $m-1\in A_0\cup $;
				%\item $\bh_{\F}(2|\alpha|-2m+2)=+$. %and $m-1\in A_0\cup A_2$.
			\end{enumerate} 
		\end{itemize}    
		%\begin{enumerate}[label=(\roman*)]
		%\item $\bh_{\F}(|\beta|-m+1)=-\bh_{\F}(2|\alpha|-m+1)=+$ and $m-1\in A_0\cup A_2$;
		%\item $\bh_{\F}(|\beta|)=-\bh_{\F}(2|\alpha|-2m+2)=+$ and $m-1\in A_0\cup A_2$;
		%\item $\bh_{\F}(|\beta|-m+1)=\bh_{\F}(2|\alpha|-m+1)=+$ and $m-1\in A_0$;
		%\item $\bh_{\F}(|\beta|)=\bh_{\F}(2|\alpha|-2m+2)=+$ and $m-1\in A_0$.
		%\item $\bh_{\F}(|\beta|)=\bh_{\F}(2|\alpha|-2m+2)=+$ and $m-1\in A_2$;
		%\item $\bh_{\F}(|\beta|-m+1)=\bh_{\F}(2|\alpha|-m+1)=+$ and $m-1\in A_2$.
		%\end{enumerate} 
		Since $[m-1]\cap A_1\subseteq\{|\alpha|-m+1,|\alpha|,|\beta|-m+1,|\beta|\}$, it follows that $\{|\beta|-m+1,|\beta|\}\cap A_1\ne \varnothing$. {\emma Hence at least} one of the four preceding subcases must {\emma hold}. Each subcase will be treated separately.
		
		Subcase (i): Consider the triple $$(x-m+1,|\beta|-|\alpha|+m-1-x,|\beta|-|\alpha|).$$
		{\emma It is straightforward} to verify the {\emma hypotheses} of Lemmas \ref{lem:GeneralBTvW1} (2) and \ref{lem:GeneralBTvW2} (2), i.e., $x-m+1\ne m-1$, and $|\beta|-|\alpha|+m-1-x\ne m-1$ {\emma since the latter} belongs to $[x]$ and $|\beta|-|\alpha|\ge r>m$. Furthermore, {\emma the antisymmetry relation}
		\begin{align}\label{E:anti}
		\bh_{\F}(|\alpha|-m+1+i)=-\bh_{\F}(|\alpha|+|\beta|-m+1-i)
		\end{align}
		holds for $i=|\beta|-|\alpha|+m-1-x$ and $i=|\beta|-|\alpha|$.
		
		If {\emma the relation (\ref{E:anti}) does not hold} for $i=x-m+1$, then Lemma \ref{lem:GeneralBTvW2} shows that $\bh_{\F}(|\beta|-x)=\bh_{\F}(|\beta|-m+1)$ is determined. However, all elements in $[x]$ are undetermined, {\emma including} $|\beta|-x$. This implies that (\ref{E:anti}) must also hold for $i=x-m+1$. Consequently, Lemma \ref{lem:GeneralBTvW1}, {\emma together} with the fact {\emma that} $x-m+1$ is determined, yields $$\bh_{\F}(|\alpha|+x-2m+2)=-\bh_{\F}(|\alpha|+|\beta|-x)=\bh_\F(|\beta|-m+1)=+.$$ Furthermore, {\emma applying Lemma \ref{lem:GeneralBTvW2}} 
		to the triple $(x-m+1,m-1,x)$ and using the fact that $x-m+1<x$ is determined, we obtain $x-m+1\in A_0\cup A_2$. Since $\bh_{\F}(|\alpha|+|\beta|-x)=-$, it follows that $x-m+1\in A_0$. Let 
		\begin{align}\label{E:lam}
			\lambda=\lambda(|\alpha|,|\beta|-x,x-m+1,m-1)
		\end{align}
		and set $d=[h_\lambda]s_\F$. The set $G_{\lambda}$ contains at most four elements (since $x-m+1\in A_0$):
		\begin{itemize}
			\item $(|\alpha|\boxdot |\beta|-x,x-m+1,m-1)$,
			\item $(|\alpha|\boxdot x-m+1,|\beta|-x,m-1)$,
			\item $(|\alpha|\boxdot x-m+1,m-1,|\beta|-x)$,
			\item $(|\beta|-x,x-m+1,|\alpha|;1)$.
		\end{itemize}
		As a consequence, since $m-1\in A_0$, we find that
		\begin{align*}
			\bh_\F(|\beta|-x)&=+ \iff d=(-1)^{\ell(\F)-1},\\
			\bh_\F(|\beta|-x)&=- \iff d=(-1)^{\ell(\F)}.
		\end{align*}
		%If $m-1\in A_2$, then $\bh_\F(|\beta|-x)=+$ (or $-$) if and only if $|d|=1$ (or $|d|=2$, respectively).
		Thus $x$ is determined by $d$, a contradiction. 
		
		%Subcases (ii) is similar to (i) by choosing the triple $(x-m+1,|\beta|-|\alpha|+m-1,|\beta|-|\alpha|+x)$ and the same $\lambda$ as (i). We omit the details.
		
		Subcase (ii): Similar to subcase (i), applying Lemmas \ref{lem:GeneralBTvW1} and \ref{lem:GeneralBTvW2} to the triple $$(x-m+1,|\beta|-|\alpha|+m-1-x,|\beta|-|\alpha|),$$ we obtain  $\bh_{\F}(|\alpha|+x-2m+2)=\bh_{\F}(|\alpha|+|\beta|-x)$. In addition, the same argument as {\emma in} subcase (i) gives $x-m+1\in A_0\cup A_2$, {\emma which implies that} $$\bh_{\F}(x-m+1)=h(|\alpha|+|\beta|-x)=\bh_{\F}(|\alpha|+x-2m+2).$$
		
		If $\bh_{\F}(x-m+1)=-$, then $\bh_{\F}(|\alpha|+x-2m+2)=\bh_{\F}(|\alpha|+|\beta|-x)=-$. Since $m-1\in A_0$, the set $G_{\lambda}$ contains at most one element, namely, $(|\beta|-x,x-m+1,|\alpha|;1)$. {\emma Consequently}, $\bh_{\F}(|\beta|-x)$ is determined by $d$. That is,
		\begin{align*}
			\bh_{\F}(|\beta|-x)=+ \iff d=1.
		\end{align*}
       This contradicts the fact that $|\beta|-x\in [x]$ is undetermined. 
		
		Otherwise, $\bh_{\F}(x-m+1)=+$. By considering the set $G_{\lambda}$, we find that	
			\begin{align*}
			\bh_{\F}(x)&=+ \iff |d|=1,\\
			\bh_{\F}(x)&=- \iff |d|=2.
		\end{align*}
		Thus $x$ is determined by $d$, a contradiction.

		Subcase (iii) is {\emma analogous} to subcase (i), using the triple $$(x-m+1,|\beta|-|\alpha|+m-1,|\beta|-|\alpha|+x).$$ 
		Subcase (iv) {\emma follows similarly}, by applying the same reasoning to the triple $(x-m+1,m-1,x)$ and the partition $\lambda=\lambda(|\beta|,x-m+1,m-1,|\alpha|-x)$. We omit the details, {\emma as they are entirely analogous to the previous cases.}
		
		Case (3): $x>m-1$ and there exists $[k]\in\mathcal{B}$ with $[k]\neq[m-1]$. That is, there exists $k\notin[m-1]$ such that one of the following holds:
		\begin{enumerate}[label=(\roman*)]
			\item $k\in A_1$ and $k+r\in A_0\cup A_2$.
			\item $k\in A_0\cup A_2$ and $k+r\in A_1$.
			\item $k,k+r\in A_1$ and $\bh_\F(k)=-\bh_\F(k+r)$.
		\end{enumerate}
		Without loss of generality, we {\emma may assume} $$k\leq(|\alpha|+|\beta|-m+1-r)/2.$$ If $k$ {\emma exceeds this bound}, set $k'=|\alpha|+|\beta|-m+1-k$. Then $k'-r<(|\alpha|+|\beta|-m+1-r)/2$ and the pairs $k',k'-r$ has the same types as $k,k+r$, satisfying one of (i)--(iii).
		
		{\emma By} Lemma \ref{lem:sym}, {\emma the elements} $k$ and $k+r$ are determined, since $[k]\in \mathcal{B}$.
		
		Suppose {\emma first} $k<x$. Since $x$ and $x+r$ are undetermined, {\emma applying Lemmas \ref{lem:BTvW1} and \ref{lem:BTvW2}} to the triples $(x-k,k,x)$ and $(x-k,k+r,x+r)$, {\emma we find} that $x-k$ must be undetermined {\emma in each of the cases} (i)--(iii), contradicting the minimality of $x$.
		
		{\emma Now} suppose $1<x<k$. We claim that $$2r+m-1-x+k\notin\{|\alpha|-m+1,|\beta|-m+1\}.$$ {\emma If not}, then $k+m-1\equiv x\mod r$, say $k+m-1=x+tr$. {\emma Applying Lemmas \ref{lem:BTvW1} and \ref{lem:BTvW2}} to the triples $(m-1,k,x+tr)$ and $(m-1,k+r,x+(t+1)r)$, we obtain that $m-1$ is undetermined, contradicting $x>m-1$. Hence the claim holds.
		
		Since $r+m-1-x$ and $2r+m-1-x$ are undetermined, {\emma applying Lemmas \ref{lem:BTvW1} and \ref{lem:BTvW2}} to the triples $(r+m-1-x,k+r,2r+m-1-x+k)$ and $(2r+m-1-x,k,2r+m-1-x+k)$, {\emma we conclude} that $2r+m-1-x+k$ must be undetermined, hence $k-x+m-1$ is also undetermined. In particular, $$k-x+m-1\notin\{|\alpha|,|\beta|\}.$$ Finally, {\emma applying the same lemmas} to the triples $(x-m+1,k-x+m-1,k)$ and $(x-m+1,k-x+m-1+r,k+r)$, we deduce that $x-m+1$ is undetermined, again contradicting the minimality of $x$.
		
		{\emma In summary}, each of the cases (1)--(3) {\emma leads to} a contradiction. Therefore, every element must be determined.
	\end{proof}
	
	\begin{proposition}\label{lem:B2}
		Suppose $\F$ is canonical with $\mathcal{B}=\varnothing$. If $\mathcal{D}\neq\varnothing$, then all elements are determined.
	\end{proposition}
	
	\begin{proof}
		Suppose not, and let $x$ be the smallest undetermined element. Then, as in Proposition \ref{lem:B1}, 
		$$m-1\leq x\leq (r+m-1)/2.$$
		Since $\mathcal{D}\neq\varnothing$, let $[j]\in\mathcal{D}$. Then 
		\vspace{2mm}
		\begin{enumerate}[label=(\Roman*)]
			\item $[j]=[m-1]$ with $$\bh_\F(k)=-\bh_\F(|\alpha|-k)=+$$ for all $k\in[j]$ such that $m-1\leq k\leq |\alpha|-m+1$ and $k\equiv m-1 \mod r$, and $$\bh_\F(|\alpha|-m+1+k)=-\bh_\F(|\alpha|+|\beta|-m+1-k)=-$$ for all $k\in[j]$ such that $m-1\leq k\leq |\beta|-m+1$ and $k\equiv m-1 \mod r$. In particular, $|\alpha|\in A_1$ and $|\beta|\in A_0$.\\
			\item There exists $k_\alpha\in[j]$ with $1\leq k_\alpha\leq |\alpha|-1-r$ {\emma such that} $\bh_\F(k_\alpha)=-\bh_\F(|\alpha|-k_\alpha)$ and one of the following {\emma holds}:\vspace{2mm}
			\begin{enumerate}[label=(\alph*)]
				\item $\bh_\F(k_\alpha+r)=\bh_\F(|\alpha|-k_\alpha-r)$;
				\item $\bh_\F(k_\alpha)=-\bh_\F(|\alpha|-k_\alpha)=-\bh_\F(k_\alpha+r)=\bh_\F(|\alpha|-k_\alpha-r)$.\vspace{2mm}
			\end{enumerate}
		    {\emma Alternatively}, there exists $k_\beta\in[j]$ with $1\leq k_\beta\leq |\beta|-1-r$ such that $\bh_\F(|\alpha|-m+1+k_\beta)=-\bh_\F(|\alpha|+|\beta|-m+1-k_\beta)$ and one of the following holds:\vspace{2mm}
			\begin{enumerate}[label=(\alph*)]
				\item $\bh_\F(|\alpha|-m+1+k_\beta+r)=\bh_\F(|\alpha|+|\beta|-m+1-k_\beta-r)$;
				\item $\bh_\F(|\alpha|-m+1+k_\beta)=-\bh_\F(|\alpha|+|\beta|-m+1-k_\beta)=-\bh_\F(|\alpha|-m+1+k_\beta+r)=\bh_\F(|\alpha|+|\beta|-m+1-k_\beta-r)$.
			\end{enumerate}
		\end{enumerate}
		If the class $[j]$ {\emma does} not satisfy (II), then for all $1\leq k\leq |\alpha|+|\beta|-m+1-r$ with $k\in[j]$, the equality $\bh_\F(k)=\bh_\F(k+r)$ holds, except when $k<|\alpha|\le k+r$. 
		
		If $[j]\subseteq A_0\cup A_2$, then $[j]\in\mathcal{C}$, contradicting $[j]\in\mathcal{D}$. Otherwise, $[j]=[m-1]$ and satisfies condition (I); in particular $\bh_\F(|\alpha|-m+1)=\bh_\F(|\alpha|)=-$. %Consequently, we have to be in case (A).
		
		As in Proposition \ref{lem:B1}, we {\emma analyze} the following cases and {\emma derive} a contradiction {\emma in each}, with the {\emma goal} of showing that all elements must be determined.
		%\begin{enumerate}
		%	\item $x=m-1$;
		%	\item $\mathcal{D}=\{[m-1]\}$ (implying that $x>m-1$);
		%	\item $x>m-1$ and $[k_{\alpha}]\in\mathcal{D}$ exists and $[k_{\alpha}]\ne [m-1]$;
		%	\item $x>m-1$ and $[k_{\beta}]\in\mathcal{D}$ exists and $[k_{\beta}]\ne [m-1]$;
		%\end{enumerate}
		
		\begin{enumerate}
			\item $x=m-1$;
			\item $[m-1]\in\mathcal{D}$ ({\emma which implies} $x>m-1$);
			\item $x>m-1$, {\emma there exists} $[k_{\alpha}]\in\mathcal{D}$ with $[k_{\alpha}]\ne [m-1]$, and $[m-1]\in\mathcal{C}$;
			\item $x>m-1$, {\emma there exists} $[k_{\beta}]\in\mathcal{D}$ with $[k_{\beta}]\ne [m-1]$, and $[m-1]\in\mathcal{C}$.
		\end{enumerate}
		
		We emphasize that if $x>m-1$ and $[m-1]\in\mathcal{A}$, then $m-1\in A_1$ is determined, and Lemma \ref{lem:final} {\emma completes the argument}. Hence we need {\emma only} to consider the case $[m-1]\in\mathcal{C}\cup\mathcal{D}$. 
		
		For case (II), we focus on the subcase {\emma where} $k_\alpha$ exists; the subcase {\emma involving} $k_\beta$ is analogous {\emma and its details are omitted.}
		
		Case (1): $x=m-1$. Suppose $k_\alpha$ exists. Then $k_\alpha\notin[m-1]$, since $[m-1]\in\mathcal{A}$. For types (a) or (b) of $k_\alpha$, applying (1) of Lemmas \ref{lem:GeneralBTvW1} and \ref{lem:GeneralBTvW2} to the triple $$(k_\alpha,|\alpha|-k_\alpha-r,|\alpha|-r),$$
		we find that $|\alpha|-r\in [m-1]$ must be determined, contradicting {\emma the assumption} that $x=m-1$ is undetermined. %The case when $k_\beta$ exists can be done similarly by considering the triple $(k_\beta,|\beta|-k_\beta-r,|\beta|-r)$, so we omit the details.
		
		Case (2): $[m-1]\in\mathcal{D}$. %If $[m-1]\subseteq A_0\cup A_2$, then $|\alpha|\in A_0$ and $|\beta|\in A_0$. Otherwise $[m-1]\cap A_1\subseteq \{|\alpha|-m+1,|\alpha|,|\beta|-m+1,|\beta|\}$. That is, $|\alpha|\in A_1$ and $|\beta|\in A_1$ may occur.
		Suppose that we are in case (I). {\emma We prove that 
		\begin{align*}
			\mathrm{Case (2)+Case (I)} %&\Longrightarrow x-m+1+tr\in A_0\cup A_2 \\
			%&\Longrightarrow \bh_\F(x-m+1+tr)=-\bh_\F(|\alpha|-x+m-1-tr)=-\\
			%&\Longrightarrow \bh_\F(|\alpha|+x-2m+2+sr)=\bh_\F(|\alpha|+|\beta|-x-sr)\\
			&\Longrightarrow [x-m+1]\in\mathcal{B},
		\end{align*}
		contradicting that $\mathcal{B}=\varnothing$.} In this case, $m-1\in A_2$, %$\bh_\F(m-1)=-\bh_\F(|\alpha|-m+1)$, 
		$|\alpha|\in A_1$ and $|\beta|\in A_0$. By {\emma the} minimality of $x$, {\emma the element} $x-m+1$ is determined. Let $t\geq 0$ be any integer such that $x+tr<|\alpha|$. Then, $x-m+1+tr\in[x-m+1]$ is also determined. From the triple $(x-m+1+tr,m-1,x+tr)$, Lemma \ref{lem:BTvW2} implies that $$x-m+1+tr\in A_0\cup A_2.$$ Using the same triple,
		\begin{itemize}
			\item if $\bh_\F(x-m+1+tr)=\bh_\F(|\alpha|-x+m-1-tr)$, then Lemma \ref{lem:GeneralBTvW2} (1) {\emma implies} that $x+tr\in[x]$ is determined, a contradiction; \vspace{1mm}
			\item if $\bh_\F(x-m+1+tr)=-\bh_\F(|\alpha|-x+m-1-tr)=\bh_\F(m-1)=+$, then Lemma \ref{lem:GeneralBTvW1} (1) implies that $x+tr\in[x]$ is determined, again a contradiction.\vspace{1mm}
		\end{itemize}
		Hence we must have 
		\begin{align*}
			\bh_\F(x-m+1+tr)=-\bh_\F(|\alpha|-x+m-1-tr)=-.
		\end{align*}
		On the other hand, let $s\geq 0$ be any integer such that $x+sr<|\beta|$. Then $x-m+1+sr\in[x-m+1]$ is also determined. We claim that 
		\begin{align}\label{E:sr}
			\bh_\F(|\alpha|+x-2m+2+sr)=\bh_\F(|\alpha|+|\beta|-x-sr).
		\end{align}
		If $\bh_\F(|\alpha|+x-2m+2+sr)=-\bh_\F(|\alpha|+|\beta|-x-sr)$, then 
		{\emma applying Lemma \ref{lem:GeneralBTvW1} (4)} to the triple $(x-m+1+sr,m-1,x+sr)$ 
		 implies that $x+sr\in[x]$ is determined, a contradiction. Therefore, we prove the claim (\ref{E:sr}).
		 
		It follows that there exists some $j\in [x-m+1]$ such that $\bh_\F(j)=-\bh_\F(|\alpha|+|\beta|-m+1-j)$. That is, $[x-m+1]\in\mathcal{B}$, contradicting the assumption $\mathcal{B}=\varnothing$.
		
		Now suppose that we are in case (II). 
    
		Since $[m-1]\in\mathcal{D}$, let $[k_\alpha]=[m-1]$. Lemma \ref{lem:symmetry} guarantees that 
		{\emma exactly one of $k_{\alpha}$ and $|\alpha|-k_{\alpha}$ is congruent to $0\!\mod r$. Choose that element and denote it by $k$. Thus} 
		 $$k\equiv 0\mod r,$$ 
		 and in particular, $\bh_\F(k)=-\bh_\F(|\alpha|-k)$. Moreover, since $r\mid k\ne 0$, we have 
		 $$x<k<|\alpha|-m+1,$$  
		 in view of $x<r+m-1-x<r$ and $k_{\alpha}\ne m-1$. This also yields that $r<|\alpha|-m+1$. We now divide the undetermined elements $x'$ such that $x'<k$ into two cases:
		\begin{itemize}
			\item there exists {\emma such} an $x'$ with $(m-1)\nmid x'$; 
			\item {\emma every such} $x'$ satisfies $(m-1)\mid x'$. 
		\end{itemize}
		For the first case, {\emma we prove that 
			\begin{align*}
				\mathrm{Case (2)+Case (II)}&\Longrightarrow 
				\mbox{Claim: If } x' \mbox{ is undetermined, then } x'-m+1 \mbox{ is undetermined.}
			\end{align*}
		 Suppose that the claim is proved, then there {\emma would} exist an undetermined element in $[1,m-1]$. {\emma However},  all elements less than $m$ are determined by Lemma \ref{lem:sym} and $[m-1]\in\mathcal{D}$, a contradiction.}
		
		We now prove the claim. For $m-1\in A_0\cup A_2$, suppose that $x'<k$ is undetermined {\emma while} $x'-m+1$ is determined. From the triple $(m-1,x'-m+1,x')$, Lemma \ref{lem:BTvW2} guarantees that $x'-m+1\in A_0\cup A_2$. We also point out that $x'-m+1\ne m-1$, since $(m-1)\nmid x'$. 
		
		If $k=k_{\alpha}$, consider the triples $$(x'-m+1,k-x'+m-1,k)
		\quad \mbox{ and }\quad(x'-m+1,k-x'+m-1+r,k+r),$$ Lemmas \ref{lem:GeneralBTvW1} and \ref{lem:GeneralBTvW2} show that the relation $\bh_{\F}(i)=-\bh_{\F}(|\alpha|-i)$ {\emma holds for} $i=k,k+r$ and {\emma moreover} $$\bh_{\F}(x'-m+1)=\bh_{\F}(k)=\bh_{\F}(k+r).$$ {\emma This contradicts the defining property} of $k$. Hence $x'-m+1$ must be undetermined.
		
		If {\emma instead} $k=|\alpha|-k_{\alpha}$, then $k_{\alpha}\equiv m-1\mod r$. Since $k_{\alpha}+r<|\alpha|$, we have $k_{\alpha}\le |\alpha|-2r$, and consequently $k_{\alpha}+x'-m+1+r<|\alpha|-m+1$, i.e., $k-x'+m-1-r>0$.
		By the triples $$(x'-m+1,k-x'+m-1,k)\quad\mbox{ and }\quad(x'-m+1,k-x'+m-1-r,k-r),$$ we {\emma again} derive a contradiction. {\emma Therefore}, $x'-m+1$ must be undetermined, as desired.
		
		For the second case, {\emma we prove that 
			\begin{align*}
				\mathrm{Case (2)+Case (II)}&\Longrightarrow 
				\mbox{Claim: If } x' \mbox{ is undetermined, then } x'+m-1 \mbox{ is undetermined.}
		\end{align*}
		Suppose that the claim is proved, then $k$ must be undetermined, since $(m-1)\mid k$ and $k>x'$. This is indeed a contradiction.}
		
		We now prove the claim.
		Since $x<k$ and $r+m-1-x<k$ are {\emma both} undetermined, we have $(m-1)\mid x$ and $(m-1)\mid (r+m-1-x)$, hence $(m-1)\mid r$. {\emma It follows that} $(m-1)\mid k$ and $(m-1)\mid (|\alpha|-k)$. By Proposition \ref{lem:main}, $x$ must be the smallest integer for which $[x]\in\mathcal{A}$.
		
		Now suppose $x'$ is undetermined. Note that $k-x'\ne m-1$; otherwise, Lemma \ref{lem:GeneralBTvW1} {\emma applied to} the triple $(m-1,x',k)$ with $k<|\alpha|-m+1$ {\emma would imply} that $x'$ is determined, contradicting the assumption.
		
		If $k=k_\alpha$, applying Lemmas \ref{lem:GeneralBTvW1} and \ref{lem:GeneralBTvW2} to the triples $(k-x',x',k)$ and $(k-x',x'+r,k+r)$ with $k-x'\ne m-1$, {\emma we obtain} that $k-x'$ is undetermined. Therefore, $x'+m-1\in [k-x']$ is undetermined, as desired. %Conversely, suppose that $x'+m-1$ is undetermined, then from the triples $(k-x',x',k)$ and $(k-x'+r,x',k+r)$ where $x'\ne m-1$ and the condition that $k-x'\in [x'+m-1]$ is undetermined, $x'$ must be undetermined. 
		%It follows that $k$ is undetermined because $k>x'$ and $k$ is a multiple of $m-1$. Nevertheless $k\in A_0\cup A_2$ is determined, a contradiction. 
		
		If instead $k=|\alpha|-k_{\alpha}$, by the triples $(k-x'-r,x'+r,k)$ and $(k-x'-r,x',k-r)$, we {\emma similarly}  conclude that $k-x'$ is undetermined, which {\emma again} implies $x'+m-1\in [k-x']$ is undetermined, as desired.
		
		Case (3): The conditions $x>m-1$ and $[m-1]\in\mathcal{C}$ imply that either
		\begin{enumerate}[label=(\roman*)]
			\item $[m-1]\subseteq A_0$, or
			\item $m-1\in A_2$, $|\alpha|\in A_1$ and $|\beta|\in A_1$.
		\end{enumerate}  %then it is determined and Lemma \ref{lem:final} tells that $x$ is determined, a contradiction. In this case, we have $m-1\in A_0\cup A_2$. %If $m-1\in A_0$, since $\mathcal{B}=\varnothing$, one sees that $|\alpha|\in A_0$ and $|\beta|\in A_0$. 
		%We only prove the case when $k_\alpha$ exists, and omit the case when $k_\beta$ exists. Furthermore, 
		We assume that $k_{\alpha}<|\alpha|-k_{\alpha}$; the case $|\alpha|-k_\alpha<k_\alpha$ {\emma is analogous}. We prove that 
		\begin{align*}
			\mathrm{Case (3)+Case (II)}\Longrightarrow \mbox{Claim}: k-x+m-1 \mbox{ is undetermined.}
		\end{align*}
		If this claim is true, then $x-m+1\ne m-1$; otherwise, from the triple $$(x-m+1,k-x+m-1,k),$$ Lemmas \ref{lem:GeneralBTvW1} and \ref{lem:GeneralBTvW2} imply that $k-x+m-1$ is determined. Furthermore, from the triples  $(x-m+1,k-x+m-1,k)$ and $(x-m+1,k-x+r+m-1,k+r)$, Lemmas \ref{lem:GeneralBTvW1} and \ref{lem:GeneralBTvW2} indicate that  $x-m+1$ must be undetermined, contradicting the minimality of $x$.
		
		We {\emma now} prove the claim that $k-x+m-1$ is undetermined. 
		If at least one of $k_\alpha$ and $|\alpha|-k_\alpha$ is less than $x$, say $1\le k=k_{\alpha}<x$. For both (i) and (ii), %if $k=k_{\alpha}$; otherwise $k=|\alpha|-k_{\alpha}$, use the triples $(x-k,k,x)$ and $(x-k,k-r,x-r)$. 
		note that $x-k\ne m-1$ for $m-1\in A_2$, by Lemma \ref{lem:GeneralBTvW1} applied to the triple $(m-1,k,x)$. Consider the triples $(x-k,k,x)$ and $(x-k,k+r,x+r)$. It follows from Lemmas \ref{lem:GeneralBTvW1} and \ref{lem:GeneralBTvW2} that $x-k$ is undetermined, contradicting the minimality of $x$. 
		
		{\emma Now assume} both $k_{\alpha}$ and $|\alpha|-k_{\alpha}$ are larger than $x$, i.e., $$x< k_\alpha<|\alpha|-k_\alpha<|\alpha|-x.$$ 
		Let $k=k_{\alpha}$. First, $k-x\ne m-1$ for $m-1\in A_2$ by Lemma \ref{lem:GeneralBTvW1} and the triple $(m-1,x,k)$. Then, for both (i) and (ii), applying Lemmas \ref{lem:GeneralBTvW1} and \ref{lem:GeneralBTvW2} to the triples 
		$$(k-x,x,k)\quad \mbox{ and }\quad (k-x,x+r,k+r),$$
		we find that $k-x$ {\emma must} be undetermined.
		
		Suppose $k-x+m-1$ is determined. Then $k-x+m-1+r$ is also determined. For $m-1\in A_2$, to apply Lemmas \ref{lem:GeneralBTvW1} and \ref{lem:GeneralBTvW2}, we {\emma must} verify that 
		\begin{align}\label{E:ine}
		k-x+m-1+r\ne |\alpha|-m+1.
		\end{align} 
		{\emma If equality held}, Lemma \ref{lem:GeneralBTvW1} {\emma applied to} the triple $$(k,r+m-1-x,|\alpha|-m+1)$$ 
		{\emma would imply} that $r+m-1-x$ is determined, contradicting {\emma the fact} that $r+m-1-x\in [x]$ is undetermined. Hence (\ref{E:ine}) holds for $m-1\in A_2$.
		
		From the triple $$(k,r-x+m-1,k-x+m-1+r),$$ 
		where $k-x+m-1+r<k+r\le |\alpha|-1$, Lemmas \ref{lem:GeneralBTvW1} and \ref{lem:GeneralBTvW2} prove that 
		\begin{align*}
			\bh_\F(k-x+m-1+r)&=-\bh_\F(|\alpha|-k+x-m+1-r),\\
			\bh_{\F}(k-x+m-1+r)&=\bh_{\F}(k).
		\end{align*}
		{\emma We now analyze subcases (i) and (ii) separately to prove the claim that the element $k-x+m-1$ is determined.}
		
		Subcase (i): 
		By the triple $(m-1,k-x+r,k-x+m-1+r)$ and Lemma \ref{lem:GeneralBTvW2}, we find that $k-x+r$ is determined, contradicting {\emma the fact} that $k-x+r\in [k-x]$ is undetermined. This implies that  $k-x+m-1\in A_1$ is undetermined. %Consider the triples $(x-m+1,k-x+m-1,k)$ and $(m-1,x-m+1,x)$, applying Lemma \ref{lem:GeneralBTvW1} and \ref{lem:GeneralBTvW2} to the first one, we obtain $\bh_{\F}(x-m+1)=-\bh_{\F}(|\alpha|-x+m-1)$, while by the second one we have $\bh_{\F}(x-m+1)=\bh_{\F}(|\alpha|-x+m-1)$, a contradiction.
		
		Subcase (ii): Consider the triple $(k-x,r+m-1,k-x+m-1+r)$ with $\bh_\F(r+m-1)=\bh_\F(|\alpha|-r-m+1)=+$, since $[m-1]\in\mathcal{C}$. Lemma \ref{lem:GeneralBTvW2} then asserts that $k-x$ is determined, a contradiction. Thus $k-x+m-1$ is undetermined.

		Case (4): {\emma This case follows by} an argument {\emma analogous} to that for case (3), so we omit the details.
		
		{\emma In conclusion}, all elements must be determined provided that $\F$ is canonical,  $\mathcal{B}=\varnothing$ and $\mathcal{D}\ne\varnothing$.
	\end{proof}

    \section{Final remarks}\label{S:finalremark}
    
    In this paper, we have characterized the equivalence classes of connected skew diagrams containing exactly one \(2\times m\) or \(m\times 2\) block of boxes.  In contrast to the ribbon case, where equivalence classes can be arbitrarily large, the presence of a single such block forces any nontrivial
    factorization to have only two potentially reversible ribbon factors. Consequently, the possible sizes of an equivalence class in this setting are \(1\), \(2\), and \(4\).
    
    The situation changes when the number of \(2\times m\) blocks is allowed to grow. Since ribbon equivalence classes can be arbitrarily large, the same holds for these thickened ribbons as the
    number of blocks increases. From the viewpoint of the \(h\)-expansion, each additional block enlarges the family of coarsenings that carry factors of \(h_{m-1}\), and therefore enlarges the possible cancellation patterns in Proposition \ref{prop:JT}.  Our proof illustrates the difficulty of controlling these
    cancellations even in the first non-ribbon case, with only one \(2\times m\) or \(m\times 2\) block.
    
    Olson-Harris's Hopf-algebraic proof \cite{OH:24} of the sufficient direction of Conjecture \ref{Conj:1} provides a complementary perspective. The difficulty with necessity is precisely that their argument is forward-looking: it begins with a specified factorization and proves that the associated composition preserves an already known equivalence. A proof of necessity requires additional rigidity capable
    of reconstructing the factorization data from the symmetric function. In the one-block setting, the coefficient-level analysis of the \(h\)-expansion provides exactly this kind of rigidity.
    
    This suggests returning to the original motivation for studying
    \(m\)-regular thickened ribbons: reducing part of the necessity problem to
    the already classified ribbon case.  For example, suppose that
    \begin{equation}\label{eq:final-thick-factorization}
    	\F=\bigl(\cdots((\E_1\circ \E_2)\circ \E_3)\cdots\bigr)\circ_\W \E_q
    \end{equation}
    is an irreducible factorization, where each \(\E_i\) is a ribbon and
    \(\W=(m-1)\).  Define
    \begin{equation}\label{eq:final-ribbonization}
    	\operatorname{Ribbon}(\F)
    	=\bigl(\cdots((\E_1\circ\E_2)\circ \E_3)\cdots\bigr)\circ \E_q
    \end{equation}
    by replacing the final \(\circ_\W\) in
    \eqref{eq:final-thick-factorization} with ordinary ribbon composition. A natural question is whether it induces a bijection, or at least a cardinality-preserving correspondence, between \([F]\) and \([\operatorname{Ribbon}(F)]\).  Such a result would provide a factor-reconstruction mechanism of the kind that is missing from the forward Hopf-algebraic argument.

	\section{Acknowledgements}
	
	We would like to thank the referees for their insightful comments and constructive suggestions.
	
	The first author was supported by the Austrian Research Fund (FWF) Elise-Richter Project V 898-N and by the National Natural Science Foundation of China (NSFC) under Project No. 12201529, and is currently supported by the NSFC under Project No. 12571358.

	The second author was supported in part by the Provincial Nature Science Foundation of Shandong, Project No. ZR2024QA026 and the Fundamental Research Funds for the Central Universities.

\end{document}